\documentclass{amsart}
\usepackage{upref,amsxtra,amscd,amsopn,amsbsy,amstext,amssymb,amsmath,amsthm}
\usepackage{bm}
\usepackage{amsfonts}
\usepackage{mathrsfs}
\usepackage{braket}
\usepackage{color}
\usepackage{graphicx}

\makeatletter
\@addtoreset{equation}{section}

\makeatother

\newtheorem{theorem}{Theorem}[section]
\newtheorem{proposition}[theorem]{Proposition}

\newtheorem{lemma}[theorem]{Lemma}
\theoremstyle{definition}
\newtheorem{definition}[theorem]{Definition}

\allowdisplaybreaks

\begin{document}

\title[The $p$-elastic flow for planar closed curves]{The $p$-elastic flow for planar closed curves \\ with constant parametrization}

\author[S.~Okabe]{Shinya Okabe}
\address[S.~Okabe]{Mathematical Institute, Tohoku University, Aoba, Sendai 980-8578, Japan}
\email{shinya.okabe@tohoku.ac.jp}
\author[G.~Wheeler]{Glen Wheeler}
\address[G.~Wheeler]{Institute for Mathematics and its Applications, University of Wollongong, Northfields Avenue, Wollongong, NSW, 2522, Australia}
\email{glenw@uow.edu.au}

\subjclass[2010]{}

\date{\today}

\begin{abstract}
In this paper, we consider the $L^2$-gradient flow for the modified $p$-elastic energy defined on planar closed curves. 
We formulate a notion of weak solution for the flow and prove the existence of global-in-time weak solutions with $p \ge 2$ for initial curves in the energy space via minimizing movements. 
Moreover, we prove the existence of unique global-in-time solutions to the flow with $p=2$ and obtain their subconvergence to an elastica as $t \to \infty$. 
\end{abstract}

\maketitle

%%%%%%%%%%%%%%%%%%%%%%%%%%%%%%%%%%%%%
%%%%%%%%%%%%%%%%%%%%%%%%%%%%%%%%%%%%%

\section{Introduction} \label{section:1}
This paper is concerned with the modified $p$-elastic flow defined on planar closed curves. 
The modified $p$-elastic energy for planar closed curve $\gamma : \mathcal{S}^1 := \mathbb{R} / \mathbb{Z} \to \mathbb{R}^2$ is defined by 
\begin{equation*}
\mathcal{E}_p(\gamma) := E_p(\gamma) + \lambda \mathcal{L}(\gamma) 
\end{equation*}
with 
\begin{align*}
E_p(\gamma) := \dfrac{1}{p} \int_\gamma |\kappa(s)|^p \, ds, \quad 
\mathcal{L}(\gamma) := \int_\gamma \, ds,  
\end{align*}
where $p>1$ and $\lambda >0$, and $\kappa$ and $s$ respectively denote the curvature and the arc length parameter of $\gamma$.
We assume here that $\gamma$ is in the \emph{energy space} $W^{2,p}(\mathcal{S}^1; \mathbb{R}^2)$.

The functional $E_2$ is well-known as the bending energy or one-dimensional Willmore functional, and the variational problem on $E_2$ has attracted great interest. 
The $L^2$-gradient flow for $\mathcal{E}_2$ and $E_2$ are called the modified elastic flow and the elastic flow respectively.
Both flows have been extensively studied in the mathematical literature 
(see for instance \cite{DLLPS,DLP_2014,DLP_2017,DP_2014,DPS,DKS,K,LS_1984,LS_1985,L,LLS,MM,MPP,MP,NO_2014,NO_2017,Oe_2014,Oe_2011,O_2007,O_2008,P,S,Wen_1993,Wen_1995,Wheeler_2015}, and references therein). 
It is significant to extend studies on the bending energy to those on the $p$-elastic energy with $p \neq 2$.
Indeed, recently the $p$-elastic energy has attracted interest (e.g., \cite{AM,BVH,DFLM,FKN,NP,OPW,SW,W}). 
The purpose of this paper is to construct and study the $L^2$-gradient flow for the $p$-elastic energy for initial curves in the energy space.

Formally the Cauchy problem for the $L^2(ds)$-gradient flow for the $p$-elastic energy defined on closed curves is given by 
\begin{align}
\label{eq:P} \tag{P}
\begin{cases}
\partial_t \gamma = - \nabla \mathcal{E}_p(\gamma) \,\,\, &\text{in} \,\,\, \mathcal{S}^1 \times (0, T), \\
\gamma(x,0)= \gamma_0(x) \,\,\, &\text{in} \,\,\, \mathcal{S}^1.   
\end{cases}
\end{align}
Here $\nabla \mathcal{E}_p(\cdot)$ denotes the Euler-Lagrange operator of $\mathcal{E}_p$ in $L^2(ds)$, i.e., 
\begin{equation*}
\dfrac{d}{d\varepsilon} \mathcal{E}_p(\gamma+\varepsilon \varphi) \Bigm|_{\varepsilon=0}=\int_\gamma \nabla \mathcal{E}_p(\gamma) \cdot \varphi \, ds. 
\end{equation*}
The equation in \eqref{eq:P} can be classified as a fourth-order quasilinear parabolic equation.
For the case $p=2$, the analytic semigroup approach applies to prove the existence of local-in-time solutions of \eqref{eq:P} for \emph{sufficiently smooth} initial data.
We emphasize that the approach can not work for an initial curve in the energy space. 
For the case $p \neq 2$, one may observe that the equation in \eqref{eq:P} is strongly degenerate.
The coefficient of the highest order term is proportional to the $(p-2)$-th power of the curvature scalar.
This degeneracy has a profound impact on the solution space.
Indeed, Watanabe~\cite{W} gave several examples of critical points of $E_p$ with a \emph{flat core}, i.e., the critical point has an open region where the curvature is identically equal to $0$.  The derivative of curvature along a flat core solution may be discontinuous \cite[Ex 1]{W}, so they are \emph{not} of class $W^{4,1}(ds)$.

This means that critical points are not smooth, reminiscent of what occurs in the analysis of the $p$-Laplacian.
Regularity issues for the $p$-elastic flow are quite delicate, and it is significant to give a weak formulation for solutions to the problem \eqref{eq:P}.
Our goal in doing this is to prove the existence of solutions starting from initial data in the energy space.

We define weak solutions of the problem~\eqref{eq:P} with a specific parametrization: the so-called constant parametrization. 
Let $\gamma_0 \in W^{2,p}(\mathcal{S}^1; \mathbb{R}^2)$ be an initial curve and assume that 
\begin{equation}
\label{eq:1.1}
|\partial_x \gamma_0(x)| \equiv \mathcal{L}(\gamma_0). 
\end{equation}
For the parameter $x$ defined by \eqref{eq:1.1}, we set  
$$
\mathcal{AC}_{\gamma_0}:= \bigl\{ \gamma \in W^{2,p}(\mathcal{S}^1;\mathbb{R}^2) \mid |\partial_x \gamma(x)| \equiv \mathcal{L}(\gamma) \bigr\}. 
$$
We formulate the definition of weak solutions to the problem \eqref{eq:P} as follows: 
%%%%%%%%%%%%%%%%%%%%%%%%%%%%%%%%%%%%%
\begin{definition} \label{theorem:1.1}
We say that $\gamma$ is a weak solution to the problem \eqref{eq:P} if the following hold$\colon$ 
\begin{enumerate}
\item[{\rm (i)}] $\gamma \in L^{\infty}(0,T; W^{2,p}(\mathcal{S}^1;\mathbb{R}^2)) \cap H^1(0,T; L^2(\mathcal{S}^1;\mathbb{R}^2));$ 
\item[{\rm (ii)}] For any $\eta \in L^\infty(0,T; W^{2,p}(\mathcal{S}^1; \mathbb{R}^2))$, it holds that 
\begin{equation}
\label{eq:1.2}
\begin{aligned}
\int^T_0 \!\!\!\! \int^1_0 \Bigl[ & \dfrac{|\partial^2_x \gamma|^{p-2} \partial^2_x \gamma}{\mathcal{L}(\gamma)^{2p-1}} \cdot \partial^2_x \eta 
  - \dfrac{2p-1}{p} \dfrac{|\partial^2_x \gamma|^{p} \partial_x \gamma}{\mathcal{L}(\gamma)^{2p+1}} \cdot \partial_x \eta \\
& + \dfrac{\lambda}{\mathcal{L}(\gamma)} \partial_x \gamma \cdot \partial_x \eta 
  + \mathcal{L}(\gamma) \partial_t \gamma \cdot \eta 
  + \mathcal{L}(\gamma) \partial_t \gamma \cdot \Phi_1(\gamma,\eta) \partial_x \gamma \Bigr] \, dxdt=0,   
\end{aligned}
\end{equation}
where 
\begin{equation*}
\Phi_1(\gamma,\eta):= 
\dfrac{1}{\mathcal{L}(\gamma)^2} \Bigl(x \int^1_0 \gamma_x \cdot \eta_x \, d\tilde{x} - \int^x_0 \gamma_x \cdot \eta_x \, d\tilde{x}\Bigr); 
\end{equation*}
\item[{\rm (iii)}] $\gamma(\cdot, t) \in \mathcal{AC}_{\gamma_0}$ for a.e. $t \in (0,T);$ 
\item[{\rm (iv)}] For a.e. $t \in (0, T)$, 
\begin{equation}
\label{eq:1.3}
\mathcal{E}_p(\gamma(\cdot,t)) \le \mathcal{E}_p(\gamma_0(\cdot)); 
\end{equation}
\item[{\rm (v)}] The following energy inequality holds.  
\begin{equation}
\label{eq:1.4}
\mathcal{E}_p(\gamma(\cdot,T)) - \mathcal{E}_p(\gamma_0(\cdot)) \le - \dfrac{1}{2} \int^T_0 \!\!\! \int^1_0 \mathcal{L}(\gamma) |\partial_t \gamma|^2 \, dx dt;   
\end{equation}
\item[{\rm (vi)}] $\gamma(x,0)=\gamma_0(x)$ for a.e. $x \in \mathcal{S}^1$. 
\end{enumerate}
\end{definition}
%%%%%%%%%%%%%%%%%%%%%%%%%%%%%%%%%%%%%
We remark that the weak formulation given for the flow in (ii) above is adapted to the constant speed framework; it may be considered as the $L^2(dx)$-gradient flow of the $p$-elastic energy. 
We also note that the conditions (iv) and (v) do not imply the energy inequality 
\begin{equation}
\label{eq:1.5}
\mathcal{E}_p(\gamma(\cdot,t)) - \mathcal{E}_p(\gamma_0(\cdot)) \le - \dfrac{1}{2} \int^t_0 \!\!\! \int^1_0 \mathcal{L}(\gamma) |\partial_t \gamma|^2 \, dx dt
\end{equation}
for $t \in [0, T]$. 
Naturally, once a regular enough solution exists, we can reparametrize it (in space and time), so that there is a direct correspondence between the solutions we construct here and solutions to the classical (smooth) $p$-elastic flow.

On the existence of weak solutions to the problem~\eqref{eq:P} we have: 
%%%%%%%%%%%%%%%%%%%%%%%%%%%%%%%%%%%%%
\begin{theorem} \label{theorem:1.2}
Let $p \ge 2$ and assume that $\gamma_0 \in W^{2,p}(\mathcal{S}^1)$ satisfies \eqref{eq:1.1}.  
Then the problem \eqref{eq:P} possesses a weak global-in-time solution. 
\end{theorem}
%%%%%%%%%%%%%%%%%%%%%%%%%%%%%%%%%%%%%
The authors of this paper and Pozzi \cite{OPW} proved the existence of local-in-time weak solutions of the $L^2(ds)$-gradient flow for $\mathcal{E}_p$ with $p \ge 2$ under the inextensibility constraint.
The resultant flow is a second order parabolic equation. 
Moreover, it was proved in \cite{OPW} that the weak solution can be extended to a global-in-time solution for the case $p=2$. 
Recently Blatt--Vorderobermeier--Hopper \cite{BVH} independently proved the existence of local-in-time weak solutions to the problem \eqref{eq:P} with $p \ge 2$ for closed space curves in $\mathbb{R}^n$.   
Although Definition~\ref{theorem:1.1} is slightly different from the definition of weak solutions in \cite{BVH} or \cite{OPW}, 
Theorem~\ref{theorem:1.2} gives an extension of \cite{BVH,OPW} for the case $p>2$.  

Thanks to Theorem \ref{theorem:1.2}, it is natural to ask whether weak solutions to \eqref{eq:P} converge to an equilibrium as $t \to \infty$. 
To this end, generally one requires a uniform bound for $\mathcal{E}_p$ along the flow with respect to $t$. 
Although \eqref{eq:1.3} and \eqref{eq:1.4} give a uniform estimate of weak solutions of \eqref{eq:P}, 
%we obtain a partial result on the energy bound for the case $p \ge 2$ (see Lemma \ref{theorem:4.7}), 
it is not enough to prove the convergence of weak solutions as $t \to \infty$. 
We expect that one of the key difficulties may be the lack of uniqueness of weak solutions of~\eqref{eq:P} with $p > 2$. 
If the uniqueness of weak solutions to problem \eqref{eq:P} with $p>2$ is proved, perhaps for a certain class, 
then we can also prove the subconvergence of weak solutions of problem \eqref{eq:P} in this class with $p>2$. 
For the case $p=2$, we obtain the uniqueness of weak solutions to \eqref{eq:P}. 
This gives the following energy inequality stronger than \eqref{eq:1.5}: 
%%%%%%%%%%%%%%%%%%%%%%%%%%%%%%%%%%%%%
\begin{theorem} \label{theorem:1.3}
Let $p=2$. Assume that $\gamma_0 \in W^{2,2}(\mathcal{S}^1;\mathbb{R}^2)$ satisfies \eqref{eq:1.1}. 
Then the problem \eqref{eq:P} possesses a unique global-in-time weak solution $\gamma$ such that 
\begin{equation}
\label{eq:1.6}
\mathcal{E}_2(\gamma(\tau_2)) - \mathcal{E}_2(\gamma(\tau_1)) 
 \le -\dfrac{1}{2} \int^{\tau_2}_{\tau_1} \!\!\! \int^1_0 \mathcal{L}(\gamma) |\partial_t \gamma|^2 dx dt
\end{equation}
for all $0 \le \tau_1 \le \tau_2 < \infty$. 
Moreover, there exist a sequence $\{ p_k \} \subset \mathbb{R}^2$ and a monotone divergent sequence $\{ t_k \} \subset (0, \infty)$ such that 
$\gamma(\cdot,t_k)-p_k$ converges to an elastica in the $H^4$-weak topology. 
\end{theorem}
%%%%%%%%%%%%%%%%%%%%%%%%%%%%%%%%%%%%%
One can infer from Theorem \ref{theorem:1.3} that our weak formulation of \eqref{eq:P} can work well at least for the case $p=2$. 
Moreover, Theorem \ref{theorem:1.3} gives an extension of \cite{OPW} for the case $p=2$.  

The paper is organized as follows. 
In Section~\ref{section:2} we collect several inequalities used in this paper. 
We construct approximations of weak solutions to~\eqref{eq:P} via minimizing movements in Section~\ref{section:3}: 
we prove the existence and regularity of approximating solutions in Section~\ref{subsection:3.1}; and the convergence of approximating solutions in Section~\ref{subsection:3.2}.  
In Section~\ref{section:4} we prove Theorem~\ref{theorem:1.2}. 
Finally we prove Theorem~\ref{theorem:1.3} in Section~\ref{section:5}.

%%%%%%%%%%%%%%%%%%%%%%%%%%%%%%%%%%%%%
%%%%%%%%%%%%%%%%%%%%%%%%%%%%%%%%%%%%%
%%%%%%%%%%%%%%%%%%%%%%%%%%%%%%%%%%%%%

\section{Preliminaries} \label{section:2}

We use the following interpolation inequalities (see e.g. \cite{Adams}, \cite[Theorem 6.4]{FFLM_2012}). 
%%%%%%%%%%%%%%%%%%%%%%%%%%%%%%%%%%%%%%%%%%%%%%%%%%%%
\begin{proposition} \label{theorem:2.1}
Let $\Omega \subset \mathbb{R}^N$ be a bounded open set satisfying the cone condition. 
Let $k$, $l$ and $m$ be integers such that $0 \le k \le l \le m$. 
Let $1 \le q \le r < \infty$ if $(m-l)q \ge N$, or 
let $1 \le q \le r \le \infty$ if $(m-l)q > N$. 
Then there exists $A>0$ such that for all $u \in W^{m,q}(\Omega)$ it holds 
$$
\| D^l u \|_{L^r(\Omega)} \le A ( \| D^m u \|^\theta_{L^q(\Omega)} \| D^k u \|^{1-\theta}_{L^q(\Omega)} + \| D^k u \|_{L^q(\Omega)}), 
$$
where 
$$
\theta := \dfrac{1}{m-k}\Bigl( \dfrac{N}{q} - \dfrac{N}{r} + l - k \Bigr). 
$$
In particular, if $u \in W^{m,q}_0(\Omega)$,  then 
$$
\| D^l u \|_{L^r(\Omega)} \le A \| D^m u \|^\theta_{L^q(\Omega)} \| D^k u \|^{1-\theta}_{L^q(\Omega)}.  
$$
\end{proposition}
%%%%%%%%%%%%%%%%%%%%%%%%%%%%%%%%%%%%%%%%%%%%%%%%%%%%

%%%%%%%%%%%%%%%%%%%%%%%%%%%%%%%%%%%%%%%%%%%%%%%%%%%%
\begin{lemma}
\label{theorem:2.2}
Let $p \ge 2$. 
Let $\gamma : \mathcal{S}^1 \to \mathbb{R}^2$ be a closed curve. 
Then 
\begin{equation}
\label{eq:2.1}
\mathcal{L}(\gamma) \ge \Bigl[ \dfrac{(2 \pi)^{p}}{p E_p(\gamma)} \Bigr]^{\frac{1}{p-1}}.  
\end{equation}
\end{lemma}
%%%%%%%%%%%%%%%%%%%%%%%%%%%%%%%%%%%%%%%%%%%%%%%%%%%%
\begin{proof}
By Poincar\`e's inequality (e.g., see \cite[Proposition 5.9]{BGH}) we have 
\begin{align*}
\mathcal{L}(\gamma)=\int^{\mathcal{L}(\gamma)}_0 |\gamma_s|^2 \, ds \le \dfrac{\mathcal{L}(\gamma)^2}{4 \pi^2} \int^{\mathcal{L}(\gamma)}_0 |\gamma_{ss}|^2 \, ds,  
\end{align*}
where $s$ denotes the arc length parameter of $\gamma$. 
This together with H\"older's inequality implies that  
\begin{align*}
\dfrac{1}{\mathcal{L}(\gamma)} \le \dfrac{1}{4 \pi^2} \mathcal{L}(\gamma)^{\frac{p-2}{p}} \Bigl[ \int^{\mathcal{L}(\gamma)}_0 |\gamma_{ss}|^p \, ds \Bigr]^{\frac{2}{p}} 
 = \dfrac{p^{\frac{2}{p}}}{4 \pi^2} \mathcal{L}(\gamma)^{\frac{p-2}{p}} E_p(\gamma)^{\frac{2}{p}}. 
\end{align*}
Then we obtain \eqref{eq:2.1}. 
\end{proof}

%%%%%%%%%%%%%%%%%%%%%%%%%%%%%%%%%%%%%%%%%%%%%%%%%%%%
%%%%%%%%%%%%%%%%%%%%%%%%%%%%%%%%%%%%%%%%%%%%%%%%%%%%
%%%%%%%%%%%%%%%%%%%%%%%%%%%%%%%%%%%%%%%%%%%%%%%%%%%%
\section{Approximate solutions} \label{section:3}
Let $\gamma_0 \in W^{2,p}(\mathcal{S}^1;\mathbb{R}^2)$ satisfy \eqref{eq:1.1}. 
In this section, we fix such $\gamma_0$ arbitrarily, and denote the admissible set $\mathcal{AC}_{\gamma_0}$ by $\mathcal{AC}$ for short. 

We note that for $\gamma \in \mathcal{AC}$ we have 
\begin{equation}
\label{eq:3.1}
\gamma_{xx}(x) = \mathcal{L}(\gamma) \kappa_\gamma(x) \mathcal{R} (\gamma_x)(x) \quad \text{for a.e.} \quad x \in I, 
\end{equation}
where 
$$
\mathcal{R}:= \begin{pmatrix} 0 & -1 \\ 1 & 0 \end{pmatrix}. 
$$ 
We construct a family of approximate solutions via minimizing movements. 
Let $n \in \mathbb{N}$, $T>0$, and set $\tau_n := T/n$. 
Let $\gamma_{0,n}:= \gamma_0 \in \mathcal{AC}$. 
We define $\{\gamma_{i,n}\}^{n}_{i=0}$ inductively. 
More precisely, we define $\gamma_{i,n}$ as a minimizer of the minimization problem: 
\begin{equation}
\label{Min} \tag{${\rm M}_{i,n}$}
\min_{\gamma \in \mathcal{AC}} G_{i,n}(\gamma), 
\end{equation}
where $G_{i,n}(\gamma) := \mathcal{E}_p(\gamma) + P_n(\gamma, \gamma_{i-1,n})$ with 
\begin{align*}
P_n(\gamma, \tilde{\gamma}) := \dfrac{\mathcal{L}(\tilde{\gamma})}{2 \tau_n} \int^1_0 |\gamma-\tilde{\gamma}|^2 \, dx  
\end{align*}
for $\gamma, \tilde{\gamma} \in \mathcal{AC}$. 

%%%%%%%%%%%%%%%%%%%%%%%%%%%%%%%%%%%%%%%%%%%%%%%%%%
%%%%%%%%%%%%%%%%%%%%%%%%%%%%%%%%%%%%%%%%%%%%%%%%%%
\subsection{Existence and regularity} \label{subsection:3.1}
To begin with, we prove the existence of approximate solutions $\{ \gamma_{i,n} \}$: 
%%%%%%%%%%%%%%%%%%%%%%%%%%%%%%%%%%%%%%%
\begin{lemma} \label{theorem:3.1}
For each $i=1$, $2,\ldots$, problem \eqref{Min} has a minimizer $\gamma_{i,n}$.  
\end{lemma}
%%%%%%%%%%%%%%%%%%%%%%%%%%%%%%%%%%%%%%%
\begin{proof}
If $\gamma_{i,n} \in \mathcal{AC}$ is a solution to \eqref{Min} for some $i \in \mathbb{N}$, 
by the minimality of $\gamma_{i,n}$, we have 
\begin{equation}
\label{eq:3.2}
G_{i,n}(\gamma_{i,n}) = \inf_{\gamma \in \mathcal{AC}}G_{i,n}(\gamma) \le G_{i,n}(\gamma_{i-1,n}) = \mathcal{E}_p(\gamma_{i-1,n}). 
\end{equation}
This together with the non-negativity of $P_n(\cdot, \cdot)$ implies that 
\begin{equation*}
\mathcal{E}_p(\gamma_{i,n}) \le \mathcal{E}_p(\gamma_{i-1,n}),  
\end{equation*}
in particular, 
\begin{equation}
\label{eq:3.3}
\mathcal{E}_p(\gamma_{i,n}) \le \mathcal{E}_p(\gamma_0). 
\end{equation}

Let $\{ \gamma_j \} \subset \mathcal{AC}$ be a minimizing sequence for \eqref{Min}, that is,  
$$
\lim_{j \to \infty} G_{i,n}(\gamma_j) = \inf_{\gamma \in \mathcal{AC}}G_{i,n}(\gamma). 
$$
By \eqref{eq:3.2} and \eqref{eq:3.3} we may assume that  
\begin{equation}
\label{eq:3.4}
G_{i,n}(\gamma_j) \le 2 \mathcal{E}_p(\gamma_0) \quad \text{for all} \quad j \in \mathbb{N}.   
\end{equation}
This together with the non-negativity of $P_n(\cdot, \cdot)$ implies that    
\begin{equation}
\label{eq:3.5}
\dfrac{\mathcal{L}(\gamma_j)}{p} \int^1_0 |\kappa_{\gamma_j}|^p \, dx + \lambda \mathcal{L}(\gamma_j) 
 \le 2 \mathcal{E}_p(\gamma_0) \quad \text{for} \quad j \in \mathbb{N}.  
\end{equation}
Since $\gamma_{j} \in \mathcal{AC}$, we deduce from \eqref{eq:3.5} that  
\begin{equation}
\label{eq:3.6}
\begin{split}
\int^1_0 |(\gamma_j)_x|^p \, dx 
  = \mathcal{L}(\gamma_j)^p 
  \le \Bigl( \dfrac{2}{\lambda} \mathcal{E}_p(\gamma_0) \Bigr)^p 
\end{split}
\end{equation}
for all $j \in \mathbb{N}$. 
Combining \eqref{eq:3.1} with \eqref{eq:3.5}, we obtain 
\begin{align}
\label{eq:3.7}
\begin{split}
\int^1_0 |(\gamma_j)_{xx}|^p \, dx &= \int^1_0 |\mathcal{L}(\gamma_j) \kappa_{\gamma_j} \mathcal{R}(\gamma_j)_x|^p \, dx \\
  &\le \mathcal{L}(\gamma_j)^{2p-1} \cdot \mathcal{L}(\gamma_j) \int^1_0 |\kappa_{\gamma_j}|^p \, dx 
    \le  p \lambda \Bigl( \dfrac{2}{\lambda} \mathcal{E}_p(\gamma_0) \Bigr)^{2p}.   
\end{split}
\end{align}
It follows from \eqref{eq:3.4} that 
\begin{align*}
\int^1_0 |\gamma_j|^2 \, dx 
&\le 2 \int^1_0 |\gamma_j - \gamma_{i-1,n}|^2 \, dx +  2 \int^1_0 |\gamma_{i-1,n}|^2 \, dx \\
&= \dfrac{4 \tau_n}{\mathcal{L}(\gamma_{i-1,n})} P_n(\gamma_j, \gamma_{i-1,n}) + 2 \int^1_0 |\gamma_{i-1,n}|^2 \, dx \\
&\le \dfrac{8 \tau_n \mathcal{E}(\gamma_0)}{\mathcal{L}(\gamma_{i-1,n})} + 2 \int^1_0 |\gamma_{i-1,n}|^2 \, dx.  
\end{align*}
This together with Lemma \ref{theorem:2.2} and \eqref{eq:3.4} implies that 
\begin{equation}
\label{eq:3.8}
\int^1_0 |\gamma_j|^2 \, dx 
\le 8 \tau_n \mathcal{E}_p(\gamma_0) \Bigl[ \dfrac{p \mathcal{E}_p(\gamma_0)}{(2 \pi)^p} \Bigr]^{\frac{1}{p-1}} + 2 \int^1_0 |\gamma_{i-1,n}|^2 \, dx. 
\end{equation}
Combining \eqref{eq:3.7} and \eqref{eq:3.8} with Proposition \ref{theorem:2.1}, 
we find a constant $C>0$ depending only on $p$, $\lambda$, $\mathcal{E}_p(\gamma_0)$, $\|\gamma_{i-1,n}\|_{L^2(\mathcal{S}^1)}$ and $T$ such that 
\begin{equation}
\label{eq:3.9}
\begin{split}
\|\gamma_j \|_{L^p(\mathcal{S}^1)} 
 &\le A \bigl( \| (\gamma_j)_{xx} \|_{L^2(\mathcal{S}^1)}^\theta \| \gamma_j \|_{L^2(\mathcal{S}^1)}^{1-\theta} + \| \gamma_j \|_{L^2(\mathcal{S}^1)} \bigr) \\
 &\le A \bigl( \| (\gamma_j)_{xx} \|_{L^p(\mathcal{S}^1)}^\theta \| \gamma_j \|_{L^2(\mathcal{S}^1)}^{1-\theta} + \| \gamma_j \|_{L^2(\mathcal{S}^1)} \bigr) 
 \le C,  
\end{split}
\end{equation}
where $\theta = (p-2)/(4p)$. 
Thanks to \eqref{eq:3.6}, \eqref{eq:3.7} and \eqref{eq:3.9}, we find $\gamma \in W^{2,p}(\mathcal{S}^1)$ such that 
\begin{align}
&\gamma_j \rightharpoonup \gamma \quad \text{weakly in} \quad W^{2,p}(\mathcal{S}^1), \label{eq:3.10} \\
&\gamma_j \to \gamma \quad \text{in} \quad C^{1,\alpha}(\mathcal{S}^1), \label{eq:3.11}
\end{align}
up to a subsequence, where $\alpha \in (0, 1-1/p)$.  
By \eqref{eq:3.11} and $\gamma_j \in \mathcal{AC}$ we see that $\gamma \in \mathcal{AC}$. 

Finally we verify that $\gamma$ is the desired minimizer. 
We deduce from \eqref{eq:3.10} and \eqref{eq:3.11} that 
\begin{equation}
\label{eq:3.12}
\liminf_{n \to \infty} \| (\gamma_j)_{xx} \|_{L^p(\mathcal{S}^1)} \ge \| \gamma_{xx} \|_{L^p(\mathcal{S}^1)}. 
\end{equation}
Since $\gamma \in \mathcal{AC}$ implies that $\gamma_x \cdot \gamma_{xx}=0$, i.e., 
\begin{equation*}
|\gamma_{xx}| = \dfrac{|\gamma_{xx} \cdot \mathcal{R}\gamma_x|}{\mathcal{L}(\gamma)} \quad \text{for each} \quad \gamma \in \mathcal{AC},  
\end{equation*}
this together with \eqref{eq:3.11} and \eqref{eq:3.12} implies that 
\begin{align*}
\liminf_{j \to \infty} E_p(\gamma_j) 
 &= \liminf_{n \to \infty} \dfrac{1}{p \mathcal{L}(\gamma_j)^{2p-1}} \int^1_0 |(\gamma_j)_{xx}|^p \, dx \\
 &= \dfrac{1}{p \mathcal{L}(\gamma)^{2p-1}} \liminf_{n \to \infty} \int^1_0 |(\gamma_j)_{xx}|^p \, dx \\
 &\ge \dfrac{1}{p \mathcal{L}(\gamma)^{2p-1}} \int^1_0 |\gamma_{xx}|^p \, dx 
  = E_p(\gamma). 
\end{align*}
Thus we obtain 
$$
\inf_{\gamma \in \mathcal{AC}}G_{i,n}(\gamma) = \lim_{j \to \infty} G_{i,n}(\gamma_j) 
 \ge G_{i,n}(\gamma). 
$$
Therefore Lemma \ref{theorem:3.1} follows. 
\end{proof}

In the following, we define $V_{i,n} : \mathcal{S}^1 \to \mathbb{R}^2$ by 
$$
V_{i,n}(x) := \frac{\gamma_{i,n}(x) - \gamma_{i-1,n}(x)}{\tau_{n}}.
$$

We define the piecewise linear interpolation of $\{\gamma_{i,n}\}$ as follows: 
%%%%%%%%%%%%%%%%%%%%%%%%%%%%%%%%%%%%%%%%%%%%%%%%%%%%%%%%%%
\begin{definition} \label{theorem:3.2}
We define $\gamma_{n}(x,t) : \mathcal{S}^1 \times [0,T] \to \mathbb{R}^2$ by 
$$
\gamma_{n}(x,t):=\gamma_{i-1,n}(x) + \bigl( t-(i-1)\tau_{n} \bigr)V_{i,n}(x), 
$$
if $(x,t) \in \mathcal{S}^1 \times [(i-1)\tau_{n},i \tau_{n}]$ for each $i=1,\ldots,n$. 
\end{definition}
%%%%%%%%%%%%%%%%%%%%%%%%%%%%%%%%%%%%%%%%%%%%%%%%%%%%%%%%%%

Furthermore, we also make use of piecewise constant interpolations of $\{\gamma_{i,n}\}$ and $\{ V_{i,n}\}$:  
%%%%%%%%%%%%%%%%%%%%%%%%%%%%%%%%%%%%%%%%
\begin{definition} \label{theorem:3.3}
We define $\tilde{\gamma}_{n} : \mathcal{S}^1 \times (0,T] \to \mathbb{R}^2$, $\tilde{\Gamma}_n : \mathcal{S}^1 \times (0,T] \to \mathbb{R}^2$ and $V_n : \mathcal{S}^1 \times (0,T] \to \mathbb{R}^2$ as 
$$
\tilde{\gamma}_{n}(x,t) :=\gamma_{i,n}(x), \quad 
\tilde{\Gamma}_n(x,t) := \gamma_{i-1,n}(x), \quad  
V_{n}(x,t) :=V_{i,n}(x), 
$$
if $(x,t) \in \mathcal{S}^1 \times ((i-1) \tau_{n},i \tau_{n}]$ for each $i=1, \ldots, n$, respectively. 
\end{definition}

From now on, we consider the regularity of approximate solutions. 
First we have: 
%%%%%%%%%%%%%%%%%%%%%%%%%%%%%%%%%%%%%%%
\begin{lemma} \label{theorem:3.4}
Let $\{ \gamma_{i,n}\}$ be a family of closed curves obtained by Lemma \ref{theorem:3.1}. 
There exist constants $C^*>0$ and $C_*>0$ being independent of $n$ such that 
\begin{align}
\sup_{1 \le i \le n} \| \gamma_{i,n} \|_{W^{2,p}(\mathcal{S}^1)} & \le C^*,  \label{eq:3.13} \\
\int^{T}_0\!\!\! \int^1_0 |V_n(x,t)|^2 \, dx dt & \le C_*. \label{eq:3.14}
\end{align}
\end{lemma}
%%%%%%%%%%%%%%%%%%%%%%%%%%%%%%%%%%%%%%%
\begin{proof}
First we prove \eqref{eq:3.14}. 
By the minimality of $\gamma_{i,n}$ we have 
\begin{align*}
G_{i,n}(\gamma_{i,n}) \le G_{i,n}(\gamma_{i-1,n}) = \mathcal{E}_p(\gamma_{i-1,n}).  
\end{align*}
This clearly implies that 
\begin{equation}
\label{eq:3.15}
P_n(\gamma_{i,n}, \gamma_{i-1,n}) \le \mathcal{E}_p(\gamma_{i-1,n}) - \mathcal{E}_p(\gamma_{i,n}). 
\end{equation}
Summing \eqref{eq:3.15} over $i=1, 2, \ldots, n$, we obtain 
\begin{equation}
\label{eq:3.16}
\sum^{n}_{i=1} P_n(\gamma_{i,n}, \gamma_{i-1,n}) 
  \le \mathcal{E}_p(\gamma_{0}) - \mathcal{E}_p(\gamma_{n,n}) \le \mathcal{E}_p(\gamma_{0}). 
\end{equation}
On the other hand, we observe from the definition of $P_n$ and Lemma \ref{theorem:2.2} that  
\begin{align*}
\sum^{n}_{i=1} P_n(\gamma_{i,n}, \gamma_{i-1,n}) 
 &= \sum^{n}_{i=1} \dfrac{\mathcal{L}(\gamma_{i-1,n})}{2 \tau_n} \int^1_0 |\gamma_{i,n} - \gamma_{i-1,n}|^2\, dx \\
 &= \dfrac{1}{2} \int^T_0 \mathcal{L}(\tilde{\Gamma}_{n}) \int^1_0 |V_{n}|^2\, dx dt \\
 &\ge \dfrac{1}{2}\Bigl[ \dfrac{(2 \pi)^{p}}{p \mathcal{E}_p(\gamma_0)} \Bigr]^{\frac{1}{p-1}} \int^T_0 \!\!\! \int^1_0 |V_{n}|^2\, dx dt. 
\end{align*}
This together with \eqref{eq:3.16} implies \eqref{eq:3.14}. 

We turn to \eqref{eq:3.13}. 
Since $\gamma_{i,n} \in \mathcal{AC}$, by \eqref{eq:3.6} and \eqref{eq:3.7} we have  
\begin{align}
& \int^1_0 |(\gamma_{i,n})_x|^p \, dx \le \Bigl( \dfrac{2}{\lambda} \mathcal{E}_p(\gamma_0) \Bigr)^p, \label{eq:3.17}\\
& \int^1_0 |(\gamma_{i,n})_{xx}|^p \, dx \le p \lambda \Bigl( \dfrac{2}{\lambda} \mathcal{E}_p(\gamma_0) \Bigr)^{2p}. \label{eq:3.18}
\end{align}
Since $\partial_t \gamma_n(x,t)= V_n(x,t)$ and $\gamma_n(x,i\tau_n)=\gamma_{i,n}(x)$, we deduce from \eqref{eq:3.14} that 
\begin{equation}
\label{eq:3.19}
\begin{aligned}
\int^1_0 |\gamma_{i,n} - \gamma_0|^2 \, dx 
 &= \int^1_0 \Bigl| \int^{i \tau_n}_0 \partial_t \gamma_n(x,t) \, dt \Bigr|^2 \, dx \\
 &\le i \tau_n \int^{i \tau_n}_0 \!\!\! \int^1_0 |V_n(x,t)|^2 \, dx dt
 \le T C_*, 
\end{aligned}
\end{equation}
where we used Jensen's inequality and Fubini's theorem. 
This together with \eqref{eq:3.9}, \eqref{eq:3.17} and \eqref{eq:3.18} implies \eqref{eq:3.13}.  
Therefore Lemma \ref{theorem:3.4} follows. 
\end{proof}

In order to study the extra regularity of approximate solutions, we make use of the Euler-Lagrange equation for $\{ \gamma_{i,n}\}$ in a weak sense. 
To this end, first we have:  
%%%%%%%%%%%%%%%%%%%%%%%%%%%%%%%%%%%%%%%%%%%%%%%%%%%%
\begin{lemma} \label{theorem:3.5}
For $\gamma \in \mathcal{AC}$, $\eta \in W^{2,p}(\mathcal{S}^1; \mathbb{R}^2)$, $0 < \delta < \mathcal{L}(\gamma)/\| \eta_x \|_{L^\infty(\mathcal{S}^1)}$, 
there exists a unique $\Phi(\delta, \cdot) : [0,1] \to [0,1]$ such that 
\begin{equation}
\label{eq:3.20}
\mu(\delta,x) := (\gamma + \delta \eta)(\Phi(\delta,x))
\end{equation}
satisfies $\mu(\delta, \cdot) \in \mathcal{AC}$. 
Moreover, it holds that 
\begin{align}
&\Phi_\delta(\delta,x)|_{\delta=0} 
 = \dfrac{1}{\mathcal{L}(\gamma)^2}\Bigl( x \int^1_0 \gamma_x \cdot \eta_x \, dx - \int^x_0 \gamma_x \cdot \eta_x \, dx \Bigr), \label{eq:3.21} \\
&\Phi_x(\delta,x)|_{\delta=0} = 1. \label{eq:3.22} 
\end{align}
\end{lemma}
%%%%%%%%%%%%%%%%%%%%%%%%%%%%%%%%%%%%%%%%%%%%%%%%%%%%
\begin{proof}
We prove Lemma \ref{theorem:3.5} along the argument of the proof of \cite[Lemma 5]{Badal}. 
We define $\Psi(\delta,\cdot) : [0,1] \to \mathbb{R}$ as 
\begin{equation}
\label{eq:3.23}
\Psi(\delta,x') := \dfrac{1}{\mathcal{L}(\gamma + \delta \eta)} \int^{x'}_0 |\gamma_x + \delta \eta_x| \, d\tilde{x}. 
\end{equation}
For $0 < \delta < \mathcal{L}(\gamma)/\| \eta_x \|_{L^\infty(\mathcal{S}^1)}$ we have 
$$
|\gamma_x + \delta \eta_x| \ge |\gamma_x| - \delta \| \eta_x \|_{L^\infty(\mathcal{S}^1)} = \mathcal{L}(\gamma) - \delta \| \eta_x \|_{L^\infty(\mathcal{S}^1)} > 0. 
$$
This implies that $\Psi_{x'}>0$ for all $x' \in [0, 1]$. 
Since $\Psi(\delta,0)=0$ and $\Psi(\delta,1)=1$, we see that $\Psi(\delta,\cdot)$ is a diffeomorphism from $[0,1]$ to itself. 
Here we define $\Phi(\delta, \cdot) : [0,1] \to [0,1]$ as 
$$
\Phi(\delta,s):= \Psi(\delta,\cdot)^{-1}(s). 
$$
Let $\mu$ satisfy \eqref{eq:3.20}. 
By the definition we have $\Phi(\delta,0)=0$ and $\Phi(\delta,1)=1$. 
Hence, if $\eta \in W^{2,p}(\mathcal{S}^1;\mathbb{R}^2)$, it holds that 
\begin{equation}
\label{eq:3.24}
\mu(\delta,\cdot) \in W^{2,p}(\mathcal{S}^1; \mathbb{R}^2). 
\end{equation}
It follows from $\Psi(\delta, \Phi(\delta,x))=x$ that 
\begin{align}
& \Psi_{x'}(\delta,\Phi(\delta,x)) \Phi_x(\delta,x) = 1, \label{eq:3.25} \\
& \Psi_\delta(\delta, \Phi(\delta, x)) + \Psi_{x'}(\delta, \Phi(\delta,x)) \Phi_\delta(\delta,x) = 0. \label{eq:3.26}
\end{align}
Moreover, we deduce from \eqref{eq:3.23} that 
\begin{align}
& \Psi_{x'}(\delta,x') = \dfrac{|\gamma_x(x') + \delta \eta_x(x')|}{\mathcal{L}(\gamma + \delta \eta)}, \label{eq:3.27} \\
\label{eq:3.28}
\begin{split}
& \Psi_\delta(\delta,x') 
  = \dfrac{1}{\mathcal{L}(\gamma + \delta \eta)} \int^{x'}_0 \dfrac{\gamma_x + \delta \eta_x}{|\gamma_x + \delta \eta_x|} \cdot \eta_x \, d\tilde{x} \\
& \qquad \qquad \quad
    - \dfrac{1}{\mathcal{L}(\gamma + \delta \eta)^2} \int^1_0 \dfrac{\gamma_x + \delta \eta_x}{|\gamma_x + \delta \eta_x|} \cdot \eta_x \, d\tilde{x} 
          \int^{x'}_0 |\gamma_x + \delta \eta_x| \, d\tilde{x}. 
\end{split}
\end{align}
Plugging \eqref{eq:3.27} into \eqref{eq:3.25}, we have 
\begin{equation}
\label{eq:3.29}
\Phi_x(\delta,x)= \dfrac{\mathcal{L}(\gamma + \delta \eta)}{|(\gamma_x + \delta \eta_x)(\Phi(\delta,x))|}. 
\end{equation}
This together with $\Phi(0,x)=x$ implies \eqref{eq:3.22}. 
It follows from \eqref{eq:3.29} that 
\begin{align*}
(\mu(\delta,x))_x
 &= (\gamma_x + \delta \eta_x)(\Phi(\delta,x)) \Phi_x(\delta,x) \\
 &= (\gamma_x + \delta \eta_x)(\Phi(\delta,x)) \dfrac{\mathcal{L}(\gamma + \delta \eta)}{|(\gamma_x + \delta \eta_x)(\Phi(\delta,x))|}. 
\end{align*}
This implies that 
\begin{equation}
\label{eq:3.30}
|(\mu(\delta,x))_x|= \mathcal{L}(\gamma +\delta \eta).  
\end{equation}
Thus, it follows from \eqref{eq:3.24} and \eqref{eq:3.30} that $\mu(\delta,\cdot) \in \mathcal{AC}$. 
By \eqref{eq:3.26}, \eqref{eq:3.27} and \eqref{eq:3.28} we have 
\begin{align*}
\Phi_\delta(\delta,x)|_{\delta=0} = -\dfrac{\Psi_\delta(0, \Phi(0,x))}{\Psi_{x'}(0,\Phi(0,x))} 
 &= -\dfrac{\Psi_\delta(0,x)}{\Psi_{x'}(0,x)} 
 = -\Psi_\delta(0,x) \\
 &= \dfrac{1}{\mathcal{L}(\gamma)^2}\Bigl( x \int^1_0 \gamma_x \cdot \eta_x \, dx - \int^x_0 \gamma_x \cdot \eta_x \, dx \Bigr), 
\end{align*}
where we used $\Phi(0,x)=x$. Thus \eqref{eq:3.21} follows. 
Therefore Lemma \ref{theorem:3.5} follows. 
\end{proof}

From now on, we set 
\begin{equation}
\label{eq:3.31}
\Phi_1(\gamma,\eta):= \dfrac{1}{\mathcal{L}(\gamma)^2} \Bigl(x \int^1_0 \gamma_x \cdot \eta_x \, d\tilde{x} - \int^x_0 \gamma_x \cdot \eta_x \, d\tilde{x}\Bigr).
\end{equation}

%%%%%%%%%%%%%%%%%%%%%%%%%%%%%%%%%%%%%%%%%%%%%%%%%%%%%
\begin{lemma} \label{theorem:3.6}
Fix $\tilde{\gamma} \in \mathcal{AC}$ and let $\gamma$ be a minimizer of 
$$
\min_{\mu \in \mathcal{AC}} [ \mathcal{E}_p(\mu) + P_n(\mu, \tilde{\gamma})]. 
$$
Then it holds that 
\begin{align*}
&\int^1_0 \Bigl[ \dfrac{|\kappa|^{p-2}}{\mathcal{L}(\gamma)^3} \gamma_{xx} \cdot \eta_{xx} 
                 -\dfrac{2p-1}{p} \dfrac{|\kappa|^p}{\mathcal{L}(\gamma)} \gamma_x \cdot \eta_x + \dfrac{\lambda}{\mathcal{L}(\gamma)} \gamma_x \cdot \eta_x 
                  + \mathcal{L}(\tilde{\gamma}) \dfrac{\gamma - \tilde{\gamma}}{\tau_n} \cdot \eta \Bigr] \, dx \\
 & \qquad + \mathcal{L}(\tilde{\gamma}) \int^1_0 \dfrac{\gamma - \tilde{\gamma}}{\tau_n} \cdot \Phi_1(\gamma,\eta) \gamma_x \, dx =0
\end{align*}
for all $\eta \in W^{2,p}(\mathcal{S}^1;\mathbb{R}^2)$.
\end{lemma}
%%%%%%%%%%%%%%%%%%%%%%%%%%%%%%%%%%%%%%%%%%%%%%%%%%%%%
\begin{proof}
First we derive the first variation of $E_p(\mu(\delta, \cdot))$. 
By \eqref{eq:3.29} we have  
\begin{align*}
E_p(\mu(\delta,\cdot)) 
&= \dfrac{1}{p} \int^1_0 \dfrac{|(\mu(\Phi(\delta,x)))_{xx} \cdot \mathcal{R} (\mu(\Phi(\delta,x)))_x|^p}{|(\mu(\Phi(\delta,x)))_x|^{3p-1}} \, dx \\
&= \dfrac{1}{p} \int^1_0 \dfrac{|(\gamma_{xx} + \delta \eta_{xx})(\Phi(\delta,x)) \cdot 
            \mathcal{R} (\gamma_x + \delta \eta_x)(\Phi(\delta,x))|^p}{| (\gamma_x + \delta \eta_x)(\Phi(\delta,x)) |^{3p-1}} \Phi_x(\delta,x) \, dx \\
&= \dfrac{1}{p} \int^1_0 \dfrac{|(\gamma_{xx} + \delta \eta_{xx})(\tilde{x}) \cdot 
            \mathcal{R} (\gamma_x + \delta \eta_x)(\tilde{x})|^p}{|(\gamma_x + \delta \eta_x)(\tilde{x})|^{3p-1}} \, d\tilde{x}. 
\end{align*}
Thus we obtain 
\begin{align}
\label{eq:3.32}
\begin{split}
\dfrac{d}{d \delta} E_p(\mu(\delta,\cdot)) \Bigm|_{\delta=0} 
 &= \int^1_0 \dfrac{|\gamma_{xx} \cdot \mathcal{R} \gamma_x|^{p-2}}{|\gamma_x|^{3p-1}}(\gamma_{xx}\cdot \mathcal{R}\gamma_x) 
            \{ \mathcal{R}\gamma_x \cdot \eta_{xx} + \gamma_{xx} \cdot \mathcal{R}\eta_x \} \, dx \\
 & \quad - \dfrac{3p-1}{p} \int^1_0 
     \dfrac{|(\gamma_{xx} ) \cdot \mathcal{R} (\gamma_x )|^p (\gamma_x \cdot \eta_x)}{|\gamma_x|^{3p+1}}\, dx.            
\end{split}
\end{align}
Since $\gamma_{xx}=\mathcal{L}(\gamma) \kappa \mathcal{R}\gamma_x$ and $|\gamma_x|=\mathcal{L}(\gamma)$, we have 
\begin{align*}
&\dfrac{|\gamma_{xx} \cdot \mathcal{R} \gamma_x|^{p-2}}{|\gamma_x|^{3p-1}}(\gamma_{xx}\cdot \mathcal{R}\gamma_x) 
            \{ \mathcal{R}\gamma_x \cdot \eta_{xx} + \gamma_{xx} \cdot \mathcal{R}\eta_x \} \\
& \quad = \dfrac{|\kappa|^{p-2} \kappa}{\mathcal{L}(\gamma)^2} \{ \mathcal{R}\gamma_x \cdot \eta_{xx} + \mathcal{L}(\gamma)\kappa \mathcal{R}\gamma_x \cdot \mathcal{R}\eta_x \}\\
& \quad = \dfrac{|\kappa|^{p-2}}{\mathcal{L}(\gamma)^3} \gamma_{xx} \cdot \eta_{xx}  
           + \dfrac{|\kappa|^{p}}{\mathcal{L}(\gamma)} \gamma_x \cdot \eta_x,  
\end{align*}
where the last equality followed from 
$$
\mathcal{R} \gamma_x \cdot \mathcal{R}\eta_x = - \mathcal{R}(\mathcal{R} \gamma_x) \cdot \eta_x = -(-\gamma_x)\cdot \eta_x = \gamma_x \cdot \eta_x. 
$$
Similarly, we have 
$$
- \dfrac{3p-1}{p} \dfrac{|(\gamma_{xx} ) \cdot \mathcal{R} (\gamma_x )|^p (\gamma_x \cdot \eta_x)}{|\gamma_x|^{3p+1}} 
 = - \dfrac{3p-1}{p} \dfrac{|\kappa|^p}{\mathcal{L}(\gamma)} \gamma_x \cdot \eta_x. 
$$
Hence \eqref{eq:3.32} is reduced into 
\begin{align}
\label{eq:3.33}
\begin{split}
\dfrac{d}{d \delta} E_p(\mu(\delta,\cdot)) \Bigm|_{\delta=0} 
 &= \int^1_0 \Bigl[ \dfrac{|\kappa|^{p-2}}{\mathcal{L}(\gamma)^3} \gamma_{xx} \cdot \eta_{xx} 
                    -\dfrac{2p-1}{p} \dfrac{|\kappa|^p}{\mathcal{L}(\gamma)} \gamma_x \cdot \eta_x \Bigr] \, dx.            
\end{split}
\end{align}

We turn to the first variation on $\mathcal{L}(\mu(\delta,\cdot))$. 
Since 
$$
\mathcal{L}(\mu(\delta,\cdot)) = \int^1_0 |(\mu(\delta,x))_x|\, dx = \mathcal{L}(\gamma + \delta \eta), 
$$
we have 
\begin{equation}
\label{eq:3.34}
\dfrac{d}{d \delta} \mathcal{L}(\mu(\delta,\cdot)) \Bigm|_{\delta=0} 
 = \int^1_0 \dfrac{\gamma_x \cdot \eta_x}{|\gamma_x|} \, dx 
 = \int^1_0 \dfrac{1}{\mathcal{L}(\gamma)} \gamma_x \cdot \eta_x \, dx. 
\end{equation}

Finally we derive the first variation of $P_n(\mu(\delta,\cdot))$. 
Since
$$
(\mu(\delta,x))_\delta= \gamma_x(\Phi(\delta,x)) \Phi_\delta(\delta,x) + \eta(\Phi(\delta,x)) + \delta \eta_x(\Phi(\delta,x)) \Phi_\delta(\delta,x) 
$$
and 
$$
\Phi_\delta(\delta,x)|_{\delta=0} = \Phi_1(\gamma,\eta), 
$$
we have 
\begin{equation}
\label{eq:3.35}
\dfrac{d}{d \delta} P_n(\mu(\delta,\cdot),\tilde{\gamma}) \Bigm|_{\delta=0} 
 = \mathcal{L}(\tilde{\gamma}) \int^1_0 \dfrac{\gamma-\tilde{\gamma}}{\tau_n} \cdot \{ \eta + \Phi_1(\gamma,\eta) \gamma_x \} \, dx.                            
\end{equation}
By \eqref{eq:3.33}, \eqref{eq:3.34} and \eqref{eq:3.35}, we complete the proof. 
\end{proof}

Here we adopt the idea used in \cite[Proposition~3.2]{DD} and \cite[Theorem~3.9]{DDG}:  
%%%%%%%%%%%%%%%%%%%%%%%%%%%%%%%%%%%%%%%%%%%%%%%%%%%
\begin{lemma} \label{theorem:3.7}
For $\psi \in C^\infty(\mathcal{S}^1;\mathbb{R}^2)$, we define $\varphi_1 : \mathcal{S}^1 \to \mathbb{R}^2$, $\varphi_2 : \mathcal{S}^1 \to \mathbb{R}^2$ by 
\begin{align*}
\varphi_1(x) &:= \int^x_0 \!\!\! \int^{\xi}_0 \psi(s) ds d\xi + x \alpha + x^2 \beta, \\ 
\varphi_2(x) &:= \int^x_0 \psi(\xi) \, d\xi + 2 x \beta, 
\end{align*}
where 
\begin{equation*}
\alpha := - \beta - \int^1_0 \!\!\! \int^{\xi}_0 \psi(s)\, ds d\xi, \qquad 
\beta := - \dfrac{1}{2} \int^{1}_{0} \psi(s)\, ds. 
\end{equation*}
Then $\varphi_1, \varphi_2 \in W^{2,p}(\mathcal{S}^1)$ and it holds that 
\begin{equation}
\label{eq:3.36}
\begin{aligned}
& \max\{ \| \varphi_1 \|_{C^1(\mathcal{S}^1)}, \| \varphi_2\|_{L^\infty(\mathcal{S}^1)},  |\alpha|, |\beta| \} \le \dfrac{7}{2} \| \psi \|_{L^1(\mathcal{S}^1)}, \\
& \| \varphi'_2\|_{L^r(\mathcal{S}^1)} \le 2 \| \psi \|_{L^r(\mathcal{S}^1)} \quad \text{for} \quad r \ge 1.  
\end{aligned}
\end{equation}
\end{lemma}
%%%%%%%%%%%%%%%%%%%%%%%%%%%%%%%%%%%%%%%%%%%%%%%%%%%%
\begin{proof}
By the definition of $\varphi_1$, $\varphi_2$, $\alpha$ and $\beta$,  we see that $\varphi_1, \varphi_2 \in W^{2,p}(\mathcal{S}^1)$. 
Thus it suffices to prove estimate \eqref{eq:3.36}. 
To begin with, we have 
\begin{align*}
|\alpha| \le \dfrac{3}{2} \| \psi \|_{L^1(\mathcal{S}^1)}, \qquad 
|\beta| \le \dfrac{1}{2}\| \psi \|_{L^1(\mathcal{S}^1)}. 
\end{align*}
This clearly implies that 
\begin{align*}
\|\varphi_1 \|_{L^\infty(\mathcal{S}^1)} \le 3 \|\psi \|_{L^1(\mathcal{S}^1)}, \qquad 
\|\varphi_2 \|_{L^\infty(\mathcal{S}^1)} \le 2 \|\psi \|_{L^1(\mathcal{S}^1)}. 
\end{align*}
Moreover, since 
\begin{align*}
\varphi'_1(x)= \int^x_0 \psi(s) \, ds + \alpha + 2 \beta, 
\end{align*}
we also obtain 
\begin{align*}
\| \varphi'_1 \|_{L^\infty(\mathcal{S}^1)} 
 \le \| \psi \|_{L^1(\mathcal{S}^1)} + |\alpha| + 2 |\beta| 
 \le \dfrac{7}{2} \| \psi\|_{L^1(\mathcal{S}^1)}. 
\end{align*}
By $\varphi'_2(x)= \psi(x) + 2 \beta$, similarly we obtain 
\begin{align*}
\|\varphi'_2 \|_{L^r(\mathcal{S}^1)} \le \| \psi\|_{L^r(\mathcal{S}^1)} + 2 |\beta| 
\le \| \psi\|_{L^r(\mathcal{S}^1)} + \| \psi\|_{L^1(\mathcal{S}^1)}
\le 2 \| \psi\|_{L^r(\mathcal{S}^1)}.  
\end{align*}
Therefore Lemma \ref{theorem:3.7} follows. 
\end{proof}

Thanks to Lemma \ref{theorem:3.7}, we have: 
%%%%%%%%%%%%%%%%%%%%%%%%%%%%%%%%%%%%%%%%%%%%%%%%%%%%
\begin{lemma} \label{theorem:3.8}
Let $\tilde{\gamma}_{n}$ be the piecewise constant interpolation of $\{ \gamma_{i,n} \}$. 
Then there exists a constant $C>0$ being independent of $n$ such that  
\begin{align*}
& \int^T_0 \| \partial^2_x \tilde{\gamma}_n \|^{2(p-1)}_{L^\infty(\mathcal{S}^1)} \, dt \le C(T+1), \\
& \int^T_0 \| \partial_x (|\partial^2_x \tilde{\gamma}_{n}|^{p-2} \partial^2_x \tilde{\gamma}_{n}) \|_{L^2(\mathcal{S}^1)}^{2} \, dt \le C(T + 1), 
\end{align*}
for all $n \in \mathbb{N}$ and $i=1, 2, \ldots, n$. 
\end{lemma}
%%%%%%%%%%%%%%%%%%%%%%%%%%%%%%%%%%%%%%%%%%%%%%%%%%%%
\begin{proof}
By Lemma \ref{theorem:3.6} we have 
\begin{equation}
\label{eq:3.37}
\begin{aligned}
&\int^1_0 \dfrac{|\kappa_{i,n}|^{p-2}}{\mathcal{L}(\gamma_{i,n})^3} \partial^2_x \gamma_{i,n} \cdot \eta_{xx} \, dx 
 - \dfrac{2p-1}{p} \int^1_0 \dfrac{|\kappa_{i,n}|^p}{\mathcal{L}(\gamma_{i,n})} \partial_x \gamma_{i,n} \cdot \eta_x \, dx \\
& \qquad + \int^1_0 \dfrac{\lambda}{\mathcal{L}(\gamma_{i,n})} \partial_x \gamma_{i,n} \cdot \eta_x \, dx 
 + \mathcal{L}(\gamma_{i-1,j}) \int^1_0 V_{i,n} \cdot \eta \, dx \\
& \qquad + \mathcal{L}(\gamma_{i-1,j}) \int^1_0 V_{i,n} \cdot \Phi_1(\gamma_{i,n}, \eta) \partial_x \gamma_{i,n} \, dx = 0 
\end{aligned}
\end{equation}
for all $\eta \in W^{2,p}(\mathcal{S}^1;\mathbb{R}^2)$. 
Fix $\psi \in C^\infty(\mathcal{S}^1; \mathbb{R}^2)$ arbitrarily and define $\varphi_1 : \mathcal{S}^1 \to \mathbb{R}^2$ as in Lemma~\ref{theorem:3.7}. 
We take $\varphi_1$ as $\eta$ in \eqref{eq:3.37}. 
First we have 
\begin{align*}
&\int^1_0 \dfrac{|\kappa_{i,n}|^{p-2}}{\mathcal{L}(\gamma_{i,n})^3} \partial^2_x \gamma_{i,n} \cdot \partial^2_x \varphi_{1} \, dx \\
& \qquad = \int^1_0 \dfrac{|\kappa_{i,n}|^{p-2}}{\mathcal{L}(\gamma_{i,n})^3} \partial^2_x \gamma_{i,n} \cdot \psi \, dx 
    + 2 \int^1_0 \dfrac{|\kappa_{i,n}|^{p-2}}{\mathcal{L}(\gamma_{i,n})^3} \partial^2_x \gamma_{i,n} \cdot \beta \, dx. 
\end{align*}
Since $\partial^2_x \gamma_{i,n}=\mathcal{L}(\gamma_{i,n}) \kappa_{i,n} \mathcal{R} \partial_x \gamma_{i,n}$, it follows from Lemma \ref{theorem:3.7} that 
\begin{align*}
\Bigl|  \int^1_0 \dfrac{|\kappa_{i,n}|^{p-2}}{\mathcal{L}(\gamma_{i,n})^3} \partial^2_x \gamma_{i,n} \cdot \beta \, dx \Bigr| 
& \le |\beta| \int^1_0 \dfrac{|\kappa_{i,n}|^{p-1}}{\mathcal{L}(\gamma_{i,n})} \, dx \\
& \le C \| \psi \|_{L^1(\mathcal{S}^1)} E_p(\gamma_{i,n})^{(p-1)/p} 
   \le C \| \psi \|_{L^1(\mathcal{S}^1)}. 
\end{align*}
Similarly we have 
\begin{align*}
& \Bigl| \int^1_0 \dfrac{|\kappa_{i,n}|^p}{\mathcal{L}(\gamma_{i,n})} \partial_x \gamma_{i,n} \cdot \partial_x \varphi_1 \, dx \Bigr| 
 \le C \| \psi \|_{L^1(\mathcal{S}^1)} E_p(\gamma_{i,n}) 
 \le C \| \psi \|_{L^1(\mathcal{S}^1)}, \\
& \Bigl| \lambda \int^1_0 \dfrac{\partial_x \gamma_{i,n}}{\mathcal{L}(\gamma_{i,n})} \cdot \partial_x \varphi_1 \, dx \Bigr| 
 \le C \| \psi \|_{L^1(\mathcal{S}^1)}, \\
& \Bigl| \mathcal{L}(\gamma_{i-1,n}) 
   \int^1_0 V_{i,n} \cdot \varphi_1 \, dx \Bigr| 
 \le C \| \psi \|_{L^1(\mathcal{S}^1)} \| V_{i,n} \|_{L^1(\mathcal{S}^1)}. 
\end{align*}
Since 
\begin{align*}
\| \Phi_1(\gamma_{i,n}, \varphi_1) \|_{L^\infty(\mathcal{S}^1)} \le  C \| \psi \|_{L^1(\mathcal{S}^1)}, 
\end{align*}
we also find 
\begin{equation*}
\Bigl| \mathcal{L}(\gamma_{i-1,j}) \int^1_0 V_{i,n} \cdot \Phi_1(\gamma_{i,n}, \eta) \partial_x \gamma_{i,n} \, dx \Bigr| 
 \le C \| \psi \|_{L^1(\mathcal{S}^1)} \| V_{i,n} \|_{L^1(\mathcal{S}^1)}. 
\end{equation*}
Thus we obtain 
\begin{equation} \label{eq:3.38}
\Bigl| \int^1_0 \dfrac{|\kappa_{i,n}|^{p-2}}{\mathcal{L}(\gamma_{i,n})^3} \partial^2_x \gamma_{i,n} \cdot \psi \, dx \Bigr| 
 \le C (1+\| V_{i,n} \|_{L^2(\mathcal{S}^1)}) \| \psi \|_{L^1(\mathcal{S}^1)}
\end{equation}
for all $\psi \in C^\infty(\mathcal{S}^1;\mathbb{R}^2)$. 
By way of a density argument, \eqref{eq:3.38} also holds for all $\psi \in L^1(\mathcal{S}^1)$. 
Then it follows from \eqref{eq:3.38} that 
\begin{equation} \label{eq:3.39}
\| |\kappa_{i,n}|^{p-2} \partial^2_x \gamma_{i,n} \|_{L^\infty(\mathcal{S}^1)} 
 \le C (1+\| V_{i,n} \|_{L^2(\mathcal{S}^1)}). 
\end{equation} 
Since 
\begin{equation}
\label{eq:3.40}
|\partial^2_x \gamma_{i,n}| = \mathcal{L}(\gamma_{i,n}) |\kappa_{i,n}| |\mathcal{R} \partial_x \gamma_{i,n}|
 = \mathcal{L}(\gamma_{i,n})^2 |\kappa_{i,n}|, 
\end{equation}
we reduce \eqref{eq:3.39} into 
\begin{equation}
\label{eq:3.41}
\| \partial^2_x \gamma_{i,n} \|_{L^\infty(\mathcal{S}^1)}^{p-1} 
 \le C (1+\| V_{i,n} \|_{L^2(\mathcal{S}^1)}).
\end{equation}
This together with Lemma \ref{theorem:3.4} implies that 
$$
\int^T_0 \| \partial^2_x \tilde{\gamma}_n \|^{2(p-1)}_{L^\infty(\mathcal{S}^1)} \, dt 
\le C \int^T_0 \bigl[ 1 + \| V_n \|^2_{L^2(\mathcal{S}^1)} \bigr]  \, dt 
\le C(T+1). 
$$

For $\psi \in C^\infty(\mathcal{S}^1; \mathbb{R}^2)$ we define $\varphi_2 : \mathcal{S}^1 \to \mathbb{R}^2$ as in Lemma \ref{theorem:3.7}. 
We take $\varphi_2$ as $\eta$ in \eqref{eq:3.37}. 
First we have 
\begin{align*}
\int^1_0 \dfrac{|\kappa_{i,n}|^{p-2}}{\mathcal{L}(\gamma_{i,n})^3} \partial^2_x \gamma_{i,n} \cdot \partial^2_x \varphi_{2} \, dx 
 = \int^1_0 \dfrac{|\kappa_{i,n}|^{p-2}}{\mathcal{L}(\gamma_{i,n})^3} \partial^2_x \gamma_{i,n} \cdot \partial_x \psi \, dx. 
\end{align*}
By \eqref{eq:3.40} and \eqref{eq:3.41} we have 
\begin{align*}
\Bigl| \int^1_0 \dfrac{|\kappa_{i,n}|^p}{\mathcal{L}(\gamma_{i,n})} \partial_x \gamma_{i,n} \cdot \partial_x \varphi_1 \, dx \Bigr| 
 &\le C \| \partial^2_x \gamma_{i,n} \|^{p-1}_{L^\infty(\mathcal{S}^1)} \| \partial^2_x \gamma_{i,n} \|_{L^2(\mathcal{S}^1)} \| \psi \|_{L^2(\mathcal{S}^1)} \\
 &\le C (1+\| V_{i,n} \|_{L^2(\mathcal{S}^1)}) \| \psi \|_{L^2(\mathcal{S}^1)}. 
\end{align*}
Along the same line as above, we have 
\begin{align*}
& \Bigl| \lambda \int^1_0 \dfrac{\partial_x \gamma_{i,n}}{\mathcal{L}(\gamma_{i,n})} \cdot \partial_x \varphi_2 \, dx \Bigr| 
 \le C \| \psi \|_{L^1(\mathcal{S}^1)}, \\
& \Bigl| \mathcal{L}(\gamma_{i-1,n}) \int^1_0 V_{i,n} \cdot \varphi_2 \, dx \Bigr| 
 \le C \| \psi \|_{L^1(\mathcal{S}^1)} \| V_{i,n} \|_{L^1(\mathcal{S}^1)}. 
\end{align*}
Moreover, by 
\begin{equation*}
\| \Phi_1(\gamma_{i,n}, \varphi_2) \|_{L^\infty(\mathcal{S}^1)} \le  C \| \psi \|_{L^1(\mathcal{S}^1)}, 
\end{equation*}
we find 
\begin{align*}
\Bigl| \mathcal{L}(\gamma_{i-1,j}) \int^1_0 V_{i,n} \cdot \Phi_1(\gamma_{i,n}, \eta) \partial_x \gamma_{i,n} \, dx \Bigr| 
 \le C \| V_{i,n} \|_{L^1(\mathcal{S}^1)}. 
\end{align*}
Thus we see that  
\begin{align*}
\Bigl| \int^1_0 \dfrac{|\kappa_{i,n}|^{p-2}}{\mathcal{L}(\gamma_{i,n})^3} \partial^2_x \gamma_{i,n} \cdot \partial_x \psi \, dx \Bigr| 
 \le C(1 + \| V_{i,n} \|_{L^2(\mathcal{S}^1)} ) \| \psi \|_{L^2(\mathcal{S}^1)}.  
\end{align*}
This together with Riesz's representation theorem implies that 
\begin{equation}
\label{eq:3.42}
\| \partial_x (|\partial^2_x \gamma_{i,n}|^{p-2} \partial^2_x \gamma_{i,n}) \|_{L^2(\mathcal{S}^1)} \le C(1 + \| V_{i,n} \|_{L^2(\mathcal{S}^1)}). 
\end{equation}
Combining \eqref{eq:3.42} with Lemma \ref{theorem:3.4}, we observe that 
$$
\int^T_0 \| \partial_x (|\partial^2_x \tilde{\gamma}_{n}|^{p-2} \partial^2_x \tilde{\gamma}_{n}) \|_{L^2(\mathcal{S}^1)}^{2} \, dt \le C(T + 1). 
$$
Therefore Lemme \ref{theorem:3.8} follows. 
\end{proof}

%%%%%%%%%%%%%%%%%%%%%%%%%%%%%%%%%%%%%%%%%%%%%%%%%%%%
%%%%%%%%%%%%%%%%%%%%%%%%%%%%%%%%%%%%%%%%%%%%%%%%%%%%
\subsection{Convergence} \label{subsection:3.2}

%%%%%%%%%%%%%%%%%%%%%%%%%%%%%%%%%%%%%%%%%%%%%%%%%%%%
\begin{lemma} \label{theorem:3.9}
Let $\gamma_n$ be the piecewise linear interpolations of the family of planar closed curves $\{ \gamma_{i,n}\}$ obtained by Lemma~{\rm \ref{theorem:3.1}}. 
Then there exists a family of planar closed curves $\gamma : \mathcal{S}^1 \times [0,T] \to \mathbb{R}^2$ such that 
\begin{align}
& \gamma_n \rightharpoonup \gamma \quad \text{weakly$^*$ in} \quad L^\infty(0,T;W^{2,p}(\mathcal{S}^1)), \label{eq:3.43} \\
& \gamma_n \rightharpoonup \gamma \quad \text{in} \quad H^1(0,T;L^2(\mathcal{S}^1)), \label{eq:3.44} 
\end{align}
up to a subsequence. 
\end{lemma}
%%%%%%%%%%%%%%%%%%%%%%%%%%%%%%%%%%%%%%%%%%%%%%%%%%%%
\begin{proof}
By Lemma \ref{theorem:3.4} we have \eqref{eq:3.43} and \eqref{eq:3.44} along the same argument as in \cite[Theorem 4.1]{OY}. 
\end{proof}

Similarly to the proof of \cite[Theorem 4.2]{NO}, we have: 
%%%%%%%%%%%%%%%%%%%%%%%%%%%%%%%%%%%%%%%%%%%%%%%%%%%%%%%%%%%%
\begin{lemma} \label{theorem:3.10}
Let $\gamma_n$ be the piecewise linear interpolations of the family of planar closed curves $\{ \gamma_{i,n}\}$ obtained by Lemma~{\rm \ref{theorem:3.1}}. 
Then 
\begin{equation}
\gamma_n \to \gamma \quad \text{in} \quad C^{0,\beta}([0,T]; C^{1,\alpha}(\mathcal{S}^1)) \label{eq:3.45}
\end{equation}
with $0 < \alpha < 1-1/p$ and $\beta=\frac{(1-\alpha)p-1}{8(p-1)}$, 
where $\gamma$ is the limit obtained by Lemma~\ref{theorem:3.9}.
\end{lemma}
%%%%%%%%%%%%%%%%%%%%%%%%%%%%%%%%%%%%%%%%%%%%%%%%%%%%%%%%%%%%
\begin{proof}
Fix $0 \le t_1 < t_2 \le T$ arbitrarily. 
Since $\partial_t \gamma_n(x,t) = V_n(x,t)$ for $x \in \mathcal{S}^1$ and a.e. $0 < t < T$, we deduce from Definition \ref{theorem:3.2} that 
\begin{align*}
|\gamma_n(x,t_1) - \gamma_n(x,t_2)| \le \int^{t_2}_{t_1} |V_n(x,t)| \, dt  
 \le ( t_2 - t_1 )^{\frac{1}{2}} \Bigl( \int^{t_2}_{t_1} | V_n(x,t)|^2 \, dt \Bigr)^{\frac{1}{2}}. 
\end{align*} 
Thus, taking the squared integral of the both side with respect to $x$ on $\mathcal{S}^1$, we observe from Lemma~\ref{theorem:3.4} and Fubini's theorem that 
\begin{equation}
\label{eq:3.46}
\| \gamma_n(\cdot, t_1) - \gamma_n(\cdot,t_2) \|_{L^2(\mathcal{S}^1)} 
% &\le ( t_2 - t_1 )^{\frac{1}{2}} \Bigl( \int^{t_2}_{t_1} \!\!\! \int^1_0 | V_n(x,t)|^2 \, dx dt \Bigr)^{\frac{1}{2}} \\
 \le C_*^{\frac{1}{2}} ( t_2 - t_1 )^{\frac{1}{2}}. 
\end{equation}

Let $\delta_n(x):= \gamma_n(x,t_1) - \gamma_n(x,t_2)$. 
By Proposition \ref{theorem:2.1} we find $A>0$ such that  
\begin{equation}
\label{eq:3.47}
\| (\delta_n)_x \|_{L^\infty(\mathcal{S}^1)} 
 \le A \bigl( \| (\delta_n)_{xx} \|^{3/4}_{L^2(\mathcal{S}^1)} \| \delta_n \|^{1/4}_{L^2(\mathcal{S}^1)} + \| \delta_n \|_{L^2(\mathcal{S}^1)} \bigr). 
% \le A \bigl( \| (\delta_n)_{xx} \|^{3/4}_{L^p(\mathcal{S}^1)} \| \delta_n \|^{1/4}_{L^2(\mathcal{S}^1)} + \| \delta_n \|_{L^2(\mathcal{S}^1)} \bigr).
\end{equation}
Since $| \partial_x^2 \gamma_n(t) | \le 2 \sup_{0 \le i \le n}|\partial^2_x \gamma_{i,n}|$ for $t \in [0, T]$, 
we deduce from \eqref{eq:3.7}, \eqref{eq:3.46} and \eqref{eq:3.47} that 
\begin{align}
\label{eq:3.48}
\begin{split}
\| (\delta_n)_x \|_{L^\infty(\mathcal{S}^1)} 
 \le A \Bigl[ \dfrac{2 p^{\frac{1}{p}} (\mathcal{E}_p(\gamma_0) + 1)^{\frac{p+1}{p}}}{\lambda} \Bigr]^{\frac{3}{4}} C_*^{\frac{1}{8}}(t_2-t_1)^{\frac{1}{8}} 
    + A C_*^{\frac{1}{2}}(t_2 - t_1)^{\frac{1}{2}}.  
\end{split}
\end{align}
%This together with the fundamental theorem of calculus implies that 
Along the same line we have 
\begin{align}
\label{eq:3.49}
\begin{split}
\| \delta_n \|_{L^\infty(\mathcal{S}^1)} 
&\le A \bigl( \| (\delta_n)_{xx} \|^{1/4}_{L^2(\mathcal{S}^1)} \| \delta_n \|^{3/4}_{L^2(\mathcal{S}^1)} + \| \delta_n \|_{L^2(\mathcal{S}^1)} \bigr)\\
&\le A \Bigl[ \dfrac{2 p^{\frac{1}{p}} (\mathcal{E}_p(\gamma_0) + 1)^{\frac{p+1}{p}}}{\lambda} \Bigr]^{\frac{1}{4}} C_*^{\frac{3}{8}} (t_2-t_1)^{\frac{3}{8}} + A C_*^{\frac{1}{2}} (t_2 - t_1)^{\frac{1}{2}}. 
% &\le |\delta_n(0)| + \int^1_0 |(\delta_n)_x| \, dx 
%  \le \| (\delta_n)_x \|_{L^\infty(\mathcal{S}^1)} \\
% &\le A \Bigl[ \dfrac{2 p^{\frac{1}{p}} (\mathcal{E}_p(\gamma_0) + 1)^{\frac{p+1}{p}}}{\lambda} \Bigr]^{\frac{3}{4}} C_*^{\frac{1}{8}} (t_2-t_1)^{\frac{1}{8}} 
%      + A C_*^{\frac{1}{2}} (t_2 - t_1)^{\frac{1}{2}}. 
\end{split}
\end{align}
We observe that \eqref{eq:3.48} and \eqref{eq:3.49} that it suffices to estimate the H\"older semi-norm of $(\delta_n)_x$. 
Fix $0 < \alpha < 1-1/p$ arbitrarily. 
Adopting Morrey's inequality, we obtain 
\begin{align*}
&\sup_{x_1, x_2 \in \mathcal{S}^1} \dfrac{|(\delta_n)_x(x_1) - (\delta_n)_x(x_2)|}{|x_1 - x_2|^{\alpha}} \\
& \quad = \sup_{x_1, x_2 \in \mathcal{S}^1}\Bigl( \dfrac{|(\delta_n)_x(x_1) - (\delta_n)_x(x_2)|}{|x_1 - x_2|^{\frac{p-1}{p}}} \Bigr)^{\frac{\alpha p}{p-1}} 
     |(\delta_n)_x(x_1) - (\delta_n)_x(x_2)|^{1- \frac{\alpha p}{p-1}} \\
& \quad \le C \| (\delta_n)_x \|_{C^{\frac{p-1}{p}}(\mathcal{S}^1)}^{\frac{\alpha p}{p-1}} \| (\delta_n)_x \|_{L^\infty(\mathcal{S}^1)}^{\frac{1- \alpha p}{p-1}} 
  \quad \le C(T) \| \delta_n \|_{W^{2,p}(\mathcal{S}^1)}^{\alpha p/(p-1)} (t_2 - t_1)^{\frac{(1-\alpha)p-1}{8(p-1)}}. 
\end{align*}
Thus Lemma \ref{theorem:3.10} follows from the Arzel\`a--Ascoli theorem (see e.g. \cite[Proposition~3.3.1]{AGS}). 
\end{proof}

%%%%%%%%%%%%%%%%%%%%%%%%%%%%%%%%%%%%%%%%%%%%%%%%%%%%
\begin{lemma} \label{theorem:3.11}
Let $\tilde{\gamma}_n$ be the piecewise linear interpolations of the family of planar closed curves $\{ \gamma_{i,n}\}$ obtained by Lemma~{\rm \ref{theorem:3.1}}. 
Then  
\begin{align}
& \tilde{\gamma}_n \rightharpoonup \gamma \quad \text{weakly in} \quad L^p(0,T;W^{2,p}(\mathcal{S}^1)), \label{eq:3.50} \\
& \tilde{\gamma}_n \to \gamma \quad \text{in} \quad L^{\infty}(0,T; C^{1,\alpha}(\mathcal{S}^1)), \label{eq:3.51} \\
& \tilde{\Gamma}_n \to \gamma \quad \text{in} \quad L^{\infty}(0,T; C^{1,\alpha}(\mathcal{S}^1)), \label{eq:3.52}
\end{align}
up to a subsequence, where $0<\alpha<1-1/p$ and $\gamma$ is the limit obtained in Lemma~{\rm \ref{theorem:3.9}}.  
\end{lemma}
%%%%%%%%%%%%%%%%%%%%%%%%%%%%%%%%%%%%%%%%%%%%%%%%%%%%
\begin{proof}

Fix $(x,t) \in \mathcal{S}^1 \times ((i-1)\tau_n, i \tau_n]$ arbitrarily. 
 We deduce from Lemma \ref{theorem:3.10} that 
\begin{align*}
| \partial^j_x \tilde{\gamma}_n(x, t) - \partial^j_x \gamma_n(x,t)| 
 &= | \partial^j_x \gamma_{i,n}(x) - \partial^j_x \gamma_n(x,t)| 
   = | \partial^j_x \gamma_n(x, i \tau_n) - \partial^j_x \gamma_n(x,t) | \\
 &\le C | t - i \tau_n|^\beta 
   \le C \tau_n^\beta, 
\end{align*}
for $j=0, 1$, where $\beta$ is given constant in Lemma \ref{theorem:3.10}. 
Thus we have 
\begin{align*}
\| \tilde{\gamma}_n - \gamma_n \|_{L^\infty(0,T; C^{1, \alpha}(\mathcal{S}^1))} \to 0 \quad \text{as} \quad n \to \infty. 
\end{align*}
This together with Lemma \ref{theorem:3.10} implies \eqref{eq:3.51}.  
Similarly we obtain \eqref{eq:3.52}. 

We turn to \eqref{eq:3.50}. 
Thanks to \eqref{eq:3.51}, we have 
\begin{align*}
\Bigl| \int^T_0 \!\!\! \int^1_0 (\partial^2_x \tilde{\gamma}_n - \partial^2_x \gamma) \cdot \eta \, dx dt \Bigr|
 &= \Bigl| \int^T_0 \!\!\! \int^1_0 (\partial_x \tilde{\gamma}_n - \partial_x \gamma) \cdot \partial_x \eta \, dx dt \Bigr| \\
 &\le \| \tilde{\gamma}_n - \gamma \|_{L^\infty(0,T;C^{1, \alpha}(\mathcal{S}^1))} \| \partial_x \eta \|_{L^1(0,T;L^1(\mathcal{S}^1))} \\
 & \to 0 \quad \text{as} \quad n \to \infty  
\end{align*}
for all $\eta \in C^\infty(\mathcal{S}^1 \times (0,T))$. This together with \eqref{eq:3.51} implies \eqref{eq:3.50}. 
Therefore Lemma \ref{theorem:3.11} follows. 
\end{proof}

%%%%%%%%%%%%%%%%%%%%%%%%%%%%%%%%%%%%%%%%%%%%%%%%%%%%%%%%%%%%%%%%%%
\begin{lemma} \label{theorem:3.12}
Let $\tilde{\gamma}_{n}$ be the piecewise constant interpolation of $\{ \gamma_{i,n} \}$. 
Then 
\begin{align*}
\int^{T}_{0} \!\!\! \int^1_0 |\partial^{2}_{x} \tilde{\gamma}_{n}|^{p-2} \partial^{2}_{x} \tilde{\gamma}_{n} \cdot \partial^{2}_{x} \eta \, dx dt 
 \to \int^{T}_{0} \!\!\! \int^1_0 |\partial^{2}_{x} \gamma|^{p-2} \partial^{2}_{x} \gamma \cdot \partial^{2}_{x} \eta \, dx dt 
 \quad \text{as} \quad n \to \infty 
\end{align*}
for $\eta \in L^{p}(0,T; W^{2,p}(\mathcal{S}^1))$, where $\gamma$ denotes the limit obtained by Lemma~{\rm \ref{theorem:3.9}}. 
\end{lemma}
%%%%%%%%%%%%%%%%%%%%%%%%%%%%%%%%%%%%%%%%%%%%%%%%%%%%%%%%%%%%%%%%%%%
\begin{proof}
By Lemma \ref{theorem:3.8} we find $w \in L^{2}(0,T; H^{1}(\mathcal{S}^1))$ such that  
\begin{align} \label{eq:3.53}
|\partial^{2}_{x} \tilde{\gamma}_{n}|^{p-2} \partial^{2}_{x} \tilde{\gamma}_{n} \rightharpoonup w \quad \text{weakly in} \quad 
L^{2}(0,T; H^{1}(\mathcal{S}^1)) \quad \text{as} \quad n \to \infty 
\end{align}
up to a subsequence. 
Here we set 
\begin{align*}
F(\psi) := \dfrac{1}{p \mathcal{L}(\gamma)} \int^{T}_{0} \!\!\! \int^1_0 | \partial^{2}_{x} \psi |^{p} \, dx dt.  
\end{align*}
From now on, fix $\psi \in L^{p}(0,T; W^{2,p}(\mathcal{S}^1))$ arbitrarily. 
From the convexity of $F(\cdot)$, we observe that 
\begin{align} \label{eq:3.54}
F(\psi) - F(\tilde{\gamma}_{n})  
 \ge \dfrac{1}{\mathcal{L}(\gamma)}\int^{T}_{0} \!\!\! \int^1_0 | \partial^{2}_{x} \tilde{\gamma}_{n}|^{p-2} \partial^{2}_{x} \tilde{\gamma}_{n} \cdot \partial^{2}_{x}(\psi - \tilde{\gamma}_{n}) \, dx dt.  
\end{align}
We claim that 
\begin{align} \label{eq:3.55}
F(\psi) - F(\gamma) \ge \dfrac{1}{\mathcal{L}(\gamma)} \int^{T}_{0}\!\!\! \int^1_0 w \cdot \partial^{2}_{x}(\psi - \gamma) \, dx dt. 
\end{align} 
To begin with, it follows from Lemma \ref{theorem:3.11} that 
\begin{equation}
\label{eq:3.56}
\liminf_{n \to \infty} F(\tilde{\gamma}_n) \ge F(\gamma). 
\end{equation}
Integrating by part, we reduce the right hand side of \eqref{eq:3.54} into 
\begin{align*}
I_1 := - \dfrac{1}{\mathcal{L}(\gamma)} \int^{T}_{0} \!\!\! \int^1_0 
                                           \partial_{x}(|\partial^{2}_{x}\tilde{\gamma}_{n}|^{p-2} \partial^{2}_{x}\tilde{\gamma}_{n}) \cdot \partial_{x}(\psi-\tilde{\gamma}_{n}) \, dx dt. 
\end{align*}
By \eqref{eq:3.53} and Lemma \ref{theorem:3.11}, extracting a subsequence, we have 
\begin{align*}
I_1 \to - \dfrac{1}{\mathcal{L}(\gamma)} \int^{T}_{0} \!\!\! \int^1_0 \partial_{x} w \cdot \partial_{x}(\psi- \gamma) \, dx dt
 = \dfrac{1}{\mathcal{L}(\gamma)} \int^{T}_{0} \!\!\! \int^1_0 w \cdot \partial^{2}_{x}(\psi- \gamma) \, dx dt 
\end{align*}
as $n \to \infty$. 
This together with \eqref{eq:3.54} and \eqref{eq:3.56} implies \eqref{eq:3.55}. 

Setting $\psi= \gamma + \varepsilon \eta$ in \eqref{eq:3.55} for $\eta \in L^p(0,T;W^{2,p}(\mathcal{S}^1))$, we obtain 
\begin{align} \label{eq:3.57}
\dfrac{F(\gamma + \varepsilon \eta) - F(\gamma)}{\varepsilon} \ge \dfrac{1}{\mathcal{L}(\gamma)} \int^{T}_{0} \!\!\! \int^1_0 w \cdot \partial^{2}_{x} \eta \, dx dt.
\end{align}
On the other hand, putting $\psi= \gamma - \varepsilon \eta$ in \eqref{eq:3.55} for $\eta \in L^p(0,T;W^{2,p}(\mathcal{S}^1))$, we get 
\begin{align} \label{eq:3.58}
\dfrac{F(\gamma) - F(\gamma - \varepsilon \eta)}{\varepsilon} \le \dfrac{1}{\mathcal{L}(\gamma)} \int^{T}_{0} \!\!\! \int^1_0 w \cdot \partial^{2}_{x} \eta \, dx dt.
\end{align}
Plugging \eqref{eq:3.58} into \eqref{eq:3.57} and letting $\varepsilon \downarrow 0$, we find  
\begin{align*}
\int^{T}_{0} \!\!\! \int^1_0 | \partial^{2}_{x} \gamma |^{p-2} \partial^{2}_{x} \gamma \cdot \partial^{2}_{x} \eta \, dx dt 
 = \int^{T}_{0} \!\!\! \int^1_0 w \cdot \partial^{2}_{x} \eta \, dx dt 
\end{align*}
for all $\eta \in L^p(0,T;W^{2,p}(\mathcal{S}^1))$. 
Thus Lemma \ref{theorem:3.12} follows. 
\end{proof}

%%%%%%%%%%%%%%%%%%%%%%%%%%%%%%%%%%%%%
\begin{lemma} \label{theorem:3.13}
Let $\tilde{\gamma}_{n}$ be the piecewise constant interpolation of $\{ \gamma_{i,n} \}$. 
Then we have 
\begin{align*}
\int^{T}_{0} \!\!\! \int^1_0 | \partial^2_x \tilde{\gamma}_n |^{p} \partial_x \tilde{\gamma}_n \cdot \partial_x \eta \, dx dt 
 \to \int^{T}_{0} \!\!\! \int^1_0 | \partial^2_x \gamma |^{p} \partial_x \gamma \cdot \partial_x \eta \, dx dt 
 \quad \text{as} \quad n \to \infty 
\end{align*}
for all $\eta \in L^{p}(0,T; W^{2,p}(\mathcal{S}^1))$, where $\gamma$ denotes the planar curve obtained by Lemma~{\rm \ref{theorem:3.9}}. 
\end{lemma}
%%%%%%%%%%%%%%%%%%%%%%%%%%%%%%%%%%%%%
\begin{proof}
To begin with, we have 
\begin{align*}
& \Bigl| \int^{T}_{0} \!\!\! \int^1_0 | \partial^2_x \tilde{\gamma}_n |^{p} \partial_x \tilde{\gamma}_n \cdot \partial_x \eta \, dx dt 
           - \int^{T}_{0} \!\!\! \int^1_0 | \partial^2_x \gamma |^{p} \partial_x \gamma \cdot \partial_x \eta \, dx dt \Bigr| \\
& \quad \le \Bigl| \int^{T}_{0} \!\!\! \int^1_0 ( | \partial^2_x \tilde{\gamma}_n |^{p} - | \partial^2_x \gamma|^{p} ) \partial_x \gamma \cdot \partial_x \eta \, dx dt \Bigr| \\
& \qquad + \Bigl| \int^{T}_{0} \!\!\! \int^1_0 | \partial^2_x \tilde{\gamma}_n |^{p} ( \partial_x \tilde{\gamma}_n - \partial_x \gamma ) \cdot \partial_x \eta \, dx dt \Bigr| 
=: I_1 + I_2. 
\end{align*}
Setting $\tilde{w}_n:= | \partial^2_x \tilde{\gamma}_n |^{p-2} \partial^2_x \tilde{\gamma}_n$ and $w:= | \partial^2_x \gamma |^{p-2} \partial^2_x \gamma$, 
we estimate $I_{1}$ as follows: 
\begin{align*}
I_1 &=  \Bigl| \int^{T}_{0} \!\!\! \int^1_0 ( \tilde{w}_{n} \cdot \partial^2_x \tilde{\gamma}_n - w \cdot \partial^2_x \gamma ) \partial_x \gamma \cdot \partial_x \eta \, dx dt \Bigr| \\
 &\le \Bigl| \int^{T}_{0} \!\!\! \int^1_0 \{ (\tilde{w}_{n} - w) \cdot \partial^2_x \gamma \} \partial_x \gamma \cdot \partial_x \eta \, dx dt \Bigr| \\
 & \qquad + \Bigl| \int^{T}_{0} \!\!\! \int^1_0 \tilde{w}_{n} \cdot  (\partial^2_x \tilde{\gamma}_n - \partial^2_x \gamma) \partial_x \gamma \cdot \partial_x \eta \, dx dt \Bigr| 
  =: I_{11} + I_{12}. 
\end{align*}
We deduce from Proposition \ref{theorem:2.1} that 
\begin{align*}
& \int^{T}_{0} \| (\partial_x \gamma \cdot \partial_x \eta) \partial^2_x \gamma \|_{L^2(\mathcal{S}^1)}^2\, dt \\
& \quad \le C \int^{T}_{0} \| \partial^2_x \gamma \|^2_{L^2(\mathcal{S}^1)} \| \partial_x \eta \|^2_{L^\infty(\mathcal{S}^1)} \, dt \\
& \quad \le C \int^T_0 \| \partial^2_x \gamma \|_{L^{p}(\mathcal{S}^1)}^{2} 
                \bigl[ \|\partial^2_x \eta \|^{\theta}_{L^p(\mathcal{S}^1)} \| \eta \|^{1-\theta}_{L^p(\mathcal{S}^1)} + \| \eta \|_{L^p(\mathcal{S}^1)} \bigr]^2 \, dt \\
& \quad \le C \int^T_0 \| \eta \|^{2}_{W^{2,p}(\mathcal{S}^1)} \, dt 
  \le C \| \eta \|_{L^p(0,T;W^{2,p}(\mathcal{S}^1))}^2. 
\end{align*}
Hence this together with \eqref{eq:3.53} implies that  
\begin{align*}
I_{11} \to 0 \quad \text{as} \quad n \to \infty  
\end{align*}
up to a subsequence. 
Similarly, integrating by part, we observe from \eqref{eq:3.39} and Lemma~\ref{theorem:3.8} that  
\begin{align*}
I_{12} &\le \Bigl| \int^{T}_{0} \!\!\! \int^1_0 \partial_x \tilde{w}_{n} \cdot (\partial_x \tilde{\gamma}_n - \partial_x \gamma) \partial_x \gamma \cdot \partial_x \eta \, dx dt \Bigr| \\
           & \qquad + \Bigl| \int^{T}_{0} \!\!\! \int^1_0 \tilde{w}_{n} \cdot (\partial_x \tilde{\gamma}_n - \partial_x \gamma) \partial_x (\partial_x \gamma \cdot \partial_x \eta) \, dx dt \Bigr| \\
 &\le C \| \tilde{\gamma}_n - \gamma \|_{L^\infty(0,T; C^{1,\alpha}(\mathcal{S}^1))} \| \partial_x \tilde{w}_{n} \|_{L^2(0,T;L^2(\mathcal{S}^1))} \| \partial_x \eta \|_{L^2(0,T; L^2(\mathcal{S}^1))} \\
 & \quad + C \| \tilde{\gamma}_n - \gamma \|_{L^\infty(0,T; C^{1,\alpha}(\mathcal{S}^1))} \| \tilde{w}_n \|_{L^2(0,T;L^\infty(\mathcal{S}^1))} 
             \| |\partial_x(\partial_x \gamma \cdot \partial_x \eta)\|_{L^2(0,T;L^1(\mathcal{S}^1))}\\
 &\le C(T+1) \| \tilde{\gamma}_n - \gamma \|_{L^\infty(0,T; C^{1,\alpha}(\mathcal{S}^1))} \| \eta \|_{L^p(0,T;W^{2,p}(\mathcal{S}^1))}.   
\end{align*}
This together with Lemma \ref{theorem:3.11} implies that 
$$
I_{12} \to 0 \quad \text{as} \quad n \to \infty
$$
up to a subsequence. 
We turn to the estimate on $I_2$. Since  
\begin{align*}
I_{2} \le \| \tilde{\gamma}_n - \gamma \|_{L^\infty(0,T; C^{1,\alpha}(\mathcal{S}^1))} \int^T_0 \!\!\! \int^1_0 | \partial^2_x \tilde{\gamma}_n |^{p} | \partial_x \eta | \, dxdt,  
\end{align*}
it suffices to estimate the integral in the right-hand side. 
Indeed, according to Lemma~\ref{theorem:3.8}, we have 
\begin{align*}
& \int^T_0 \!\!\! \int^1_0 | \partial^2_x \tilde{\gamma}_n |^{p} | \partial_x \eta | \, dxdt 
 = \int^T_0 \!\!\! \int^1_0 | \partial^2_x \tilde{\gamma}_n |^{p-1} \cdot | \partial^2_x \tilde{\gamma}_n | | \partial_x \eta | \, dxdt \\
& \quad \le \Bigl( \int^T_0 \| \partial^2_x \tilde{\gamma}_n \|_{L^{\infty}(\mathcal{S}^1)}^{2(p-1)} \, dt \Bigr)^{\tfrac{1}{2}}  
        \Bigl( \int^T_0 \| \partial^2_x \tilde{\gamma}_n \|^{2}_{L^{2}(\mathcal{S}^1)} \| \partial_x \eta \|^{2}_{L^{2}(\mathcal{S}^1)} \, dt \Bigr)^{\tfrac{1}{2}} \\
& \quad \le C( T^{1/2} + 1) \| \eta \|_{L^{p}(0,T;W^{2,p}(\mathcal{S}^1))}^{1/2}. 
\end{align*}
Thus we deduce from Lemma \ref{theorem:3.11} that $I_{2} \to 0$ as $n \to \infty$. 
Therefore Lemma~\ref{theorem:3.13} follows. 
\end{proof}

%%%%%%%%%%%%%%%%%%%%%%%%%%%%%%%%%%%%%%%%%%%%%%%%%%%%
%%%%%%%%%%%%%%%%%%%%%%%%%%%%%%%%%%%%%%%%%%%%%%%%%%%%
\section{Proof of Theorem \ref{theorem:1.2}} \label{section:4}
Let $\gamma_0 \in W^{2,p}(\mathcal{S}^1;\mathbb{R}^2)$ satisfy \eqref{eq:1.1}. 
In this section, we fix such $\gamma_0$ arbitrarily, and denote the admissible set $\mathcal{AC}_{\gamma_0}$ by $\mathcal{AC}$ for short. 
%%%%%%%%%%%%%%%%%%%%%%%%%%%%%%%%%%%%%%%%%%%%%%%%%%
\begin{lemma} \label{theorem:4.1}
Let $\gamma : \mathcal{S}^1 \times [0, T] \to \mathbb{R}^2$ be a family of closed curves obtained by Lemma \ref{theorem:3.9}. 
Then it holds that 
\begin{equation}
\label{eq:4.1}
\begin{aligned}
& \int^T_0 \!\!\!\! \int^1_0 \Bigl[ \dfrac{|\partial^2_x \gamma|^{p-2} \partial^2_x \gamma}{\mathcal{L}(\gamma)^{2p-1}} \cdot \partial^2_x \eta 
    - \dfrac{2p-1}{p} \dfrac{|\partial^2_x \gamma|^{p} \partial_x \gamma}{\mathcal{L}(\gamma)^{2p+1}} \cdot \partial_x \eta \\
& \qquad \quad 
+ \dfrac{\lambda}{\mathcal{L}(\gamma)} \partial_x \gamma \cdot \partial_x \eta 
+ \mathcal{L}(\gamma) \partial_t \gamma \cdot \eta  
+ \mathcal{L}(\gamma) \partial_t \gamma \cdot \Phi_1(\gamma,\eta) \partial_x \gamma \Bigr] \, dxdt=0   
\end{aligned}
\end{equation}
for all $\eta \in L^\infty(0,T; W^{2,p}(\mathcal{S}^1))$. 
Moreover, 
\begin{equation}
\label{eq:4.2}
\gamma(\cdot,t) \in \mathcal{AC} \quad \text{for a.e.} \quad t \in (0, T).
\end{equation} 
\end{lemma}
%%%%%%%%%%%%%%%%%%%%%%%%%%%%%%%%%%%%%%%%%%%%%%%%%%%%%%
\begin{proof}
By the definition of $\gamma_n$, $\tilde{\gamma}_n$ and $V_n$ we observe from Lemma \ref{theorem:3.6} that 
\begin{align}
\label{eq:4.3}
\begin{split}
& \int^T_0 \!\!\! \int^1_0 \Bigl[ \dfrac{|\tilde{\kappa}_n|^{p-2}\partial^2_x \tilde{\gamma}_n}{\mathcal{L}(\tilde{\gamma}_n)^3} \cdot \partial^2_x \eta  
                 -\dfrac{2p-1}{p} \dfrac{|\tilde{\kappa}_n|^p \partial_x \tilde{\gamma}_n}{\mathcal{L}(\tilde{\gamma}_n)} \cdot \partial_x \eta 
                 + \dfrac{\lambda \partial_x \tilde{\gamma}_n}{\mathcal{L}(\tilde{\gamma}_n)} \cdot \partial_x \eta \Bigr] \, dxdt \\
&              +\int^T_0 \!\!\! \int^1_0 \mathcal{L}(\tilde{\Gamma}_n) V_n \cdot \eta \, dx dt
                 + \int^T_0 \!\!\! \int^1_0 \mathcal{L}(\tilde{\Gamma}_n) V_n \cdot \Phi_1(\tilde{\gamma}_n,\eta) \partial_x \tilde{\gamma}_n \, dxdt =0 
\end{split}
\end{align}
for all $\eta \in L^\infty(0,T;W^{2,p}(\mathcal{S}^1))$. 
Since $\partial^2_x \tilde{\gamma}_n= \mathcal{L}(\tilde{\gamma}_n) \tilde{\kappa}_n \mathcal{R} \partial_x \tilde{\gamma}_n$ 
and $|\partial_x \tilde{\gamma}_n|= \mathcal{L}(\tilde{\gamma}_n)$, we reduce~\eqref{eq:4.3} into 
\begin{align}
\label{eq:4.4}
\begin{split}
& \int^T_0 \!\!\!\! \int^1_0 \Bigl[ \dfrac{|\partial^2_x \tilde{\gamma}_n|^{p-2} \partial^2_x \tilde{\gamma}_n}{\mathcal{L}(\tilde{\gamma}_n)^{2p-1}} \cdot \partial^2_x \eta 
    - \dfrac{2p-1}{p} \dfrac{|\partial^2_x \tilde{\gamma}_n|^{p} \partial_x \tilde{\gamma}_n}{\mathcal{L}(\tilde{\gamma}_n)^{2p+1}} \cdot \partial_x \eta 
    + \dfrac{\lambda}{\mathcal{L}(\tilde{\gamma}_n)} \partial_x \tilde{\gamma}_n \cdot \partial_x \eta \\
& \qquad \quad 
+ \mathcal{L}(\tilde{\Gamma}_n) V_n \cdot \eta  
+ \mathcal{L}(\tilde{\Gamma}_n) V_n \cdot \Phi_1(\tilde{\gamma}_n,\eta) \partial_x \tilde{\gamma}_n \Bigr] \, dxdt=0.  
\end{split}
\end{align}
 
By Lemmata \ref{theorem:3.11}, \ref{theorem:3.12} and~\ref{theorem:3.13} we have 
\begin{equation}
\label{eq:4.5}
\int^T_0 \!\!\!\! \int^1_0 \dfrac{|\partial^2_x \tilde{\gamma}_n|^{p-2} \partial^2_x \tilde{\gamma}_n}{\mathcal{L}(\tilde{\gamma}_n)^{2p-1}} \cdot \partial^2_x \eta \, dx dt 
\to 
\int^T_0 \!\!\!\! \int^1_0 \dfrac{|\partial^2_x \gamma|^{p-2} \partial^2_x \gamma}{\mathcal{L}(\gamma)^{2p-1}} \cdot \partial^2_x \eta \, dx dt  
\end{equation}
for all $\eta \in L^\infty(0,T;W^{2,p}(\mathcal{S}^1))$ as $n \to \infty$ up to a subsequence. 
Similarly we also obtain 
\begin{equation}
\label{eq:4.6}
\begin{split}
&\int^T_0 \!\!\! \int^1_0 \dfrac{|\partial^2_x \tilde{\gamma}_n|^{p} \partial_x \tilde{\gamma}_n}{\mathcal{L}(\tilde{\gamma}_n)^{2p+1}} \cdot \partial_x \eta \, dx dt 
\to 
\int^T_0 \!\!\! \int^1_0 \dfrac{|\partial^2_x \gamma|^{p} \partial_x \gamma}{\mathcal{L}(\gamma)^{2p+1}} \cdot \partial_x \eta \, dx dt, \\
&\int^T_0 \!\!\! \int^1_0 \dfrac{\lambda}{\mathcal{L}(\tilde{\gamma}_n)} \partial_x \tilde{\gamma}_n \cdot \partial_x \eta \, dxdt 
\to 
\int^T_0 \!\!\! \int^1_0 \dfrac{\lambda}{\mathcal{L}(\gamma)} \partial_x \gamma \cdot \partial_x \eta \, dxdt, 
\end{split}
\end{equation}
for all $\eta \in L^\infty(0,T;W^{2,p}(\mathcal{S}^1))$ as $n \to \infty$ up to a subsequence. 
Recalling that $\partial_t \gamma_n = V_n$, we infer from Lemmata \ref{theorem:3.9} and \ref{theorem:3.11} that  
\begin{equation}
\label{eq:4.7}
\begin{split}
&\int^T_0 \!\!\! \int^1_0 \mathcal{L}(\tilde{\Gamma}_n) V_n \cdot \eta \, dxdt 
  \to \int^T_0 \!\!\! \int^1_0 \mathcal{L}(\gamma) \partial_t \gamma \cdot \eta \, dxdt, \\
& \int^T_0 \!\!\! \int^1_0 \mathcal{L}(\tilde{\Gamma}_n) V_n \cdot \Phi_1(\tilde{\gamma}_n,\eta) \partial_x \tilde{\gamma}_n \, dx dt 
   \to \int^T_0 \!\!\! \int^1_0 \mathcal{L}(\gamma) \partial_t \gamma \cdot \Phi_1(\gamma,\eta) \partial_x \gamma \, dx dt,                  
\end{split}
\end{equation}
as $n \to \infty$ up to a subsequence.  
Extracting a subsequence and letting $n \to \infty$ in \eqref{eq:4.4}, 
we observe from \eqref{eq:4.5}, \eqref{eq:4.6} and \eqref{eq:4.7} that $\gamma$ satisfies the desired weak form. 

Since 
\begin{equation*}
|\partial_x \tilde{\gamma}_n(x,t)|=|\partial_x \gamma_{i,n}(x)|=\mathcal{L}(\gamma_{i,n}) =\mathcal{L}(\tilde{\gamma}_n(t))
\end{equation*}
for all $(x,t) \in \mathcal{S}^1 \times ((i-1) \tau_n, i \tau_n]$, we deduce from Lemma \ref{theorem:3.11} that 
\begin{align*}
|\partial_x \gamma(x,t)|= \mathcal{L}(\gamma(t))
\end{align*}
for all $x \in \mathcal{S}^1$ and a.e. $t \in (0,T)$. 
Therefore Lemma \ref{theorem:4.1} follows. 
\end{proof}

%%%%%%%%%%%%%%%%%%%%%%%%%%%%%%%%%%%%%%%%%
\begin{lemma} \label{theorem:4.2}
Let $\gamma : \mathcal{S}^1 \times [0, T] \to \mathbb{R}^2$ be a family of closed curves obtained by Lemma \ref{theorem:3.9}. 
Then 
\begin{equation}
\label{eq:4.8}
\gamma(\cdot,t) \in W^{2,p}(\mathcal{S}^1; \mathbb{R}^2) \quad \text{for all} \quad t\in [0,T].
\end{equation}
\end{lemma}
%%%%%%%%%%%%%%%%%%%%%%%%%%%%%%%%%%%%%%%%%%
\begin{proof}
Fix $t\in [0,T]$ arbitrarily. By Lemma \ref{theorem:3.4} we have  
$$
\| \gamma_n(t) \|_{W^{2,p}(\mathcal{S}^1)} \le 2 \sup_{0 \le i \le n} \| \gamma_{i,n}\|_{W^{2,p}(\mathcal{S}^1)} \le 2 C^*, 
$$
where the constant $C^*>0$ depends only on $p$, $\lambda$, $\gamma_0$ and $T$ (more precisely, see~\eqref{eq:3.17}, \eqref{eq:3.18} and \eqref{eq:3.19}).  
Extracting a subsequence, we find $\Gamma \in W^{2,p}(\mathcal{S}^1; \mathbb{R}^2)$ such that 
$$
\gamma_n(\cdot, t) \rightharpoonup \Gamma \quad \text{weakly in} \quad W^{2,p}(\mathcal{S}^1; \mathbb{R}^2). 
$$
By way of the Rellich--Kondrachov compactness theorem, we also see that $\gamma_n(\cdot, t)$ converges to $\Gamma$ in $C^{1, \theta}(\mathcal{S}^1)$ with $\theta \in (0, 1-1/p)$.  
This together with \eqref{eq:3.45} implies that $\Gamma(\cdot)=\gamma(\cdot, t)$ in $C^{1, \theta}(\mathcal{S}^1)$. 
Since now $\Gamma \in W^{2,p}(\mathcal{S}^1;\mathbb{R}^2)$, we obtain 
\begin{equation}
\label{eq:4.9}
\int^1_0 \gamma(x, t) \cdot \partial_x \varphi(x) \, dx
 = \int^1_0 \Gamma(x) \cdot \partial_x \varphi(x) \, dx = - \int^1_0 \partial_x \Gamma(x) \cdot \varphi(x) \, dx 
\end{equation}
for $\varphi \in C^\infty(\mathcal{S}^1 ; \mathbb{R}^2)$. 
Similarly to \eqref{eq:4.9}, we obtain \eqref{eq:4.8}. 
Therefore Lemma~\ref{theorem:4.2} follows. 
\end{proof}

%%%%%%%%%%%%%%%%%%%%%%%%%%%%%%%%%%%%%%%%%
\begin{lemma} \label{theorem:4.3}
Let $\gamma : \mathcal{S}^1 \times [0, T] \to \mathbb{R}^2$ be a family of closed curves obtained by Lemma \ref{theorem:3.9}. 
Then 
\begin{equation} \label{eq:4.10}
\mathcal{E}_p(\gamma(\cdot,t)) \le \mathcal{E}_p(\gamma_0), \quad 
\Bigl[ \dfrac{(2 \pi)^p}{p \mathcal{E}_p(\gamma_0)} \Bigr]^{\frac{1}{p-1}} \le \mathcal{L}(\gamma(t)), 
\end{equation}
for a.e. $t \in [0, T]$. 
\end{lemma}
%%%%%%%%%%%%%%%%%%%%%%%%%%%%%%%%%%%%%%%%%%
\begin{proof}
Along the same line as in the proof of Lemma \ref{theorem:4.2}, we see that 
\begin{equation}
\label{eq:4.11}
\tilde{\gamma}_n(\cdot, t) \rightharpoonup \gamma(\cdot,t) \quad \text{weakly in} \quad W^{2,p}(\mathcal{S}^1) \quad \text{for a.e.} \quad t \in (0, T). 
\end{equation}
Since 
\begin{equation*}
E_p(\Gamma) = \dfrac{1}{p \mathcal{L}(\gamma)^{2p-1}} \int^1_0 |\partial^2_x \Gamma|^p \, dx 
\end{equation*}
for $\Gamma \in \mathcal{AC}$, we deduce from \eqref{eq:3.3}, \eqref{eq:4.11} and Lemma \ref{theorem:3.11} that 
\begin{equation}
\label{eq:4.12}
\mathcal{E}_p(\gamma(t)) \le \liminf_{n \to \infty} \bigl[ E_p(\tilde{\gamma}_n(t) + \lambda \mathcal{L}(\tilde{\gamma}_n(t)) \Bigr] \le \mathcal{E}_p(\gamma_0). 
\end{equation}
Moreover, combining Lemma \ref{theorem:2.2} with \eqref{eq:4.12}, we obtain the lower estimate on $\mathcal{L}(\gamma)$ as in \eqref{eq:4.10}. 
Therefore Lemma \ref{theorem:4.3} follows. 
\end{proof}

%%%%%%%%%%%%%%%%%%%%%%%%%%%%%%%%%%%%%%%%%%%%%%%%%%%%%%%%%%%
\begin{lemma} \label{theorem:4.4}
Let $\gamma : \mathcal{S}^1 \times [0, T] \to \mathbb{R}^2$ be a family of closed curves obtained by Lemma \ref{theorem:3.9}. 
Then 
\begin{align}
\label{eq:4.13}
\begin{split}
& \int^{\tau_2}_{\tau_1} \!\!\! \int^1_0 \Bigl[ \dfrac{|\partial^2_x \gamma|^{p-2} \partial^2_x \gamma}{\mathcal{L}(\gamma)^{2p-1}} \cdot \partial^2_x \eta 
    - \dfrac{2p-1}{p} \dfrac{|\partial^2_x \gamma|^{p} \partial_x \gamma}{\mathcal{L}(\gamma)^{2p+1}} \cdot \partial_x \eta \\
& \qquad \quad 
+ \dfrac{\lambda}{\mathcal{L}(\gamma)} \partial_x \gamma \cdot \partial_x \eta 
+ \mathcal{L}(\gamma) \partial_t \gamma \cdot \eta  
+ \mathcal{L}(\gamma) \partial_t \gamma \cdot  \Phi_1(\gamma,\eta) \partial_x \gamma \Bigr] \, dxdt=0   
\end{split}
\end{align}
for all $0 \le \tau_1 \le \tau_2 \le T$ and $\eta \in L^\infty(0,T;W^{2,p}(\mathcal{S}^1))$. 
\end{lemma}
%%%%%%%%%%%%%%%%%%%%%%%%%%%%%%%%%%%%%%%%%%%%%%%%%%%%%%%%%%%
\begin{proof}
For the simplicity, we prove the case of $\tau_1 = 0$. 
Assume that \eqref{eq:4.13} does not hold. 
Then we find $0 < \tau < T$ and $\eta \in L^\infty(0,T;W^{2,p}(\mathcal{S}^1))$ such that 
\begin{align*}
& -\delta:= \int^{\tau}_{0} \!\!\! \int^1_0 \Bigl[ \dfrac{|\partial^2_x \gamma|^{p-2} \partial^2_x \gamma}{\mathcal{L}(\gamma)^{2p-1}} \cdot \partial^2_x \eta 
    - \dfrac{2p-1}{p} \dfrac{|\partial^2_x \gamma|^{p} \partial_x \gamma}{\mathcal{L}(\gamma)^{2p+1}} \cdot \partial_x \eta \\
& \qquad \qquad \qquad
+ \dfrac{\lambda}{\mathcal{L}(\gamma)} \partial_x \gamma \cdot \partial_x \eta 
+ \mathcal{L}(\gamma) \partial_t \gamma \cdot \eta  
+ \mathcal{L}(\gamma) \partial_t \gamma \cdot  \Phi_1(\gamma,\eta) \partial_x \gamma \Bigr] \, dxdt < 0.   
\end{align*}
For $0 < \varepsilon < 1$ we define $\rho_\varepsilon \in H^1(0,T)$ by 
\begin{align*}
\rho_\varepsilon(t):= 
\begin{cases}
1 & \qquad \text{if} \quad 0 \le t \le \tau,  \\
-\dfrac{t-\tau}{\varepsilon} + 1 & \qquad \text{if} \quad \tau < t \le \tau + \varepsilon,  \\
0 & \qquad \text{if} \quad \tau + \varepsilon < t \le T,  
\end{cases}
\end{align*}
and set $\eta_\varepsilon := \rho_\varepsilon \eta$. 
Taking $\eta_\varepsilon$ as $\eta$ in \eqref{eq:4.1} we observe from \eqref{eq:4.2} that  
\begin{equation}
\label{eq:4.14}
\begin{aligned}
0 &= \int^T_0 \rho_\varepsilon \int^1_0 \Bigl[ \dfrac{|\partial^2_x \gamma|^{p-2} \partial^2_x \gamma}{\mathcal{L}(\gamma)^{2p-1}} \cdot \partial^2_x \eta 
    - \dfrac{2p-1}{p} \dfrac{|\partial^2_x \gamma|^{p} \partial_x \gamma}{\mathcal{L}(\gamma)^{2p+1}} \cdot \partial_x \eta \\
& \qquad \quad 
+ \dfrac{\lambda}{\mathcal{L}(\gamma)} \partial_x \gamma \cdot \partial_x \eta 
+ \mathcal{L}(\gamma) \partial_t \gamma \cdot \eta  
+ \mathcal{L}(\gamma) (\partial_t \gamma \cdot \partial_x \gamma) \Phi_1(\gamma,\eta) \Bigr] \, dxdt  \\
&\le -\delta 
+ \int^{\tau+\varepsilon}_\tau \!\!\! \int^1_0 \Bigl[ \dfrac{|\partial^2_x \gamma|^{p-1}}{\mathcal{L}(\gamma)^{2p-1}} | \partial^2_x \eta | 
    + \dfrac{2p-1}{p} \dfrac{|\partial^2_x \gamma|^{p}}{\mathcal{L}(\gamma)^{2p}} | \partial_x \eta | \\
& \qquad \quad 
+ \lambda | \partial_x \eta |
+ \mathcal{L}(\gamma) | \partial_t \gamma | | \eta | 
+ \mathcal{L}(\gamma)^2 |\partial_t \gamma| |\Phi_1(\gamma,\eta)| \Bigr] \, dxdt. 
\end{aligned}
\end{equation}
Since $\gamma \in L^\infty(0,T; W^{2,p}(\mathcal{S}^1))$ and $\eta \in L^\infty(0,T;W^{2,p}(\mathcal{S}^1))$, by Lemma \ref{theorem:4.3} we have   
\begin{equation}
\label{eq:4.15}
\int^{\tau+\varepsilon}_\tau \!\!\! \int^1_0 \dfrac{|\partial^2_x \gamma|^{p-1}}{\mathcal{L}(\gamma)^{2p-1}} | \partial^2_x \eta | \, dx dt 
\le C \int^{\tau+\varepsilon}_\tau \| \partial^2_x \gamma \|_{L^p(\mathcal{S}^1)}^{p-1} \| \partial^2_x \eta \|_{L^p(\mathcal{S}^1)} \, dt 
\le C \varepsilon. 
\end{equation}
Thanks to Lemma \ref{theorem:3.8}, similarly we obtain 
\begin{equation}
\label{eq:4.16}
\begin{aligned}
&\int^{\tau+\varepsilon}_\tau \!\!\! \int^1_0 \dfrac{|\partial^2_x \gamma|^{p}}{\mathcal{L}(\gamma)^{2p}} | \partial_x \eta | \, dx dt \\
& \quad  \le C \int^{\tau+\varepsilon}_\tau \| \partial^2_x \gamma \|_{L^\infty(\mathcal{S}^1)} \| \partial^2_x \gamma \|_{L^p(\mathcal{S}^1)}^{p-1} \| \partial_x \eta \|_{L^p(\mathcal{S}^1)} \, dt \\
& \quad \le C \Bigl[ \int^{\tau+\varepsilon}_\tau \| \partial^2_x \gamma \|_{L^\infty(\mathcal{S}^1)}^2 \, dt \Bigr]^{\frac{1}{2}}
         \Bigl[ \int^{\tau+\varepsilon}_\tau \| \partial_x \eta \|_{L^p(\mathcal{S}^1)}^2 \, dt \Bigr]^{\frac{1}{2}} 
\le C \sqrt{\varepsilon}.   
\end{aligned}
\end{equation}
Since $\eta \in L^\infty(0,T;W^{2,p}(\mathcal{S}^1))$, we have 
\begin{equation}
\label{eq:4.17}
\int^{\tau+\varepsilon}_\tau \!\!\! \int^1_0 \lambda | \partial_x \eta | \, dx dt 
\le \lambda \int^{\tau+\varepsilon}_\tau \| \partial_x \eta \|_{L^p(\mathcal{S}^1)} \, dt 
\le C \varepsilon. 
\end{equation}
It follows from $\gamma \in H^1(0,T;L^2(\mathcal{S}^1))$ that 
\begin{equation}
\label{eq:4.18}
\begin{aligned}
& \int^{\tau+\varepsilon}_\tau \!\!\! \int^1_0 \mathcal{L}(\gamma) | \partial_t \gamma | | \eta | \, dx dt 
  \le C \int^{\tau+\varepsilon}_\tau \| \partial_t \gamma \|_{L^2(\mathcal{S}^1)} \| \eta \|_{L^2(\mathcal{S}^1)} \, dt \\
& \qquad \le C \Bigl[ \int^{\tau+\varepsilon}_\tau \| \partial_t \gamma \|_{L^2(\mathcal{S}^1)}^2 \, dt \Bigr]^{\frac{1}{2}} 
           \Bigl[ \int^{\tau+\varepsilon}_\tau \| \eta \|_{L^2(\mathcal{S}^1)}^2 \, dt \Bigr]^{\frac{1}{2}} 
\le C \sqrt{\varepsilon}. 
\end{aligned}
\end{equation}
Since $|\Phi_1(\gamma, \eta)| \le C \| \eta_x \|_{L^1(\mathcal{S}^1)}$, along the same line as in \eqref{eq:4.18}, we see that 
\begin{equation}
\label{eq:4.19}
\int^{\tau+\varepsilon}_\tau \!\!\! \int^1_0 \mathcal{L}(\gamma)^2 |\partial_t \gamma| |\Phi_1(\gamma,\eta)| \, dxdt 
\le C \sqrt{\varepsilon}.  
\end{equation}
Plugging \eqref{eq:4.15}, \eqref{eq:4.16}, \eqref{eq:4.17}, \eqref{eq:4.18} and \eqref{eq:4.19} into \eqref{eq:4.14}, we observe that  
\begin{equation*}
0 \le -\delta + C \sqrt{\varepsilon}.  
\end{equation*}
This clearly leads a contradiction for $0 < \varepsilon < (\delta/2C)^2$. 
For the case of $\tau_1 \in (0,T)$, setting 
\begin{align*}
\rho_\varepsilon(t):= 
\begin{cases}
\frac{1}{\varepsilon}(t-\tau_1) + 1 & \qquad \text{if} \quad \tau_1 - \varepsilon \le t < \tau_1,  \\
1 & \qquad \text{if} \quad \tau_1 \le t \le \tau_2,  \\
-\frac{1}{\varepsilon}(t-\tau_2) + 1 & \qquad \text{if} \quad \tau_2 < t \le \tau_2 + \varepsilon,  \\
0 & \qquad \text{otherwise},  
\end{cases}
\end{align*}
we obtain \eqref{eq:4.13} along the same line as above. 
Therefore Lemma \ref{theorem:4.4} follows. 
\end{proof}

%%%%%%%%%%%%%%%%%%%%%%%%%%%%%%%%%%%%%%%%%%%%%%%%%%%%%%%%%%%
\begin{lemma} \label{theorem:4.5}
Let $\gamma : \mathcal{S}^1 \times [0, T] \to \mathbb{R}^2$ be a family of closed curves obtained by Lemma \ref{theorem:3.9}. 
Then 
\begin{equation*}
\partial_x (|\partial^2_x \gamma|^{p-2} \partial^2_x \gamma) \in L^2(0,T; L^2(\mathcal{S}^1)). 
\end{equation*}
\end{lemma}
%%%%%%%%%%%%%%%%%%%%%%%%%%%%%%%%%%%%%%%%%%%%%%%%%%%%%%%%%%%
\begin{proof}
Fix $\tau \in (0, T)$ arbitrarily and let $\varepsilon > 0$ small enough. 
By Lemma \ref{theorem:4.4} we have 
\begin{equation}
\label{eq:4.20}
\begin{aligned}
& \int^{\tau + \varepsilon}_{\tau} \!\!\! \int^1_0 \Bigl[ \dfrac{|\partial^2_x \gamma|^{p-2} \partial^2_x \gamma}{\mathcal{L}(\gamma)^{2p-1}} \cdot \partial^2_x \eta 
    - \dfrac{2p-1}{p} \dfrac{|\partial^2_x \gamma|^{p} \partial_x \gamma}{\mathcal{L}(\gamma)^{2p+1}} \cdot \partial_x \eta \\
& \qquad \qquad 
+ \dfrac{\lambda}{\mathcal{L}(\gamma)} \partial_x \gamma \cdot \partial_x \eta 
+ \mathcal{L}(\gamma) \partial_t \gamma \cdot \eta  
+ \mathcal{L}(\gamma) \partial_t \gamma \cdot \Phi_1(\gamma,\eta) \partial_x \gamma \Bigr] \, dxdt=0    
\end{aligned}
\end{equation}
for all $\eta \in L^\infty(0,T;W^{2,p}(\mathcal{S}^1))$. 
From now on we take $\varphi \in W^{2,p}(\mathcal{S}^1)$ as $\eta$ in \eqref{eq:4.20}. 
This together with the Lebesgue differentiation theorem implies that 
\begin{align}
\label{eq:4.21}
\begin{aligned}
&\int^1_0 \Bigl[ \dfrac{|\partial^2_x \gamma|^{p-2} \partial^2_x \gamma}{\mathcal{L}(\gamma)^{2p-1}} \cdot \partial^2_x \varphi 
    - \dfrac{2p-1}{p} \dfrac{|\partial^2_x \gamma|^{p} \partial_x \gamma}{\mathcal{L}(\gamma)^{2p+1}} \cdot \partial_x \varphi \\
& \qquad  
+ \dfrac{\lambda}{\mathcal{L}(\gamma)} \partial_x \gamma \cdot \partial_x \varphi 
+ \mathcal{L}(\gamma) \partial_t \gamma \cdot \varphi  
+ \mathcal{L}(\gamma) \partial_t \gamma \cdot \Phi_1(\gamma,\varphi) \partial_x \gamma \Bigr] \, dx=0    
\end{aligned}
\end{align}
for a.e. $t \in (0, T)$. 
Fix $\psi \in C^\infty(\mathcal{S}^1)$ arbitrarily and set $\varphi_1$ as in Lemma \ref{theorem:3.7}. 
Taking $\varphi_1$ as $\varphi$ in \eqref{eq:4.21}, we obtain 
\begin{align*}
& \int^1_0 \dfrac{|\partial^2_x \gamma|^{p-2} \partial^2_x \gamma}{\mathcal{L}(\gamma)^{2p-1}} \cdot \psi \, dx \\
& = -2\int^1_0 \dfrac{|\partial^2_x \gamma|^{p-2} \partial^2_x \gamma}{\mathcal{L}(\gamma)^{2p-1}} \cdot \beta \, dx 
          + \dfrac{2p-1}{p} \int^1_0 \dfrac{|\partial^2_x \gamma|^{p} \partial_x \gamma}{\mathcal{L}(\gamma)^{2p+1}} \cdot \partial_x \varphi_1 \, dx \\
& \quad  - \dfrac{\lambda}{\mathcal{L}(\gamma)} \int^1_0 \partial_x \gamma \cdot \partial_x \varphi_1 \, dx  
          - \mathcal{L}(\gamma) \int^1_0 \partial_t \gamma \cdot \varphi_1 \, dx \\
& \quad - \mathcal{L}(\gamma) \int^1_0 \partial_t \gamma \cdot \Phi_1(\gamma,\varphi_1) \partial_x \gamma \, dx 
 =: I_1 + I_2 + I_3 + I_4 + I_5. 
\end{align*}
By \eqref{eq:4.2} and Lemmas \ref{theorem:3.7} and \ref{theorem:4.3} we have 
\begin{align*}
|I_1| &\le \dfrac{2}{\mathcal{L}(\gamma)^{2p-1}} |\beta| \|\partial^2_x \gamma\|^{p-1}_{L^p(\mathcal{S}^1)} 
      \le C \|\partial^2_x \gamma\|^{p-1}_{L^p(\mathcal{S}^1)} \|\psi\|_{L^1(\mathcal{S}^1)}, \\
|I_2| &\le \dfrac{2p-1}{p \mathcal{L}(\gamma)^{2p}} \|\partial^2_x \gamma\|^{p}_{L^p(\mathcal{S}^1)} \|\varphi_1\|_{C^1(\mathcal{S}^1)} 
      \le C \|\partial^2_x \gamma\|^{p}_{L^p(\mathcal{S}^1)} \|\psi\|_{L^1(\mathcal{S}^1)}, \\
|I_3| &\le \lambda  \|\varphi_1\|_{C^1(\mathcal{S}^1)} \le C \|\psi\|_{L^1(\mathcal{S}^1)}, \\
|I_4| &\le \mathcal{L}(\gamma) \| \partial_t \gamma \|_{L^1(\mathcal{S}^1)} \| \varphi_1\|_{C(\mathcal{S}^1)} 
      \le C \mathcal{L}(\gamma) \| \partial_t \gamma \|_{L^1(\mathcal{S}^1)} \| \psi\|_{L^1(\mathcal{S}^1)}. 
\end{align*}
Since 
\begin{align*}
\|\Phi_1(\gamma, \varphi_1)\|_{L^\infty(\mathcal{S}^1)} \le \dfrac{2}{\mathcal{L}(\gamma)} \| \varphi_1\|_{C^1(\mathcal{S}^1)} \le C \|\psi\|_{L^1(\mathcal{S}^1)}, 
\end{align*}
we also obtain 
\begin{align*}
|I_5| \le \mathcal{L}(\gamma)^2 \|\Phi_1(\gamma, \varphi_1)\|_{L^\infty(\mathcal{S}^1)} \| \partial_t \gamma \|_{L^1(\mathcal{S}^1)} 
      \le C \mathcal{L}(\gamma) \|\partial_t \gamma \|_{L^1(\mathcal{S}^1)} \|\psi\|_{L^1(\mathcal{S}^1)}.  
\end{align*}
Thus we see that 
\begin{align*}
\Bigl| \int^1_0 \dfrac{|\partial^2_x \gamma|^{p-2} \partial^2_x \gamma}{\mathcal{L}(\gamma)^{2p-1}} \cdot \psi \, dx \Bigr| 
 \le C\bigl( 1 + \|\partial^2_x \gamma\|^{p}_{L^p(\mathcal{S}^1)}   
               + \mathcal{L}(\gamma) \| \partial_t \gamma \|_{L^1(\mathcal{S}^1)}  \bigr) \|\psi\|_{L^1(\mathcal{S}^1)},   
\end{align*}
and then 
\begin{equation}
\label{eq:4.22}
\begin{aligned}
\|\partial^2_x \gamma\|^{p-1}_{L^\infty(\mathcal{S}^1)}
 \le C\bigl( 1 + \|\partial^2_x \gamma\|^{p}_{L^p(\mathcal{S}^1)}  
                    + \mathcal{L}(\gamma) \| \partial_t \gamma \|_{L^1(\mathcal{S}^1)}  \bigr) 
\end{aligned}
\end{equation}
for a.e. $t \in (0, T)$. 

Set $\varphi_2$ as in Lemma \ref{theorem:3.7}. 
Taking $\varphi_2$ as $\varphi$ in \eqref{eq:4.21}, we have 
\begin{align*}
& \int^1_0 \dfrac{|\partial^2_x \gamma|^{p-2} \partial^2_x \gamma}{\mathcal{L}(\gamma)^{2p-1}} \cdot \partial_x \psi \, dx \\
& \quad = \dfrac{2p-1}{p}\int^1_0 \dfrac{|\partial^2_x \gamma|^{p} \partial_x \gamma}{\mathcal{L}(\gamma)^{2p+1}} \cdot \partial_x \varphi_2 \, dx 
               - \int^1_0 \dfrac{\lambda}{\mathcal{L}(\gamma)} \partial_x \gamma \cdot \partial_x \varphi_2 \, dx \\
& \qquad \quad - \int^1_0 \mathcal{L}(\gamma) \partial_t \gamma \cdot \varphi_2 \, dx   
               - \int^1_0 \mathcal{L}(\gamma) \partial_t \gamma \cdot \Phi_1(\gamma,\varphi_2) \partial_x \gamma \, dx  \\
& \quad =: I'_1 + I'_2 + I'_3 + I'_4. 
\end{align*}
Along the same line as above, we have  
\begin{align*}
|I'_1| &\le \dfrac{2p-1}{p \mathcal{L}(\gamma)^{2p}} \| \partial^2_x \gamma\|_{L^\infty(\mathcal{S}^1)}^{p-1} 
                                                     \|\partial^2_x \gamma\|_{L^2(\mathcal{S}^1)} \|\partial_x \varphi_2\|_{L^2(\mathcal{S}^1)} \\
       &\le C \| \partial^2_x \gamma\|_{L^\infty(\mathcal{S}^1)}^{p-1} \| \partial^2_x \gamma\|_{L^2(\mathcal{S}^1)} \|\psi\|_{L^2(\mathcal{S}^1)}, \\
|I'_2| &\le \lambda \| \partial_x \varphi_2 \|_{L^1(\mathcal{S}^1)}
       \le C \|\psi\|_{L^1(\mathcal{S}^1)}, \\
|I'_3| &\le \mathcal{L}(\gamma) \| \partial_t \gamma \|_{L^1(\mathcal{S}^1)} \| \varphi_2 \|_{L^\infty(\mathcal{S}^1)}
       \le C \mathcal{L}(\gamma) \| \partial_t \gamma \|_{L^1(\mathcal{S}^1)} \| \psi \|_{L^1(\mathcal{S}^1)}. 
\end{align*}
Since
\begin{equation*}
\|\Phi_1(\gamma, \varphi_2)\|_{L^\infty(\mathcal{S}^1)} 
 \le \dfrac{2}{\mathcal{L}(\gamma)} \|\partial_x \varphi_2 \|_{L^1(\mathcal{S}^1)} \le C \|\psi\|_{L^1(\mathcal{S}^1)}, 
\end{equation*}
we also obtain 
\begin{equation*}
|I'_4| \le C \mathcal{L}(\gamma) \| \partial_t \gamma \|_{L^1(\mathcal{S}^1)} \|\psi\|_{L^1(\mathcal{S}^1)}. 
\end{equation*}
Thus we see that 
\begin{align*}
& \Bigl| \int^1_0 \dfrac{|\partial^2_x \gamma|^{p-2} \partial^2_x \gamma}{\mathcal{L}(\gamma)^{2p-1}} \cdot \partial_x \psi \, dx \Bigr| \\
& \quad \le C \Bigl(1 + \|\partial^2_x \gamma\|_{L^\infty(\mathcal{S}^1)}^{p-1} \|\partial^2_x \gamma\|_{L^2(\mathcal{S}^1)} 
         + \mathcal{L}(\gamma) \| \partial_t \gamma \|_{L^1(\mathcal{S}^1)} \Bigr) \| \psi \|_{L^2(\mathcal{S}^1)},  
\end{align*}
and then 
\begin{equation}
\label{eq:4.23}
\begin{aligned}
& \| \partial_x (|\partial^2_x \gamma|^{p-2} \partial^2_x \gamma) \|_{L^2(\mathcal{S}^1)} \\
& \qquad \le C \bigl(1 + \|\partial^2_x \gamma\|_{L^\infty(\mathcal{S}^1)}^{p-1} \|\partial^2_x \gamma\|_{L^2(\mathcal{S}^1)} 
         + \mathcal{L}(\gamma) \| \partial_t \gamma \|_{L^1(\mathcal{S}^1)} \bigr) 
\end{aligned}
\end{equation}
for a.e. $t \in (0, T)$. 
This together with \eqref{eq:4.22} and Lemmas \ref{theorem:3.4} and \ref{theorem:4.3} implies that 
\begin{align*}
& \int^T_0 \| \partial_x (|\partial^2_x \gamma|^{p-2} \partial^2_x \gamma) \|^2_{L^2(\mathcal{S}^1)} \, dt \\
& \quad \le C T + C \int^T_0 \| \partial^2_x \gamma\|^{2(p-1)}_{L^\infty(\mathcal{S}^1)} \, dt + C \int^T_0 \mathcal{L}(\gamma)\| \partial_t \gamma \|^2_{L^2(\mathcal{S}^1)} \, dt < \infty. 
\end{align*}
Therefore Lemma \ref{theorem:4.5} follows. 
\end{proof}

%%%%%%%%%%%%%%%%%%%%%%%%%%%%%%%%%%%%%%%%%%%%%%%%%%%%%%%%%%%%%
\begin{lemma}
\label{theorem:4.6}
Let $\gamma : \mathcal{S}^1 \times [0, T] \to \mathbb{R}^2$ be a family of closed curves obtained by Lemma \ref{theorem:3.9}. 
Then 
\begin{equation*}
\int^{\tau}_0 \!\!\! \int^1_0 \mathcal{L}(\gamma) \partial_t \gamma \cdot \partial_x \gamma \, dx dt =0   
\end{equation*}
for all $\tau \in [0, T]$. 
\end{lemma}
%%%%%%%%%%%%%%%%%%%%%%%%%%%%%%%%%%%%%%%%%%%%%%%%%%%%%%%%%%%%%
\begin{proof}
Fix $\tau \in [0, T]$ arbitrarily. 
By Lemmas \ref{theorem:4.4} and \ref{theorem:4.5} we can reduce the weak form \eqref{eq:4.1} into 
\begin{align}
\label{eq:4.24}
\begin{split}
& \int^\tau_0 \!\!\!\! \int^1_0 \Bigl[ - \dfrac{\partial_x(|\partial^2_x \gamma|^{p-2} \partial^2_x \gamma)}{\mathcal{L}(\gamma)^{2p-1}} \cdot \partial_x \eta 
    - \dfrac{2p-1}{p} \dfrac{|\partial^2_x \gamma|^{p} \partial_x \gamma}{\mathcal{L}(\gamma)^{2p+1}} \cdot \partial_x \eta \\
& \qquad \quad 
+ \dfrac{\lambda}{\mathcal{L}(\gamma)} \partial_x \gamma \cdot \partial_x \eta 
+ \mathcal{L}(\gamma) \partial_t \gamma \cdot \eta  
+ \mathcal{L}(\gamma)\partial_t \gamma \cdot \Phi_1(\gamma,\eta) \partial_x \gamma \Bigr] \, dxdt=0     
\end{split}
\end{align}
for all $\eta \in L^\infty(0,T;W^{1,p}(\mathcal{S}^1))$. 
Fix $\rho \in C^\infty(\mathcal{S}^1; \mathbb{R})$ arbitrarily. 
We take $\rho \partial_x \gamma$ as $\eta$ in \eqref{eq:4.24}. 
First, by $|\partial_x \gamma|=\mathcal{L}(\gamma)$ and 
\begin{equation*}
\partial_x \gamma \cdot \partial^2_x \gamma = 0 \quad \text{for a.e.} \quad x \in \mathcal{S}^1 \quad \text{and} \quad t \in (0,\tau), 
\end{equation*}
we observe from \eqref{eq:3.31} that 
\begin{align*}
\Phi_1(\gamma,\rho \partial_x \gamma) 
 &= \dfrac{1}{\mathcal{L}(\gamma)^2} \Bigl[ x \int^1_0 \partial_x \gamma \cdot \partial_x(\rho \partial_x \gamma) \, d\tilde{x} 
                                            - \int^x_0 \partial_x \gamma \cdot \partial_x(\rho \partial_x \gamma) \, d\tilde{x} \Bigr] \\
 &= x \int^1_0 \partial_x \rho \, d\tilde{x} - \int^x_0 \partial_x \rho \, d\tilde{x} 
  = \rho(0) - \rho(x) 
\end{align*}
for a.e. $t \in (0,\tau)$, and then 
\begin{equation}
\label{eq:4.25}
\begin{aligned}
& \int^\tau_0 \!\!\! \int^1_0 \mathcal{L}(\gamma) \partial_t \gamma \cdot \Phi_1(\gamma,\rho \partial_x \gamma) \partial_x \gamma \, dxdt \\
& =\rho(0) \int^\tau_0 \!\!\! \int^1_0 \mathcal{L}(\gamma) \partial_t \gamma \cdot \partial_x \gamma \, dx dt 
   -\int^\tau_0 \!\!\! \int^1_0 \mathcal{L}(\gamma) \partial_t \gamma \cdot \rho \partial_x \gamma \, dx dt. 
\end{aligned}
\end{equation}
Similarly we obtain 
\begin{equation}
\label{eq:4.26}
\int^\tau_0 \!\!\! \int^1_0 \dfrac{\lambda}{\mathcal{L}(\gamma)} \partial_x \gamma \cdot \partial_x (\rho \partial_x \gamma) \, dx dt 
 = \lambda \int^\tau_0 \!\!\! \int^1_0 \mathcal{L}(\gamma) \partial_x \rho \, dx dt = 0 
\end{equation}
and 
\begin{equation}
\label{eq:4.27}
\int^\tau_0 \!\!\! \int^1_0 \dfrac{|\partial^2_x \gamma|^{p} \partial_x \gamma}{\mathcal{L}(\gamma)^{2p+1}} \cdot \partial_x (\rho \partial_x \gamma) \, dx dt 
 = \int^\tau_0 \!\!\! \int^1_0 \dfrac{|\partial^2_x \gamma|^{p}}{\mathcal{L}(\gamma)^{2p-1}} \partial_x \rho \, dx dt. 
\end{equation}
Moreover, integrating by parts, we see that 
\begin{equation}
\label{eq:4.28}
\begin{aligned}
& -\int^\tau_0 \!\!\!\! \int^1_0 \dfrac{\partial_x(|\partial^2_x \gamma|^{p-2} \partial^2_x \gamma)}{\mathcal{L}(\gamma)^{2p-1}} \cdot \partial_x (\rho \partial_x \gamma) \, dx dt \\
& \quad = - \int^\tau_0 \!\!\!\! \int^1_0 \dfrac{\partial_x(|\partial^2_x \gamma|^{p-2} \partial^2_x \gamma)}{\mathcal{L}(\gamma)^{2p-1}} 
                                \cdot \bigl[ \partial_x \rho \partial_x \gamma + \rho \partial^2_x \gamma \bigr] \, dx dt \\
& \quad = \int^\tau_0 \!\!\!\! \int^1_0 \dfrac{|\partial^2_x \gamma|^{p-2} \partial^2_x \gamma}{\mathcal{L}(\gamma)^{2p-1}} 
                                \cdot \bigl[ 2 \partial_x \rho \partial^2_x \gamma + \rho \partial^3_x \gamma \bigr] \, dx dt \\
& \quad = \int^\tau_0 \!\!\!\! \int^1_0 \Bigl[ 2 \dfrac{|\partial^2_x \gamma|^{p}}{\mathcal{L}(\gamma)^{2p-1}} \partial_x \rho 
                                             + \dfrac{|\partial^2_x \gamma|^{p-2} \partial^2_x \gamma \cdot \partial^3_x \gamma}{\mathcal{L}(\gamma)^{2p-1}} \rho \Bigr] \, dx dt \\
& \quad = \int^\tau_0 \!\!\!\! \int^1_0 \Bigl[ 2 \dfrac{|\partial^2_x \gamma|^{p}}{\mathcal{L}(\gamma)^{2p-1}} \partial_x \rho 
                                             + \dfrac{1}{p}\dfrac{\partial_x(|\partial^2_x \gamma|^{p})}{\mathcal{L}(\gamma)^{2p-1}} \rho \Bigr] \, dx dt. 
\end{aligned}
\end{equation}
Plugging \eqref{eq:4.25}, \eqref{eq:4.26}, \eqref{eq:4.27} and \eqref{eq:4.28} into \eqref{eq:4.24}, we have 
\begin{align*}
\rho(0) \int^T_0 \!\!\!\! \int^1_0 \mathcal{L}(\gamma) \partial_t \gamma \cdot \partial_x \gamma \, dx dt 
 &= -\dfrac{1}{p} \int^\tau_0 \!\!\!\! \int^1_0 \dfrac{\partial_x(|\partial^2_x \gamma|^{p}) \rho + |\partial^2_x \gamma|^{p} \partial_x \rho}{\mathcal{L}(\gamma)^{2p-1}} \, dx dt \\
 &= -\dfrac{1}{p} \int^\tau_0 \dfrac{1}{\mathcal{L}(\gamma)^{2p-1}} \Bigl[ |\partial^2_x \gamma|^{p} \rho \Bigr]^{x=1}_{x=0} \, dt
  =0. 
\end{align*}
Thus Lemma \ref{theorem:4.6} follows. 
\end{proof}

%%%%%%%%%%%%%%%%%%%%%%%%%%%%%%%%%%%%%%%%%
\begin{lemma} \label{theorem:4.7}
Let $\gamma : \mathcal{S}^1 \times [0, T] \to \mathbb{R}^2$ be a family of closed curves obtained by Lemma~{\rm \ref{theorem:3.9}}. 
Then 
\begin{equation}
\label{eq:4.29} 
\mathcal{E}_p(\gamma(T)) - \mathcal{E}_p(\gamma_0) \le - \dfrac{1}{2} \int^T_0 \!\!\! \int^1_0 \mathcal{L}(\gamma) |\partial_t \gamma|^2 \, dx dt. 
\end{equation}
\end{lemma}
%%%%%%%%%%%%%%%%%%%%%%%%%%%%%%%%%%%%%%%%%
\begin{proof} 
Since $\gamma_n(x,0)=\gamma_{0,n}(x)=\gamma_0(x)$ and $\gamma_{n}(x,T)=\gamma_{n,n}(x)$, it follows from \eqref{eq:3.15} that 
\begin{equation}
\label{eq:4.30}
\dfrac{1}{2} \int^T_0 \!\!\! \int^1_0 \mathcal{L}(\tilde{\Gamma}_{n}) |V_{n}|^2\, dx dt 
 \le \mathcal{E}_p(\gamma_0) - \mathcal{E}_p(\gamma_n(T)). 
\end{equation}
First we claim that 
\begin{equation}
\label{eq:4.31}
\liminf_{n \to \infty} \int^T_0 \!\!\! \int^1_0 \mathcal{L}(\tilde{\Gamma}_{n}) |V_{n}|^2\, dx dt 
  \ge \int^T_0 \!\!\! \int^1_0 \mathcal{L}(\gamma) |\partial_t \gamma|^2\, dx dt.
\end{equation}
Thanks to Lemma \ref{theorem:4.3}, we find $C_1>0$ and $C_2>0$ such that $C_1 < \mathcal{L}(\gamma) < C_2$ for a.e. $t \in (0, T)$. 
This together with Lemma \ref{theorem:3.9} implies that 
\begin{equation}
\label{eq:4.32}
\liminf_{n \to \infty} \int^T_0 \!\!\! \int^1_0 \mathcal{L}(\gamma) |V_{n}|^2\, dx dt 
  \ge \int^T_0 \!\!\! \int^1_0 \mathcal{L}(\gamma) |\partial_t \gamma|^2\, dx dt.  
\end{equation}
Moreover, it follows from Lemmata \ref{theorem:3.4} and \ref{theorem:3.11} that 
\begin{equation}
\label{eq:4.33}
\begin{aligned}
& \Bigl| \int^T_0 \!\!\! \int^1_0 \mathcal{L}(\tilde{\Gamma}_{n}) |V_{n}|^2\, dx dt - \int^T_0 \!\!\! \int^1_0 \mathcal{L}(\gamma) |V_{n}|^2\, dx dt  \Bigr| \\
& \qquad \le \| \tilde{\Gamma}_n - \gamma \|_{L^\infty(0,T;C^{1, \alpha}(\mathcal{S}^1))}  \int^T_0 \!\!\! \int^1_0 |V_{n}|^2\, dx dt \\
& \qquad \le C \| \tilde{\Gamma}_n - \gamma \|_{L^\infty(0,T;C^{1, \alpha}(\mathcal{S}^1))} 
 \to 0 
\end{aligned}
\end{equation}
as $n \to \infty$. 
Combining \eqref{eq:4.32} with \eqref{eq:4.33}, we obtain 
\begin{align*}
&\liminf_{n \to \infty} \int^T_0 \!\!\! \int^1_0 \mathcal{L}(\tilde{\Gamma}_{n}) |V_{n}|^2\, dx dt \\
& = \liminf_{n \to \infty} \Bigl[ \int^T_0 \!\!\! \int^1_0 ( \mathcal{L}(\tilde{\Gamma}_n) - \mathcal{L}(\gamma) ) |V_n|^2 \, dx dt + \int^T_0 \!\!\! \int^1_0 \mathcal{L}(\gamma) |V_{n}|^2\, dx dt  \Bigr] \\
& = \int^T_0 \!\!\! \int^1_0 \mathcal{L}(\gamma) |\partial_t \gamma|^2\, dx dt. 
\end{align*} 
Thus \eqref{eq:4.31} follows. 
On the other hand, along the same line as in the proof of Lemma~\ref{theorem:4.2}, we have  
\begin{equation*}
\gamma_{n}(T) \rightharpoonup \gamma(T) \quad \text{weakly in}\quad W^{2,p}(\mathcal{S}^1).
\end{equation*} 
This together with Lemma \ref{theorem:3.11} implies that   
\begin{equation}\label{eq:4.34}
\liminf_{n \to \infty} \mathcal{E}_p(\gamma_{n}(T)) \ge \mathcal{E}_p(\gamma(T)). 
\end{equation}
Thus, plugging \eqref{eq:4.31} and \eqref{eq:4.34} into \eqref{eq:4.30}, we obtain \eqref{eq:4.29}. 
\end{proof}

We are in a position to prove Theorem \ref{theorem:1.2}. 

\begin{proof}[Proof of Theorem \ref{theorem:1.2}]
Let $\gamma : \mathcal{S}^1 \times [0, T] \to \mathbb{R}^2$ be a family of closed curves obtained by Lemma \ref{theorem:3.9}. 
We prove that $\gamma$ is the desired weak solution of \eqref{eq:P}. 
To begin with, it follows from Lemma \ref{theorem:3.11} that $\gamma(x,t)=\gamma_0(x)$ for all $x \in \mathcal{S}^1$. 
We also deduce from Lemma \ref{theorem:4.1} that $\gamma$ satisfies \eqref{eq:1.2} for all $\eta \in L^\infty(0,T; W^{2,p}(\mathcal{S}^1))$. 
Thanks to Lemmas \ref{theorem:4.3} and \ref{theorem:4.7}, we see that \eqref{eq:1.3} and \eqref{eq:1.4} hold. 
Thus $\gamma : \mathcal{S}^1 \times [0,T] \to \mathbb{R}^2$ is a weak solution to \eqref{eq:P}.  
Since $0<T<\infty$ is arbitrary, Theorem \ref{theorem:1.2} follows. 
\end{proof}

%%%%%%%%%%%%%%%%%%%%%%%%%%%%%%%%%%%%%%%%%%%%%%%%%%%%%%%%%%%%%%%%%%
%%%%%%%%%%%%%%%%%%%%%%%%%%%%%%%%%%%%%%%%%%%%%%%%%%%%%%%%%%%%%%%%%%
%%%%%%%%%%%%%%%%%%%%%%%%%%%%%%%%%%%%%%%%%%%%%%%%%%%%%%%%%%%%%%%%%%

\section{Proof of Theorem \ref{theorem:1.3}} \label{section:5}
Let $\gamma_0 \in W^{2,p}(\mathcal{S}^1;\mathbb{R}^2)$ satisfy \eqref{eq:1.1}. 
In this section, we fix such $\gamma_0$ arbitrarily, and denote the admissible set $\mathcal{AC}_{\gamma_0}$ by $\mathcal{AC}$ for short. 

%%%%%%%%%%%%%%%%%%%%%%%%%%%%%%%%%%%%%%%%%%%%%%%%%%%%%%%%%%%
\begin{lemma} \label{theorem:5.1}
Let $p=2$. 
Assume that $\gamma_0 \in W^{2,2}(\mathcal{S}^1)$ satisfies \eqref{eq:1.1}. 
Let $\gamma_1 : \mathcal{S}^1 \times [0, T] \to \mathbb{R}^2$ and $\gamma_2 : \mathcal{S}^1 \times [0, T] \to \mathbb{R}^2$ be weak solutions to \eqref{eq:P}. 
Then $\gamma_1 = \gamma_2$ in $H^1(0,T;L^2(\mathcal{S}^1)) \cup L^\infty(0,T; H^2(\mathcal{S}^1))$.  
\end{lemma}
%%%%%%%%%%%%%%%%%%%%%%%%%%%%%%%%%%%%%%%%%%%%%%%%%%%%%%%%%%%
\begin{proof}
Let $p=2$ and fix $\tau \in (0, T]$ arbitrarily. 
Let $\tilde{\gamma}_i(x,t):= \rho \gamma_i(x, \rho^{-4} t)$ for $\rho>0$.  
Define $t(\tau) \in (0, \rho^4 \tau]$ by 
\begin{equation}
\label{eq:5.1}
\|(\tilde{\gamma}_1 - \tilde{\gamma}_2)(\cdot,t(\tau))\|_{L^2(\mathcal{S}^1)} = \max_{t \in [0, \rho^4 \tau]}\|(\tilde{\gamma}_1 - \tilde{\gamma}_2)(\cdot,t)\|_{L^2(\mathcal{S}^1)}. 
\end{equation}
Set $\tilde{\lambda}:= \rho^{-2} \lambda$. 
Then 
\begin{equation}
\label{eq:5.2}
\begin{aligned}
&\int^{\rho^4 \tau}_0 \!\!\! \int^1_0 \Bigl[ \dfrac{\partial^2_x \tilde{\gamma}_i}{\mathcal{L}(\tilde{\gamma}_i)^{3}} \cdot \partial^2_x \varphi 
  - \dfrac{3}{2} \dfrac{|\partial^2_x \tilde{\gamma}_i|^{2} \partial_x \tilde{\gamma}_i}{\mathcal{L}(\tilde{\gamma}_i)^{5}} \cdot \partial_x \varphi 
  + \dfrac{\tilde{\lambda}}{\mathcal{L}(\tilde{\gamma}_i)} \partial_x \tilde{\gamma}_i \cdot \partial_x \varphi \\
& \qquad \qquad \qquad \qquad + \mathcal{L}(\tilde{\gamma}_i) \partial_t \tilde{\gamma}_i \cdot \varphi 
  + \mathcal{L}(\tilde{\gamma}_i) \partial_t \tilde{\gamma}_i \cdot \Phi_1(\tilde{\gamma}_i,\varphi) \partial_x \tilde{\gamma}_i \Bigr] \, dx dt \\
&= \rho^{2} \int^{\tau}_0 \!\!\int^1_0 \Bigl[ \dfrac{\partial^2_x \gamma_i}{\mathcal{L}(\gamma_i)^{3}} \cdot \partial^2_x \varphi 
  - \dfrac{3}{2} \dfrac{|\partial^2_x \gamma_i|^{2} \partial_x \gamma_i}{\mathcal{L}(\gamma_i)^{5}} \cdot \partial_x \varphi 
  + \dfrac{\lambda}{\mathcal{L}(\gamma_i)} \partial_x \gamma_i \cdot \partial_x \varphi \\
& \qquad \qquad \qquad \qquad + \mathcal{L}(\gamma_i) \partial_t \gamma_i \cdot \varphi 
  + \mathcal{L}(\gamma_i) \partial_t \gamma_i \cdot \Phi_1(\gamma_i,\varphi) \partial_x \gamma_i \Bigr] \, dx dt =0   
\end{aligned}
\end{equation}
for $i=1, 2$. 
Taking $(\tilde{\gamma}_1 - \tilde{\gamma}_2)/\mathcal{L}(\tilde{\gamma}_1)$ and $(\tilde{\gamma}_1 - \tilde{\gamma}_2)/\mathcal{L}(\tilde{\gamma}_2)$ 
as $\varphi$ in \eqref{eq:5.2} with $i=1$ and $i=2$ respectively, and subtracting the latter from the former, we have 
\begin{align*}
0&= \int^{t(\tau)}_0 \!\!\!\! \int^1_0 
    \Bigl[ \dfrac{ \partial^2_x \tilde{\gamma}_1}{\mathcal{L}(\tilde{\gamma}_1)^{4}} - \dfrac{ \partial^2_x \tilde{\gamma}_2}{\mathcal{L}(\tilde{\gamma}_2)^{4}} \Bigr] 
    \cdot \partial^2_x (\tilde{\gamma}_1 - \tilde{\gamma}_2) \, dx dt \\
& \qquad - \dfrac{3}{2} \int^{t(\tau)}_0 \!\!\!\! \int^1_0 \Bigl[ \dfrac{|\partial^2_x \tilde{\gamma}_1|^{2} \partial_x \tilde{\gamma}_1}{\mathcal{L}(\tilde{\gamma}_1)^{6}}  
                       - \dfrac{|\partial^2_x \tilde{\gamma}_2|^{2} \partial_x \tilde{\gamma}_2}{\mathcal{L}(\tilde{\gamma}_2)^{6}} \Bigr] \cdot \partial_x (\tilde{\gamma}_1 - \tilde{\gamma}_2) \, dx dt \\
& \qquad + \tilde{\lambda} \int^{t(\tau)}_0 \!\!\!\! \int^1_0 
   \Bigl[ \dfrac{\partial_x \tilde{\gamma}_1}{\mathcal{L}(\tilde{\gamma}_1)^2} - \dfrac{\partial_x \tilde{\gamma}_2}{\mathcal{L}(\tilde{\gamma}_2)^2} \Bigr] 
   \cdot \partial_x (\tilde{\gamma}_1 - \tilde{\gamma}_2) \, dx dt \\
& \qquad + \int^{t(\tau)}_0 \!\!\!\! \int^1_0 \partial_t(\tilde{\gamma}_1 - \tilde{\gamma}_2)  \cdot (\tilde{\gamma}_1 - \tilde{\gamma}_2) \, dxdt \\
& \qquad + \int^{t(\tau)}_0 \!\!\!\! \int^1_0 
  \mathcal{L}(\tilde{\gamma}_1) \partial_t \tilde{\gamma}_1 \cdot \Phi_1(\tilde{\gamma}_1, \frac{\tilde{\gamma}_1 - \tilde{\gamma}_2}{\mathcal{L}(\tilde{\gamma}_1)}) \partial_x \tilde{\gamma}_1 \, dx dt \\
& \qquad - \int^{t(\tau)}_0 \!\!\!\! \int^1_0 
  \mathcal{L}(\tilde{\gamma}_2) \partial_t \tilde{\gamma}_2 \cdot \Phi_1(\tilde{\gamma}_2, \frac{\tilde{\gamma}_1 - \tilde{\gamma}_2}{\mathcal{L}(\tilde{\gamma}_2)}) \partial_x \tilde{\gamma}_2 \, dx dt \\ 
&:= I_1 + I_2 + I_3 + I_4 + I_5 + I_6.   
\end{align*}
Since 
\begin{equation}
\label{eq:5.3}
\mathcal{L}(\tilde{\gamma}_i) = \rho \mathcal{L}(\gamma_i), 
\end{equation}
we see that $|\partial_x \tilde{\gamma}_i| \equiv \mathcal{L}(\tilde{\gamma}_i)$. 
We also observe from \eqref{eq:1.3}, \eqref{eq:5.3} and Lemma \ref{theorem:4.3} that 
\begin{align}
& \tilde{C}_1 := \dfrac{2 \pi^2}{\mathcal{E}_2(\gamma_0)} \rho \le \mathcal{L}(\tilde{\gamma}_i) \le \dfrac{\mathcal{E}_2(\gamma_0)}{\lambda} \rho =: \tilde{C}_0, \label{eq:5.4} \\ 
& E_2(\tilde{\gamma}_i) = \dfrac{1}{2} \int^1_0 |\tilde{\kappa}_i|^2 \mathcal{L}(\tilde{\gamma}_i) \, dx 
  = \dfrac{1}{\rho} \dfrac{1}{2} \int^1_0 |\kappa_i|^2 \mathcal{L}(\gamma_i) \, dx 
  = \dfrac{1}{\rho} E_2(\gamma_i)
  \le \dfrac{1}{\rho} \mathcal{E}_2(\gamma_0). \label{eq:5.5}
\end{align} 
Moreover, it follows from the definition of $\tilde{\gamma}_i$ and Lemma \ref{theorem:4.7} that  
\begin{equation}
\label{eq:5.6}
\int^{\rho^4 \tau}_0 \!\!\! \int^1_0 \mathcal{L}(\tilde{\gamma}_i)|\partial_t \tilde{\gamma}_i|^2 \, dx dt 
 = \dfrac{1}{\rho} \int^{\tau}_0 \!\!\! \int^1_0 \mathcal{L}(\gamma_i)|\partial_t \gamma_i|^2 \, dx dt 
 \le \dfrac{2 \mathcal{E}_2(\gamma_0)}{\rho}. 
\end{equation}

First we reduce $I_1$ into  
\begin{align*}
I_1 &=  \int^{t(\tau)}_0 \!\!\!\! \int^1_0 \dfrac{1}{\mathcal{L}(\tilde{\gamma}_1)^{4}} |\partial^2_x (\tilde{\gamma}_1 - \tilde{\gamma}_2)|^2 \, dxdt  \\
    & \qquad +  \int^{t(\tau)}_0 \!\!\!\! \int^1_0 \Bigl[ \dfrac{1}{\mathcal{L}(\tilde{\gamma}_1)^{4}} - \dfrac{1}{\mathcal{L}(\tilde{\gamma}_2)^{4}} \Bigr] 
                                       \partial^2_x \tilde{\gamma}_2 \cdot \partial^2_x (\tilde{\gamma}_1 - \tilde{\gamma}_2) \, dxdt  
 =: I_{11} + I_{12}.                                            
\end{align*}
It follows from \eqref{eq:5.4} that  
\begin{equation*}
I_{11} \ge \dfrac{1}{\tilde{C}_0^{4}} \int^{t(\tau)}_0 \!\!\!\! \int^1_0 |\partial^2_x (\gamma_1-\gamma_2)|^{2} \, dx dt.    
\end{equation*}
Since 
\begin{align*}
\Bigl| \dfrac{1}{\mathcal{L}(\tilde{\gamma}_1)^{r}} - \dfrac{1}{\mathcal{L}(\tilde{\gamma}_2)^{r}} \Bigr| 
&\le \dfrac{r}{[ \theta \mathcal{L}(\tilde{\gamma}_1) + (1-\theta)\mathcal{L}(\tilde{\gamma}_2) ]^{r+1}} | \mathcal{L}(\tilde{\gamma}_1)-\mathcal{L}(\tilde{\gamma}_2)| \\
&\le \dfrac{r}{\tilde{C}_1^{r+1}} \| \partial_x(\tilde{\gamma}_1-\tilde{\gamma}_2)\|_{L^1(\mathcal{S}^1)} 
\end{align*}
for each $r \ge 1$, and it follows from \eqref{eq:5.5} that 
\begin{equation*}
\int^1_0 |\partial^2_x \tilde{\gamma}_i|^2 \, dx 
= 2 \mathcal{L}(\tilde{\gamma}_i)^2 E_2(\tilde{\gamma}_i) 
\le \dfrac{2}{\rho} \tilde{C}_0^2 \mathcal{E}_2(\gamma_0) 
= \dfrac{2 \lambda}{\rho^2} \tilde{C}_0^3 \quad \text{for} \quad i=1, 2,  
\end{equation*}
we deduce from H\"older's inequality and Lemma \ref{theorem:4.3} that  
\begin{align*}
|I_{12}| &\le \dfrac{4}{\tilde{C}_1^{5}} 
            \int^{t(\tau)}_0 \| \partial_x (\tilde{\gamma}_1 - \tilde{\gamma}_2) \|_{L^1(\mathcal{S}^1)} 
            \| \partial^2_x \tilde{\gamma}_2 \|_{L^2(\mathcal{S}^1)} \| \partial^2_x (\tilde{\gamma}_1- \tilde{\gamma}_2)\|_{L^2(\mathcal{S}^1)} \, dt \\
 &\le \dfrac{4 \sqrt{2 \lambda} \tilde{C}_0^{\frac{3}{2}}}{\tilde{C}_1^{5} \rho} 
      \int^{t(\tau)}_0 \| \partial_x (\tilde{\gamma}_1 - \tilde{\gamma}_2) \|_{L^1(\mathcal{S}^1)} \| \partial^2_x (\tilde{\gamma}_1- \tilde{\gamma}_2)\|_{L^2(\mathcal{S}^1)} \, dt.  
\end{align*}
Since it follows from Proposition \ref{theorem:2.1} that 
\begin{equation*}
\| \partial_x (\tilde{\gamma}_1 - \tilde{\gamma}_2) \|_{L^\infty(\mathcal{S}^1)} 
\le A \| \tilde{\gamma}_1 - \tilde{\gamma}_2 \|_{L^2(\mathcal{S}^1)}^{\frac{1}{4}} \| \partial^2_x(\tilde{\gamma}_1-\tilde{\gamma}_2)\|_{L^2(\mathcal{S}^1)}^{\frac{3}{4}}, 
\end{equation*}
we deduce from Young's inequality that  
\begin{equation}
\label{eq:5.7}
\begin{aligned}
|I_{12}| &\le \dfrac{4 A \sqrt{2 \lambda} \tilde{C}_0^{\frac{3}{2}}}{\tilde{C}_1^{5} \rho}
            \int^{t(\tau)}_0 \| \tilde{\gamma}_1 - \tilde{\gamma}_2 \|_{L^2(\mathcal{S}^1)}^{\frac{1}{4}} \| \partial^2_x (\tilde{\gamma}_1-\tilde{\gamma}_2)\|_{L^2(\mathcal{S}^1)}^{\frac{7}{4}} \, dt \\
         &\le \varepsilon \int^{t(\tau)}_0 \| \partial^2_x (\tilde{\gamma}_1-\tilde{\gamma}_2)\|_{L^2(\mathcal{S}^1)}^{2} \, dt 
              + C(\varepsilon) \int^{\rho^4 \tau}_0 \| \tilde{\gamma}_1 - \tilde{\gamma}_2 \|_{L^2(\mathcal{S}^1)}^{2} \, dt  
\end{aligned}
\end{equation}
for $\varepsilon>0$. 
Regarding $I_2$, we have 
\begin{align*}
I_2 &= \int^{t(\tau)}_0 \!\!\! \int^1_0 \Bigl[ \dfrac{1}{\mathcal{L}(\tilde{\gamma}_1)^{4}} - \dfrac{1}{\mathcal{L}(\tilde{\gamma}_2)^{4}} \Bigr] 
                                |\partial^2_x \tilde{\gamma}_1|^2 \partial_x \tilde{\gamma}_1 \cdot \partial_x(\tilde{\gamma}_1 - \tilde{\gamma}_2) \, dx dt \\
    & \qquad + \int^{t(\tau)}_0 \!\!\! \int^1_0 \Bigl[ |\partial^2_x \tilde{\gamma}_1|^2 - |\partial^2_x \tilde{\gamma}_2|^2 \Bigr] 
                                        \dfrac{\partial_x \tilde{\gamma}_1}{\mathcal{L}(\tilde{\gamma}_2)^{4}} \cdot \partial_x(\tilde{\gamma}_1 - \tilde{\gamma}_2) \, dx dt \\ 
    & \qquad + \int^{t(\tau)}_0 \!\!\! \int^1_0 \dfrac{|\partial^2_x \tilde{\gamma}_2|^2}{\mathcal{L}(\tilde{\gamma}_2)^{4}} |\partial_x (\tilde{\gamma}_1 - \tilde{\gamma}_2)|^2 \, dx dt
    =: I_{21} + I_{22} + I_{23}.  
\end{align*}
Similarly to \eqref{eq:5.7} we obtain 
\begin{align*}
|I_{21}| &\le \dfrac{4 \tilde{C}_0}{\tilde{C}_1^{5}} 
 \int^{t(\tau)}_0 \|\partial_x (\tilde{\gamma}_1-\tilde{\gamma}_2)\|_{L^1(\mathcal{S}^1)} \int^1_0 |\partial^2_x \tilde{\gamma}_1|^2 |\partial_x(\tilde{\gamma}_1 - \tilde{\gamma}_2)| \, dx dt \\
       &\le \dfrac{4 \tilde{C}_0}{\tilde{C}_1^{5}} 
            \int^{t(\tau)}_0 \|\partial_x (\tilde{\gamma}_1-\tilde{\gamma}_2)\|_{L^1(\mathcal{S}^1)} 
                                               \|\partial_x (\tilde{\gamma}_1-\tilde{\gamma}_2)\|_{L^\infty(\mathcal{S}^1)} \| \partial^2_x \tilde{\gamma}_1\|_{L^2(\mathcal{S}^1)}^2 \, dt \\
       &\le \dfrac{8 \lambda \tilde{C}_0^4}{\tilde{C}_1^{5}\rho^2} 
            \int^{t(\tau)}_0 \|\partial_x (\tilde{\gamma}_1-\tilde{\gamma}_2)\|_{L^1(\mathcal{S}^1)} \|\partial_x (\tilde{\gamma}_1-\tilde{\gamma}_2)\|_{L^\infty(\mathcal{S}^1)} \, dt \\ 
       &\le \dfrac{8 \lambda A \tilde{C}_0^4}{\tilde{C}_1^{5} \rho^2} 
            \int^{t(\tau)}_0 \| \tilde{\gamma}_1 - \tilde{\gamma}_2 \|_{L^2(\mathcal{S}^1)}^{\frac{1}{2}} \| \partial^2_x(\tilde{\gamma}_1-\tilde{\gamma}_2)\|_{L^2(\mathcal{S}^1)}^{\frac{3}{2}} \, dt \\
       &\le \varepsilon \int^{t(\tau)}_0 \| \partial^2_x (\tilde{\gamma}_1- \tilde{\gamma}_2)\|_{L^2(\mathcal{S}^1)}^{2} \, dt 
              + C(\varepsilon) \int^{\rho^4 \tau}_0 \| \tilde{\gamma}_1 - \tilde{\gamma}_2 \|_{L^2(\mathcal{S}^1)}^{2} \, dt
\end{align*}
for $\varepsilon>0$. 
Since 
\begin{equation*}
\bigl| |\partial^2_x \gamma_1|^2 - |\partial^2_x \gamma_2|^2 \bigr| \le ( |\partial^2_x \gamma_1| + |\partial^2_x \gamma_2| ) |\partial^2_x(\gamma_1 - \gamma_2)|, 
\end{equation*}
we see that  
\begin{align*}
|I_{22}| &\le \dfrac{\tilde{C}_0}{\tilde{C}_1^{4}} \int^{t(\tau)}_0 \!\!\! \int^1_0 
            \bigl[ |\partial^2_x \tilde{\gamma}_1| + |\partial^2_x \tilde{\gamma}_2| \bigr] |\partial^2_x (\tilde{\gamma}_1 - \tilde{\gamma}_2)| |\partial_x (\tilde{\gamma}_1 - \tilde{\gamma}_2)| \, dx dt \\
       &\le \dfrac{\tilde{C}_0}{\tilde{C}_1^{4}} \int^{t(\tau)}_0 \|\partial_x (\tilde{\gamma}_1-\tilde{\gamma}_2)\|_{L^\infty(\mathcal{S}^1)} 
                   \Bigl[ \sum^2_{i=1} \|\partial^2_x \tilde{\gamma}_i\|_{L^2(\mathcal{S}^1)} \Bigr] \|\partial^2_x (\tilde{\gamma}_1 - \tilde{\gamma}_2) \|_{L^2(\mathcal{S}^1)} \, dt \\
       &\le \dfrac{2 \sqrt{2 \lambda} \tilde{C}_0^{\frac{5}{2}}}{\tilde{C}_1^{5} \rho}  
              \int^{t(\tau)}_0 \|\partial_x (\tilde{\gamma}_1-\tilde{\gamma}_2)\|_{L^\infty(\mathcal{S}^1)} \|\partial^2_x (\tilde{\gamma}_1 - \tilde{\gamma}_2) \|_{L^2(\mathcal{S}^1)} \, dt \\
       &\le \varepsilon \int^{t(\tau)}_0 \| \partial^2_x (\tilde{\gamma}_1-\tilde{\gamma}_2)\|_{L^2(\mathcal{S}^1)}^{2} \, dt 
              + C(\varepsilon) \int^{\rho^4 \tau}_0 \| \tilde{\gamma}_1 - \tilde{\gamma}_2 \|_{L^2(\mathcal{S}^1)}^{2} \, dt
\end{align*}
for $\varepsilon>0$. 
Similarly we have 
\begin{equation*}
|I_{23}| \le \varepsilon \int^{t(\tau)}_0 \| \partial^2_x (\tilde{\gamma}_1-\tilde{\gamma}_2)\|_{L^2(\mathcal{S}^1)}^{2} \, dt 
              + C(\varepsilon) \int^{\rho^4 \tau}_0 \| \tilde{\gamma}_1 - \tilde{\gamma}_2 \|_{L^2(\mathcal{S}^1)}^{2} \, dt
\end{equation*}
for $\varepsilon>0$. 
On $I_3$, we have 
\begin{align*}
I_3 &= \tilde{\lambda} \int^{t(\tau)}_0 \!\!\! \int^1_0  
\Bigl[ \dfrac{1}{\mathcal{L}(\tilde{\gamma}_1)^2} - \dfrac{1}{\mathcal{L}(\tilde{\gamma}_2)^2} \Bigr] \partial_x \tilde{\gamma}_1 \cdot \partial_x(\tilde{\gamma}_1-\tilde{\gamma}_2) \, dx dt \\
    & \qquad 
    + \tilde{\lambda} \int^{t(\tau)}_0 \!\!\! \int^1_0 \dfrac{1}{\mathcal{L}(\tilde{\gamma}_2)^2} |\partial_x(\tilde{\gamma}_1-\tilde{\gamma}_2)|^2 \, dx dt 
    =: I_{31} + I_{32}. 
\end{align*}
Along the same line as above, we see that 
\begin{align*}
I_{31} &\le \dfrac{2 \tilde{\lambda} \tilde{C}_0}{\tilde{C}_1^3} \int^{t(\tau)}_0 \|\partial_x(\tilde{\gamma}_1-\tilde{\gamma}_2)\|_{L^1(\mathcal{S}^1)}^2 \, dt \\
       &\le \varepsilon \int^{t(\tau)}_0 \| \partial^2_x (\tilde{\gamma}_1-\tilde{\gamma}_2)\|_{L^2(\mathcal{S}^1)}^{2} \, dt 
              + C(\varepsilon) \int^{\rho^4 \tau}_0 \| \tilde{\gamma}_1 - \tilde{\gamma}_2 \|_{L^2(\mathcal{S}^1)}^{2} \, dt, \\
I_{32} &\ge \dfrac{\tilde{\lambda}}{\tilde{C}_1^2} \int^{t(\tau)}_0 \| \partial_x(\tilde{\gamma}_1-\tilde{\gamma}_2) \|^2_{L^2(\mathcal{S}^1)} \, dt. 
\end{align*}
On $I_4$ we deduce from $\tilde{\gamma}_1(\cdot,0)=\tilde{\gamma}_2(\cdot,0)=\rho \gamma_0(\cdot)$ that 
\begin{equation*}
I_4 = \int^{t(\tau)}_0 \!\!\!\! \int^1_0 \partial_t (\tilde{\gamma}_1-\tilde{\gamma}_2) \cdot (\tilde{\gamma}_1 - \tilde{\gamma}_2) \, dxdt     
    = \dfrac{1}{2} \int^1_0 |\tilde{\gamma}_1 - \tilde{\gamma}_2|^2 \, dx \Bigm|_{t=t(\tau)}. 
\end{equation*}
We turn to $I_5$. First we deduce from $|\partial_x \tilde{\gamma}_i| \equiv \mathcal{L}(\tilde{\gamma}_i)$ for $i=1, 2$ that  
\begin{align*}
\Phi_1(\tilde{\gamma}_1, \frac{\tilde{\gamma}_1 - \tilde{\gamma}_2}{\mathcal{L}(\tilde{\gamma}_1)})
 &= \dfrac{1}{\mathcal{L}(\tilde{\gamma}_1)^3} \Bigl[ x \int^1_0 \partial_x \tilde{\gamma}_1 \cdot \partial_x(\tilde{\gamma}_1 - \tilde{\gamma}_2) \, d\tilde{x}
     - \int^x_0 \partial_x \tilde{\gamma}_1 \cdot \partial_x(\tilde{\gamma}_1 - \tilde{\gamma}_2) \, d\tilde{x} \Bigr] \\
 &= \dfrac{1}{\mathcal{L}(\tilde{\gamma}_1)^3} \Bigl[ x \int^1_0 \bigl[ \mathcal{L}(\tilde{\gamma}_1)^2 - \partial_x \tilde{\gamma}_1 \cdot \partial_x \tilde{\gamma}_2 \bigr] \, d\tilde{x} \\
 & \qquad \qquad \qquad 
 - \int^x_0 \bigl[ \mathcal{L}(\tilde{\gamma}_1)^2 - \partial_x \tilde{\gamma}_1 \cdot \partial_x \tilde{\gamma}_2 \bigr] \, d\tilde{x} \Bigr] \\
 &= -\dfrac{1}{\mathcal{L}(\tilde{\gamma}_1)^3} \Bigl[ x \int^1_0 \partial_x \tilde{\gamma}_1 \cdot \partial_x \tilde{\gamma}_2 \, d\tilde{x} 
    - \int^x_0 \partial_x \tilde{\gamma}_1 \cdot \partial_x \tilde{\gamma}_2 \, d\tilde{x} \Bigr] \\
 &= - \dfrac{1}{\mathcal{L}(\tilde{\gamma}_1)} \Phi_1(\tilde{\gamma}_1, \tilde{\gamma}_2),  
\end{align*}
and then 
\begin{equation*}
I_5 = - \int^{t(\tau)}_0 \!\!\!\! \int^1_0 \partial_t \tilde{\gamma}_1 \cdot \Phi_1(\tilde{\gamma}_1, \tilde{\gamma}_2) \partial_x \tilde{\gamma}_1 \, dx dt.  
\end{equation*}
Since it follows from $|\partial_x \tilde{\gamma}_i| \equiv \mathcal{L}(\tilde{\gamma}_i)$ for $i=1, 2$ that 
\begin{equation*}
|\partial_x( \tilde{\gamma}_1 - \tilde{\gamma}_2)|^2 
 = \mathcal{L}(\tilde{\gamma}_1)^2 + \mathcal{L}(\tilde{\gamma}_2)^2 - 2 \partial_x \tilde{\gamma}_1 \cdot \partial_x \tilde{\gamma}_2, 
\end{equation*}
we see that 
\begin{align*}
\Phi_1(\tilde{\gamma}_1, \tilde{\gamma}_2) 
 &= \dfrac{1}{\mathcal{L}(\tilde{\gamma}_1)^2} \Bigl[ 
    -\dfrac{x}{2} \int^1_0 \bigl[ |\partial_x(\tilde{\gamma}_1 - \tilde{\gamma}_2)|^2 - \mathcal{L}(\tilde{\gamma}_1)^2 - \mathcal{L}(\tilde{\gamma}_2)^2 \bigr] \, d\tilde{x} \\
 & \qquad \qquad \quad
    + \dfrac{1}{2} \int^x_0 \bigl[ |\partial_x(\tilde{\gamma}_1 - \tilde{\gamma}_2)|^2 - \mathcal{L}(\tilde{\gamma}_1)^2 - \mathcal{L}(\tilde{\gamma}_2)^2 \bigr] \, d\tilde{x} \Bigr] \\
 &= - \dfrac{1}{2 \mathcal{L}(\gamma_1)^2} \Bigl[ 
      x \int^1_0 |\partial_x(\tilde{\gamma}_1 - \tilde{\gamma}_2)|^2 \, d\tilde{x} 
    - \int^x_0 |\partial_x(\tilde{\gamma}_1 - \tilde{\gamma}_2)|^2 \, d\tilde{x} \Bigr].  
\end{align*}
Thus we obtain 
\begin{equation*}
|I_5| \le \int^{t(\tau)}_0 \dfrac{1}{\mathcal{L}(\tilde{\gamma}_1)} \| \partial_x(\tilde{\gamma}_1 - \tilde{\gamma}_2) \|^2_{L^2(\mathcal{S}^1)} \int^1_0 |\partial_t \tilde{\gamma}_1| \, dx dt.  
\end{equation*}
Since 
\begin{equation*}
\| \partial_x(\tilde{\gamma}_1 - \tilde{\gamma}_2)\|_{L^2(\mathcal{S}^1)} 
 \le \| \tilde{\gamma}_1 - \tilde{\gamma}_2 \|_{L^2(\mathcal{S}^1)}^{\frac{1}{2}} \| \partial^2_x(\tilde{\gamma}_1 - \tilde{\gamma}_2)\|_{L^2(\mathcal{S}^1)}^{\frac{1}{2}}, 
\end{equation*}
we observe from H\"older's inequality, Cauchy's inequality and \eqref{eq:5.6} that 
\begin{align*}
|I_5| &\le \int^{t(\tau)}_0 \dfrac{1}{\mathcal{L}(\tilde{\gamma}_1)^{\frac{3}{2}}} \| \partial_x(\tilde{\gamma}_1 - \tilde{\gamma}_2) \|^2_{L^2(\mathcal{S}^1)} 
           \Bigl( \int^1_0 \mathcal{L}(\tilde{\gamma}_1) |\partial_t \tilde{\gamma}_1|^2 \, dx \Bigr)^{\frac{1}{2}} \, dt \\  
&\le \Bigl[ \int^{t(\tau)}_0 \dfrac{1}{\mathcal{L}(\tilde{\gamma}_1)^3} \| \partial_x(\tilde{\gamma}_1 - \tilde{\gamma}_2) \|^4_{L^2(\mathcal{S}^1)} \, dt \Bigr]^{\frac{1}{2}}
           \Bigl[\int^{t(\tau)}_0 \!\! \int^1_0 \mathcal{L}(\tilde{\gamma}_1) |\partial_t \tilde{\gamma}_1|^2 \, dx dt\Bigr]^{\frac{1}{2}} \\
&\le \dfrac{\sqrt{2 \mathcal{E}_2(\gamma_0)}}{\tilde{C}_1^3 \sqrt{\rho}} 
    \Bigl[ \int^{t(\tau)}_0 \| \tilde{\gamma}_1 - \tilde{\gamma}_2 \|^2_{L^2(\mathcal{S}^1)} \| \partial^2_x(\tilde{\gamma}_1 - \tilde{\gamma}_2) \|^2_{L^2(\mathcal{S}^1)} \, dt \Bigr]^{\frac{1}{2}} \\
&\le \dfrac{\sqrt{2} \mathcal{E}_2(\gamma_0)^{\frac{7}{2}}}{8 \pi^6 \rho^{\frac{7}{2}}} \| (\tilde{\gamma}_1 - \tilde{\gamma}_2)(\cdot, t(\tau)) \|_{L^2(\mathcal{S}^1)}
     \Bigl[ \int^{t(\tau)}_0 \| \partial^2_x(\tilde{\gamma}_1 - \tilde{\gamma}_2) \|^2_{L^2(\mathcal{S}^1)} \, dt \Bigr]^{\frac{1}{2}} \\
&\le \dfrac{1}{8} \| (\tilde{\gamma}_1 - \tilde{\gamma}_2)(\cdot, t(\tau)) \|^2_{L^2(\mathcal{S}^1)} 
     + \dfrac{\mathcal{E}_2(\gamma_0)^{7}}{16 \pi^{12} \rho^7} \int^{t(\tau)}_0 \| \partial^2_x(\tilde{\gamma}_1 - \tilde{\gamma}_2) \|^2_{L^2(\mathcal{S}^1)} \, dt. 
\end{align*}
Along the same line we also find 
\begin{equation*}
|I_6| \le \dfrac{1}{8} \| (\tilde{\gamma}_1 - \tilde{\gamma}_2)(\cdot, t(\tau)) \|^2_{L^2(\mathcal{S}^1)} 
     + \dfrac{\mathcal{E}_2(\gamma_0)^{7}}{16\pi^{12} \rho^7} \int^{t(\tau)}_0 \| \partial^2_x(\tilde{\gamma}_1 - \tilde{\gamma}_2) \|^2_{L^2(\mathcal{S}^1)} \, dt. 
\end{equation*}
Thus, taking $\varepsilon>0$ small enough, we obtain 
\begin{equation}
\label{eq:5.8}
\begin{aligned}
&\dfrac{1}{4} \|(\tilde{\gamma}_1 - \tilde{\gamma}_2)(\cdot,t(\tau))\|^2_{L^2(\mathcal{S}^1)} 
 + \dfrac{\tilde{\lambda}}{\tilde{C}_1^2} \int^{t(\tau)}_0 \| \partial_x(\tilde{\gamma}_1-\tilde{\gamma}_2) \|^2_{L^2(\mathcal{S}^1)} \, dt \\
& \quad  + \dfrac{1}{2 \tilde{C}_0^{4}} \int^{t(\tau)}_0 \|\partial^2_x (\tilde{\gamma}_1-\tilde{\gamma}_2)\|^{2}_{L^2(\mathcal{S}^1)} \, dt \\
&\le C \int^{\rho^4 \tau}_0 \| \tilde{\gamma}_1 - \tilde{\gamma}_2 \|_{L^2(\mathcal{S}^1)}^{2} \, dt 
  + \dfrac{\mathcal{E}_2(\gamma_0)^{7}}{8 \pi^{12} \rho^7} \int^{t(\tau)}_0 \| \partial^2_x(\tilde{\gamma}_1 - \tilde{\gamma}_2) \|^2_{L^2(\mathcal{S}^1)} \, dt.  
\end{aligned}
\end{equation}
Taking $\rho>0$ large enough such that 
\begin{equation*}
\rho^3 > \dfrac{\mathcal{E}_2(\gamma_0)^{11}}{4 \pi^{12} \lambda^4}, 
\end{equation*}
we see that 
\begin{equation*}
\dfrac{1}{2 \tilde{C}_0^{4}} - \dfrac{\mathcal{E}_2(\gamma_0)^{7}}{8 \pi^{12} \rho^7} =: C_2 > 0, 
\end{equation*}
and then we observe from \eqref{eq:5.1} that \eqref{eq:5.8} is reduced into  
\begin{equation*}
\|(\tilde{\gamma}_1 - \tilde{\gamma}_2)(\rho^4 \tau) \|^2_{L^2(\mathcal{S}^1)} \le C \int^{\rho^4 \tau}_0 \| \tilde{\gamma}_1 - \tilde{\gamma}_2 \|_{L^2(\mathcal{S}^1)}^{2} \, dt.    
\end{equation*}
Since $\tau \in [0, T]$ is arbitrary, this together with Gronwall's inequality implies that 
\begin{equation}
\label{eq:5.9}
\|(\tilde{\gamma}_1 - \tilde{\gamma}_2)(\cdot,t)\|_{L^2(\mathcal{S}^1)} = 0 \quad \text{for all} \quad t \in [0, \rho^{4} T]. 
\end{equation}
Plugging \eqref{eq:5.9} into the above argument, we have 
\begin{equation*}
C_2 \int^{\tau}_0 \|\partial^2_x (\tilde{\gamma}_1-\tilde{\gamma}_2) \|^{2}_{L^2(\mathcal{S}^1)} \, dt 
 + \dfrac{\tilde{\lambda}}{\tilde{C}_1^2} \int^{\tau}_0 \| \partial_x(\tilde{\gamma}_1-\tilde{\gamma}_2) \|^2_{L^2(\mathcal{S}^1)} \, dt
\le 0 
\end{equation*}
for $\tau \in [0, \rho^{4} T]$, and then, 
\begin{equation*}
\int^\tau_0 \| \partial_x(\tilde{\gamma}_1-\tilde{\gamma}_2) \|^2_{L^2(\mathcal{S}^1)} \, dt 
 = \int^\tau_0 \|\partial^2_x (\tilde{\gamma}_1-\tilde{\gamma}_2)\|^{2}_{L^2(\mathcal{S}^1)} \, dt = 0 
\end{equation*}
for $\tau \in [0, \rho^{4} T]$. 
Therefore Theorem \ref{theorem:5.1} follows. 
\end{proof}

%%%%%%%%%%%%%%%%%%%%%%%%%%%%%%%%%
\begin{lemma} \label{theorem:5.2}
Let $p=2$. Let $\gamma : \mathcal{S}^1 \times [0, \infty) \to \mathbb{R}^2$ be a global-in-time weak solution to problem \eqref{eq:P}. 
Then there exists a constant $C>0$ such that
\begin{equation}
\label{eq:5.10}
\partial^4_x \gamma \in L^{2}(0,T; L^2(\mathcal{S}^1)). 
\end{equation}
\end{lemma}
%%%%%%%%%%%%%%%%%%%%%%%%%%%%%%%%%
\begin{proof}
Let $p=2$. Fix $T>0$ arbitrarily. 
Thanks to Lemma \ref{theorem:4.5} we see that $\partial^3_x \gamma \in L^2(0,T; L^2(\mathcal{S}^1))$. 
Along the same argument as in the proof of Lemma \ref{theorem:4.5}, we observe from \eqref{eq:1.2} that 
\begin{equation}
\label{eq:5.11}
\begin{aligned}
\int^1_0 \Bigl[ \dfrac{\partial^2_x \gamma}{\mathcal{L}(\gamma)^{3}} \cdot \partial^2_x \varphi 
& - \dfrac{3}{2} \dfrac{|\partial^2_x \gamma|^{2} \partial_x \gamma}{\mathcal{L}(\gamma)^{5}} \cdot \partial_x \varphi 
  + \dfrac{\lambda}{\mathcal{L}(\gamma)} \partial_x \gamma \cdot \partial_x \varphi \\
& + \mathcal{L}(\gamma) \partial_t \gamma \cdot \varphi 
  + \mathcal{L}(\gamma) \partial_t \gamma \cdot \Phi_1(\gamma,\varphi) \partial_x \gamma \Bigr] \, dx=0   
\end{aligned}
\end{equation}
for a.e. $t \in (0, T)$ and all $\varphi \in H^{2}(\mathcal{S}^1; \mathbb{R}^2)$. 
Fix $\varphi \in H^2(\mathcal{S}^1; \mathbb{R}^2)$ arbitrarily. 
Integrating by parts, we have 
\begin{equation}
\label{eq:5.12}
\begin{aligned}
\int^1_0 \dfrac{\partial^3_x \gamma}{\mathcal{L}(\gamma)^{3}} \cdot \partial_x \varphi \, dx 
&= \dfrac{3}{2} \int^1_0 \dfrac{\partial_x(|\partial^2_x \gamma|^{2} \partial_x \gamma)}{\mathcal{L}(\gamma)^{5}} \cdot \varphi \, dx 
   - \int^1_0 \dfrac{\lambda}{\mathcal{L}(\gamma)} \partial^2_x \gamma \cdot \varphi \, dx \\
& \quad - \int^1_0 \mathcal{L}(\gamma) \partial_t \gamma \cdot \varphi \, dx 
   - \int^1_0 \mathcal{L}(\gamma) \partial_t \gamma \cdot \Phi_1(\gamma,\varphi) \partial_x \gamma \, dx \\
&=: I_1 + I_2 + I_3 + I_4.     
\end{aligned}
\end{equation}
First we have 
\begin{equation}
\label{eq:5.13}
\begin{aligned}
|I_2| &\le C \| \partial^2_x \gamma \|_{L^2(\mathcal{S}^1)} \| \varphi\|_{L^2(\mathcal{S}^1)}, \\
|I_3| &\le C \| \partial_t \gamma \|_{L^2(\mathcal{S}^1)} \| \varphi \|_{L^2(\mathcal{S}^1)}. 
\end{aligned}
\end{equation}
By integrating by part, we obtain 
\begin{equation*}
\Phi_1(\gamma, \varphi) 
 = \dfrac{1}{\mathcal{L}(\gamma)^2} \Bigl[ - x \int^1_0 \partial^2_x \gamma \cdot \varphi \, d\tilde{x} + \int^x_0 \partial^2_x \gamma \cdot \varphi \, d\tilde{x} 
                                           - (\partial_x \gamma \cdot \varphi)(x) + (\partial_x \gamma \cdot \varphi)(0) \Bigr]. 
\end{equation*}
Thanks to Lemma \ref{theorem:4.6} we have 
\begin{align*}
I_4 &= - \int^1_0 \partial^2_x \gamma \cdot \varphi \, d\tilde{x} \int^1_0 \dfrac{x}{\mathcal{L}(\gamma)} \partial_t \gamma \cdot \partial_x \gamma \, dx 
      + \int^1_0 \dfrac{1}{\mathcal{L}(\gamma)} \partial_t \gamma \cdot \partial_x \gamma \Bigl( \int^x_0 \partial^2_x \gamma \cdot \varphi \, d\tilde{x} \Bigr) dx \\
    & \qquad - \int^1_0 \dfrac{1}{\mathcal{L}(\gamma)} (\partial_t \gamma \cdot \partial_x \gamma)(\partial_x \gamma \cdot \varphi) \, dx,     
\end{align*}
and then 
\begin{equation*}
|I_4| \le (2 \| \partial^2_x \gamma \|_{L^2(\mathcal{S}^1)} \| \partial_t \gamma\|_{L^1(\mathcal{S}^1)} + C \| \partial_t \gamma\|_{L^2(\mathcal{S}^1)}) \| \varphi\|_{L^2(\mathcal{S}^1)}.
\end{equation*}
Since $|\partial_x \gamma| \equiv \mathcal{L}(\gamma)$, we see that 
\begin{align*}
\partial_x(|\partial^2_x \gamma|^{2} \partial_x \gamma) 
&= (\partial^2_x \gamma \cdot \partial^3_x \gamma) \partial_x \gamma + |\partial^2_x \gamma|^2 \partial^2_x \gamma \\
&= (\partial^2_x \gamma \cdot \partial^3_x \gamma) \partial_x \gamma - (\partial_x \gamma \cdot \partial^3_x \gamma) \partial^2_x \gamma, 
\end{align*}
and then 
\begin{equation}
\label{eq:5.14}
\begin{aligned}
|I_1| &\le C \int^1_0 |\partial^2_x \gamma| |\partial^3_x \gamma| |\varphi| \, dx 
       \le C \| \partial^2_x \gamma \|_{L^\infty(\mathcal{S}^1)} \| \partial^3_x \gamma\|_{L^2(\mathcal{S}^1)} \|\varphi\|_{L^2(\mathcal{S}^1)}. 
\end{aligned}
\end{equation}
Combining \eqref{eq:5.12} with \eqref{eq:5.13} and \eqref{eq:5.14}, we observe from \eqref{eq:3.7} that 
\begin{equation*}
\Bigl| \int^1_0 \dfrac{\partial^3_x \gamma}{\mathcal{L}(\gamma)^{3}} \cdot \partial_x \varphi \, dx  \Bigr| 
 \le C \bigl( 1 + \|\partial^2_x \gamma\|_{L^\infty(\mathcal{S}^1)} \| \partial^3_x \gamma\|_{L^2(\mathcal{S}^1)} + \|\partial_t \gamma \|_{L^2(\mathcal{S}^1)} \bigr)\|\varphi\|_{L^2(\mathcal{S}^1)} 
\end{equation*}
for a.e. $t \in (0,T)$. 
This implies 
\begin{equation}
\label{eq:5.15}
\|\partial^4_x \gamma\|_{L^2(\mathcal{S}^1)} 
 \le C \bigl( 1 + \|\partial^2_x \gamma\|_{L^\infty(\mathcal{S}^1)} \| \partial^3_x \gamma\|_{L^2(\mathcal{S}^1)} 
                + \|\partial_t \gamma \|_{L^2(\mathcal{S}^1)} \bigr) 
\end{equation}
for a.e. $t \in (0,T)$. 
Combining \eqref{eq:5.15} with Proposition \ref{theorem:2.1} and Lemma \ref{theorem:4.3}, we obtain 
\begin{align*}
\|\partial^4_x \gamma\|_{L^2(\mathcal{S}^1)} 
& \le C \bigl( 1 + \|\partial_x \gamma\|_{L^2(\mathcal{S}^1)}^{1/4} \| \partial^3_x \gamma\|_{L^2(\mathcal{S}^1)}^{7/4} 
                + \|\partial_t \gamma \|_{L^2(\mathcal{S}^1)} \bigr) \\
& \le C \bigl( 1 + \|\partial^2_x \gamma \|_{L^2(\mathcal{S}^1)}^{7/8} \| \partial^4_x \gamma\|_{L^2(\mathcal{S}^1)}^{7/8} + \|\partial_t \gamma \|_{L^2(\mathcal{S}^1)} \bigr) \\
&\le C \bigl( 1 + \| \partial^4_x \gamma\|_{L^2(\mathcal{S}^1)}^{7/8} + \|\partial_t \gamma \|_{L^2(\mathcal{S}^1)} \bigr)                
\end{align*}
for a.e. $t \in (0,T)$. 
This together with Young's inequality implies that 
\begin{equation}
\label{eq:5.16}
\|\partial^4_x \gamma\|_{L^2(\mathcal{S}^1)} 
 \le C \bigl( 1 + \|\partial_t \gamma \|_{L^2(\mathcal{S}^1)} \bigr) \quad \text{for a.e.} \quad t \in (0,T). 
\end{equation}
Integrating the both side with respect to $t$ on $(0,T)$, we observe from Lemmas~\ref{theorem:3.9} and~\ref{theorem:4.5} that \eqref{eq:5.10} holds. 
Therefore Lemme \ref{theorem:5.2} follows. 
\end{proof}

We are in a position to prove Theorem \ref{theorem:1.3}: 

\begin{proof}[Proof of Theorem \ref{theorem:1.3}]
Let $p=2$. 
Let $\gamma_0 \in \mathcal{AC}$, and fix $T>0$ arbitrarily. 
Then it follows from Theorem~\ref{theorem:1.2} that problem \eqref{eq:P} has a weak solution $\gamma : \mathcal{S}^1 \times [0, T] \to \mathbb{R}^2$. 
Thanks to Lemma \ref{theorem:5.1} we obtain the uniqueness of weak solutions to problem \eqref{eq:P}.  

We prove the energy structure \eqref{eq:1.6}. 
Fix $0 \le \tau_1 \le \tau_2 \le T$ arbitrarily. 
Then, from Theorem~\ref{theorem:1.2} and Lemma \ref{theorem:4.7} we find a weak solution $\tilde{\gamma} : \mathcal{S}^1 \times [\tau_1, \tau_2] \to \mathbb{R}^2$ of~\eqref{eq:P} 
starting from the `initial data' $\gamma(\cdot, \tau_1)$ such that 
\begin{equation}
\label{eq:5.17}
\mathcal{E}_2(\tilde{\gamma}(\tau_2)) - \mathcal{E}_2(\gamma(\tau_1)) \le -\dfrac{1}{2} \int^{\tau_2}_{\tau_1} \!\!\! \int^1_0 \mathcal{L}(\tilde{\gamma}) |\partial_t \tilde{\gamma}|^2 \, dx dt.  
\end{equation}
On the other hand, $\gamma |_{[\tau_1, \tau_2]}$ is also a weak solution of \eqref{eq:P} with `initial data' $\gamma(\cdot, \tau_1)$ in $\mathcal{S}^1 \times [\tau_1,\tau_2]$. 
It follows from Lemma~\ref{theorem:5.2} that $\tilde{\gamma}=\gamma$ in $H^1(\tau_1,\tau_2;L^2(\mathcal{S}^1)) \cup L^\infty(\tau_1,\tau_2; H^2(\mathcal{S}^1))$. 
Recalling that $\tilde{\gamma}(\cdot,\tau_2)=\gamma(\cdot,\tau_2)$, along the same line as in Lemma~\ref{theorem:4.2}, 
we see that $\tilde{\gamma}(\cdot, \tau_2)=\gamma(\cdot,\tau_2)$ in $H^2(\mathcal{S}^1)$. 
Thus \eqref{eq:1.6} follows from~\eqref{eq:5.17}. 

We prove the subconvergence of weak solution to an elastica. 
By Lemma~\ref{theorem:4.3} we see that 
\begin{equation}
\label{eq:5.18}
\| \partial_x \gamma\|_{L^\infty(\mathcal{S}^1)} = \mathcal{L}(\gamma) \le \dfrac{\mathcal{E}_2(\gamma_0)}{\lambda}, \quad 
\|\partial^2_x \gamma\|_{L^2(\mathcal{S}^1)}^2 = 2 \mathcal{L}(\gamma)^3 E_2(\gamma) \le \dfrac{2 \mathcal{E}_2(\gamma_0)^4}{\lambda^3}, 
\end{equation}
for a.e. $t \in (0, \infty)$. 
Let 
\begin{equation*}
p(t) := \int^1_0 \gamma(x,t) \, dx, \quad \tilde{\gamma}(x,t):= \gamma(x,t) -p(t).  
\end{equation*}
By Poincar\`e's inequalty we have 
\begin{equation}
\label{eq:5.19}
\| \tilde{\gamma} \|_{L^2(\mathcal{S}^1)} \le \dfrac{1}{2 \pi} \|\partial_x \gamma\|_{L^2(\mathcal{S}^1)} \le \dfrac{\mathcal{E}_2(\gamma_0)}{2 \pi \lambda}   
\end{equation}
for a.e. $t \in (0, \infty)$. 
It follows from \eqref{eq:4.23}, \eqref{eq:5.18} and Proposition \ref{theorem:2.1} that 
\begin{align*}
\| \partial^3_x \tilde{\gamma}\|_{L^2}=\| \partial^3_x \gamma\|_{L^2} 
 &\le C(1 + \| \partial^2_x \gamma\|_{L^\infty(\mathcal{S}^1)} \|\partial^2_x \gamma\|_{L^2(\mathcal{S}^1)} + \|\partial_t \gamma\|_{L^2(\mathcal{S}^1)}) \\
 &\le C(1 + \|\partial^3_x \gamma\|_{L^2(\mathcal{S}^1)}^{3/4} + \|\partial_t \gamma\|_{L^2(\mathcal{S}^1)}), 
\end{align*}
and then 
\begin{equation}
\label{eq:5.20}
\| \partial^3_x \tilde{\gamma}\|_{L^2} \le C(1 + \|\partial_t \gamma\|_{L^2(\mathcal{S}^1)}) 
\end{equation}
for a.e. $t \in (0, \infty)$. 
By \eqref{eq:1.6} we find a monotone divergent sequence $\{ t_k \}$ such that 
\begin{equation} 
\label{eq:5.21} 
\partial_t \gamma(\cdot, t_k) \to 0 \quad \text{in} \quad L^2(\mathcal{S}^1) \quad \text{as} \quad k \to \infty. 
\end{equation}
Combining \eqref{eq:5.21} with \eqref{eq:5.16}, \eqref{eq:5.18}, \eqref{eq:5.19} and \eqref{eq:5.20}, 
we see that $\{ \tilde{\gamma}(\cdot, t_k) \}_{k=1}^\infty$ is bounded in $H^4(\mathcal{S}^1)$. 
Thus we find a closed curve $\gamma_* \in H^4(\mathcal{S}^1;\mathbb{R}^2)$ such that 
\begin{equation}
\label{eq:5.22}
\tilde{\gamma}(\cdot, t_k) \rightharpoonup \gamma_*(\cdot) \quad \text{weakly in} \quad H^4(\mathcal{S}^1; \mathbb{R}^2)
\end{equation}
up to a subsequence. 
Then it is clear that $\gamma_* \in \mathcal{AC}$. 
Since \eqref{eq:5.11} is equivalent to 
\begin{equation}
\label{eq:5.23}
\begin{aligned}
\int^1_0 \Bigl[ \dfrac{\partial^2_x \tilde{\gamma}}{\mathcal{L}(\tilde{\gamma})^{3}} \cdot \partial^2_x \varphi 
& - \dfrac{3}{2} \dfrac{|\partial^2_x \tilde{\gamma}|^{2} \partial_x \tilde{\gamma}}{\mathcal{L}(\tilde{\gamma})^{5}} \cdot \partial_x \varphi 
  + \dfrac{\lambda}{\mathcal{L}(\tilde{\gamma})} \partial_x \tilde{\gamma} \cdot \partial_x \varphi \\
& + \mathcal{L}(\tilde{\gamma}) \partial_t \gamma \cdot \varphi 
  + \mathcal{L}(\tilde{\gamma}) \partial_t \gamma \cdot \Phi_1(\tilde{\gamma},\varphi) \partial_x \tilde{\gamma} \Bigr] \, dx=0,   
\end{aligned}
\end{equation}
taking a limit in~\eqref{eq:5.23} along the subsequence, we deduce from \eqref{eq:5.21} and \eqref{eq:5.22} that  
\begin{equation}
\label{eq:5.24}
\begin{aligned}
\int^1_0 \Bigl[ \dfrac{\partial^2_x \gamma_*}{\mathcal{L}(\gamma_*)^{3}} \cdot \partial^2_x \varphi 
 - \dfrac{3}{2} \dfrac{|\partial^2_x \gamma_*|^{2} \partial_x \gamma_*}{\mathcal{L}(\gamma_*)^{5}} \cdot \partial_x \varphi 
 + \dfrac{\lambda}{\mathcal{L}(\gamma_*)} \partial_x \gamma_* \cdot \partial_x \varphi \Bigr] \, dx=0   
\end{aligned}
\end{equation}
for all $\varphi \in H^2(\mathcal{S}^1; \mathbb{R}^2)$. 
Since $\gamma_* \in H^4(\mathcal{S}^1; \mathbb{R}^2)$, we observe from \eqref{eq:5.24} that 
\begin{equation*}
-\dfrac{\partial^4_x \gamma_*}{\mathcal{L}(\gamma_*)^{4}}  
 - \dfrac{3}{2} \dfrac{\partial_x(|\partial^2_x \gamma_*|^{2} \partial_x \gamma_*)}{\mathcal{L}(\gamma_*)^{6}} 
 + \dfrac{\lambda}{\mathcal{L}(\gamma_*)^2} \partial^2_x \gamma_*=0     
\end{equation*}
for a.e. $x \in \mathcal{S}^1$. Since the equation is equivalent to 
\begin{equation*}
-\partial^2_s \kappa_* - \dfrac{1}{2} \kappa_*^3 + \lambda \kappa_*=0, 
\end{equation*}
where $\kappa_*$ denotes the curvature of $\gamma_*$, we see that $\gamma_*$ is an elastica. 
Therefore Theorem~\ref{theorem:1.3} follows. 
\end{proof}

\noindent
{\bf Acknowledgements.}
This work was initiated during the first author's visit at University of Wollongong (before the Corona pandemic), partially funded via the EIS Near Miss grant scheme. 
The first author is very grateful to second author for his warm hospitality and the inspiring working atmosphere. 
The first author was supported in part by JSPS KAKENHI Grant Numbers JP19H05599, JP20KK0057 and JP21H00990.

%%%%%%%%%%%%%%%%%%%%%%%%%%%%%%%%%%%%%%%%%%%%%%%%%%%%%%%%%%%%%%%%%%
%%%%%%%%%%%%%%%%%%%%%%%%%%%%%%%%%%%%%%%%%%%%%%%%%%%%%%%%%%%%%%%%%%
%%%%%%%%%%%%%%%%%%%%%%%%%%%%%%%%%%%%%%%%%%%%%%%%%%%%%%%%%%%%%%%%%%


\begin{thebibliography}{99}

\bibitem{AM} E. Acerbi and D. Mucci, {\it Curvature-dependent energies: the elastic case}, 
Nonlinear Anal. {\bf 153} (2017), 7--34.

\bibitem{Adams} R. A. Adams and J. J. F. Fournier, Sobolev spaces, second ed., Pure and Applied Mathematics (Amsterdam) {\bf 140}, Elsevier/Academic Press, Amsterdam, 2003. 

\bibitem{AGS} L. Ambrosio, N. Gigli and G. Savar\`e, Gradient Flows, Birkh\"auser, Basel (2008). 

\bibitem{Badal} R. Badal, {\it Curve-shortening flow of open, elastic curves in $\mathbb{R}^2$ with repelling endpoints: A minimizing movement approach}, 
arXiv: 1902.08079v1, 2019. 

\bibitem{BVH} S. Blatt, N. Vorderobermeier and C. Hopper, 
{\it A minimizing movement scheme for the $p$-elastic energy of curves}, arXiv:2101.10101. 

\bibitem{BGH} G. Buttazzo, M. Giaquinta and S. Hildebrandt, {\it One-dimensional variational problems. An introduction}, 
Oxford Lecture Series in Mathematics and its Applications, {\bf 15} 
The Clarendon Press, Oxford University Press, New York, 1998. 

\bibitem{DD} A. Dall'Acqua and K. Deckelnick, {\it An obstacle problem for elastic graphs}, SIAM J. Math. Anal. {\bf 50} (2018), no. 1, 119--137. 

\bibitem{DDG} A. Dall'Acqua, K. Deckelnick and H.-C. Grunau, {\it Classical solutions to the Dirichlet problem for Willmore surfaces of revolution}, Adv. Calc. Var. {\bf 1} (2008), no. 4, 379--397. 

\bibitem{DLLPS} A. Dall'Acqua, T. Laux, C.C. Lin, P. Pozzi and A. Spener, 
{\it The elastic flow of curves on the sphere}, Geom. Flows {\bf 3} (2018), 1--13.

\bibitem{DLP_2014} A. Dall'Acqua, C.-C. Lin and P. Pozzi, 
{\it Evolution of open elastic curves in $\mathbb{R}^n$ subject to fixed length and natural boundary conditions}, Analysis (Berlin) {\bf 34} (2014), no. 2, 209--222. 

\bibitem{DLP_2017} A. Dall'Acqua, C.-C. Lin and P. Pozzi, 
{\it A gradient flow for open elastic curves with fixed length and clamped ends}, Ann. Sc. Norm. Super. Pisa Cl. Sci. (5) {\bf 17} (2017), no. 3, 1031--1066. 

\bibitem{DP_2014}  A. Dall'Acqua and P. Pozzi, 
{\it A Willmore-Helfrich $L^2$-flow of curves with natural boundary conditions}, Comm. Anal. Geom. {\bf 22} (2014), no. 4, 617--669. 

\bibitem{DPS} A. Dall'Acqua, P. Pozzi and A. Spener, 
{\it The \L{}ojasiewicz-Simon gradient inequality for open elastic curves}, J. Differential Equations {\bf 261} (2016), no. 3, 2168--2209. 

\bibitem{DFLM} G. Dal Maso, I. Fonseca, G. Leoni and M. Morini, 
{\it A higher order model for image restoration: the one-dimensional case}, SIAM J. Math. Anal. {\bf 40}, no. 6 (2009), 2351--2391.

\bibitem{DKS} G. Dziuk, E. Kuwert and R. Sch\"{a}tzle, 
{\it Evolution of elastic curves in $\mathbb{R}^n$: existence and computation} SIAM J. Math. Anal. {\bf 33} (2002), no. 5, 1228--1245. 

\bibitem{FKN} V. Ferone, B. Kawohl and C. Nitsch, 
{\it Generalized elastica problems under area constraint}, Math. Res. Lett. {\bf 25}, (2018), no. 2, 521--533. 

\bibitem{FFLM_2012} I. Fonseca, N. Fusco, G. Leoni and M. Morini, {\it Motion of elastic thin films by anisotropic surface diffusion with curvature regularization}, 
Arch. Ration. Mech. Anal. {\bf 205} (2012), no. 2, 425--466. 

\bibitem{K} N. Koiso, {\it On the motion of a curve towards elastica}, 
In Actes de la Table Ronde de G\'eom\'etrie Diff\'erentielle (Luminy, 1992), vol. 1 of S\'emin. Congr. Soc. Math. France, Paris, 1996, 403--436. 

\bibitem{LS_1984} J. Langer and D. A. Singer, {\it Knotted elastic curves in $\mathbb{R}^3$}, 
J. London Math. Soc. (2) {\bf 30} (1984), no. 3, 512--520.

\bibitem{LS_1985} J. Langer and D. A. Singer, 
{\it Curve straightening and a minimax argument for closed elastic curves}, Topology {\bf 24} (1985), no. 1, 75--88. 

\bibitem{L} C.-C. Lin, {\it $L^2$-flow of elastic curves with clamped boundary conditions}, 
J. Differential Equations {\it 252} (2012), no. 12, 6414--6428. 

\bibitem{LLS} C.-C. Lin, Y.-K. Lue and H. R. Schwetlick, 
{\it The second-order $L^2$-flow of inextensible elastic curves with hinged ends in the plane}, 
J. Elasticity {\bf 119} (2015), no. 1-2, 263--291.

\bibitem{MM} C. Mantegazza and L. Martinazzi, 
{\it A note on quasilinear parabolic equations on manifolds}, 
Ann. Sc. Norm. Super. Pisa Cl. Sci. (5) {\bf 11} (2012), no. 4, 857--874.

\bibitem{MPP} C. Mantegazza, A. Pluda and M. Pozzetta, 
{\it A survey of the elastic flow of curves and networks}, Milan J. Math. (2021), to appear. 

\bibitem{MP} C. Mantegazza and M. Pozzetta, 
{\it The {\L}ojasiewicz--Simon inequality for the elastic flow}, Calc. Var. {\bf 60} 56 (2021). 

\bibitem{NO_2014} M. Novaga and S. Okabe, 
{\it Curve shortening-straightening flow for non-closed planar curves with infinite length}, 
J. Differential Equations {\bf 256} (2014), no. 3, 1093--1132.

\bibitem{NO} M. Novaga and S. Okabe, {\it Regularity of the obstacle problem for the parabolic biharmonic equation}, 
Math. Ann. {\bf 363} (2015), no. 3-4, 1147--1186. 

\bibitem{NO_2017} M. Novaga and S. Okabe, 
{\it Convergence to equilibrium of gradient flows defined on planar curves}, 
J. Reine Angew. Math. {\bf 733} (2017), no. 3, 87--119. 

\bibitem{NP} M. Novaga and P. Pozzi, {\it A second order gradient flow of $p$-elastic planar networks}, 
SIAM J. Math. Anal. {\bf 52} (2020), no. 1, 682--708. 

\bibitem{Oe_2014} D. Oelz, {\it Convergence of the penalty method applied to a constrained curve straightening flow}, 
Commun. Math. Sci. {\bf 12} (2014), no. 4, 601--621. 

\bibitem{Oe_2011} D. B. \"Oelz, {\it On the curve straightening flow of inextensible, open, planar curves}, 
SeMA J., {\bf 54} (2011), 5--24. 

\bibitem{O_2007} S. Okabe, {\it The motion of elastic planar closed curves under the area-preserving condition}, 
Indiana Univ. Math. J. {\bf 56} (2007), no. 4, 1871--1912.

\bibitem{O_2008} S. Okabe, {\it The dynamics of elastic closed curves under uniform high pressure}, 
Calc. Var. {\bf 33} (2008), no. 4, 493--521.

\bibitem{OPW} S. Okabe, P. Pozzi and G. Wheeler, {\it A gradient flow for the $p$-elastic energy defined on closed planar curves}, 
Math. Ann. {\bf 378} (2020), no. 1-2, 777--828. 

\bibitem{OY} S. Okabe and K. Yoshizawa, {\it The obstacle problem for a fourth order semilinear parabolic equation}, 
Nonlinear Anal. {\bf 198} (2020), 111902, 23 pp. 

\bibitem{P} A. Polden, {\it Curves and surfaces of least total curvature and fourth-order flows}, 
PhD Thesis, Universit\"at T\"ubingen (1996). 

\bibitem{SW} N. Shioji and K. Watanabe, {\it Total $p$-powered curvature of closed curves and flat-core closed $p$-curves in $S^2(G)$}, 
Comm. Anal. Geom. {\bf 28} (2020), no. 6, 1451--1487. 

\bibitem{S} A. Spener, {\it Short time existence for the elastic flow of clamped curves}, Math. Nachr. {\bf 290} (2017), no. 13, 2052--2077.

\bibitem{W} K. Watanabe, {\it Planar $p$-p-elastic curves and related generalized complete elliptic integrals}, 
Kodai Math. J. {\bf 37} (2014), no. 2, 453--474. 

\bibitem{Wen_1993} Y. Wen, {\it $L^2$ flow of curve straightening in the plane}. Duke Math. J. {\bf 70} (1993), no. 3, 683--698.

\bibitem{Wen_1995} Y. Wen, {\it Curve straightening flow deforms closed plane curves with nonzero rotation number to circles}, 
J. Differential Equations {\bf 120} (1995), no. 1, 89--107. 

\bibitem{Wheeler_2015} G. Wheeler, {\it Global analysis of the generalised Helfrich flow of closed curves immersed in $\mathbb{R}^n$}, 
Trans. Amer. Math. Soc. {\bf 367} (2015), no. 4, 2263--2300.

\end{thebibliography}
\end{document}